\documentclass[10pt]{svjour3}                     
\smartqed  
\usepackage{amsmath}
\usepackage{amssymb}
\usepackage{graphicx}
\usepackage{color}
\usepackage{ifpdf}
\usepackage{url}
\usepackage{hyperref}
\usepackage{algorithm}
\usepackage{abstract}
\usepackage[usenames,dvipsnames]{xcolor}
\usepackage{xfrac}
\usepackage{comment}
\usepackage[ruled,vlined, linesnumbered, algo2e]{algorithm2e}
\usepackage{algorithmicx}
\usepackage{algpseudocode}
\usepackage{amssymb,makeidx,mathrsfs}
\usepackage{fixltx2e}
\usepackage{multicol} 
\usepackage{booktabs} 
\newtheorem{assumption}{Assumption}

\makeatletter
\renewcommand\@biblabel[1]{#1.}
\makeatother

\newcommand{\Id}{\mathbb{I}}

\newcommand{\R}{\mathbb{R}}

\newcommand{\Rext}{\R\cup\{+\infty\}}

\newcommand{\abs}[1]{\left\vert#1\right\vert}
\newcommand{\absc}[1]{\big\vert#1\big\vert}
\newcommand{\set}[1]{\left\{#1\right\}}
\newcommand{\sets}[1]{\{#1\}}
\newcommand{\norm}[1]{\left\Vert#1\right\Vert}

\newcommand{\norms}[1]{\Vert#1\Vert}

\newcommand{\Eproof}{\hfill $\square$}
\newcommand{\prox}{\mathrm{prox}}
\newcommand{\kprox}[2]{\mathrm{prox}_{#1}\big(#2\big)}
\newcommand{\proj}{\mathrm{proj}}
\newcommand{\kproj}[2]{\mathrm{proj}_{#1}\left(#2\right)}
\newcommand{\relint}[1]{\mathrm{ri}\left(#1\right)}
\newcommand{\argmin}{\mathrm{arg}\!\displaystyle\min}

\newcommand{\dom}[1]{\mathrm{dom}(#1)}
\newcommand{\zero}[1]{\boldsymbol{0}}

\newcommand{\xb}{x}
\newcommand{\xopt}{x^{\star}}
\newcommand{\yopt}{y^{\star}}

\newcommand{\yb}{y}

\newcommand{\zb}{z}

\newcommand{\ub}{u}  
\newcommand{\vb}{v}

\renewcommand{\sb}{s}
\newcommand{\rb}{r}
\newcommand{\qb}{q}
\newcommand{\ab}{a}
\newcommand{\bb}{b}
\newcommand{\cb}{c}

\newcommand{\Ab}{A}
\newcommand{\Bb}{B}

\newcommand{\Ac}{\mathcal{A}}
\newcommand{\Db}{D}
\newcommand{\Qb}{Q}
\newcommand{\Qc}{\mathcal{Q}}
\newcommand{\Bc}{\mathcal{B}}

\newcommand{\Xb}{X}
\newcommand{\Yb}{Y}
\newcommand{\Zb}{Z}

\newcommand{\Yc}{\mathcal{Y}}
\newcommand{\Zc}{\mathcal{Z}}
\newcommand{\Sc}{\mathcal{S}}

\newcommand{\Lc}{\mathcal{L}}

\newcommand{\Cc}{\mathcal{C}}
\newcommand{\Nc}{\mathcal{N}}

\newcommand{\Kc}{\mathcal{K}}

\newcommand{\lbd}{\lambda} 
\newcommand{\Lbd}{\Lambda} 
\newcommand{\lbdopt}{\lambda^{\star}}

\newcommand{\iprods}[1]{\langle #1\rangle}

\newcommand{\kdist}[2]{\mathrm{dist}_{#1}\big(#2\big)}

\newcommand{\BigO}[1]{\mathcal{O}\left(#1\right)}

\newcommand{\beforesubsec}{\vspace{-3.75ex}}
\newcommand{\aftersubsec}{\vspace{-2.2ex}}
\newcommand{\beforesec}{\vspace{-3.5ex}}
\newcommand{\aftersec}{\vspace{-2.25ex}}
\newcommand{\beforesubsubsec}{\vspace{-1.5ex}}
\newcommand{\aftersubsubsec}{\vspace{-2.5ex}}

\begin{document}

\title{Proximal Alternating Penalty Algorithms for Nonsmooth Constrained Convex Optimization
}

\titlerunning{Proximal Alternating Penalty Algorithms for Constrained Convex Optimization}        

\author{Quoc Tran-Dinh
}


\institute{Quoc Tran-Dinh \at
		Department of Statistics and Operations Research,  University of North Carolina at Chapel Hill\\
		333 Hanes Hall, CB\# 3260, UNC Chapel Hill, NC 27599-3260, USA\\ 
		Email: \url{quoctd@email.unc.edu}
}

\date{Received: date / Accepted: date}

\maketitle

\begin{abstract}
We develop two new proximal alternating penalty algorithms  to solve a wide range class of constrained convex optimization problems.
Our approach mainly relies on a novel combination of the classical quadratic penalty, alternating minimization, Nesterov's acceleration, and adaptive strategy for parameters.
The first algorithm is designed to solve generic and possibly nonsmooth constrained convex problems without requiring any Lipschitz gradient continuity or strong convexity, while achieving the best-known $\BigO{\frac{1}{k}}$-convergence rate  in a non-ergodic sense, where $k$ is the iteration counter.
The second algorithm is also designed to solve non-strongly convex, but semi-strongly convex problems. 
This algorithm can achieve the best-known $\BigO{\frac{1}{k^2}}$-convergence rate on the primal constrained problem.
Such a rate is obtained in two cases: (i)~averaging only on the iterate sequence of the strongly convex term, or (ii) using two proximal operators of this term without averaging.
In both algorithms, we allow one to linearize the second subproblem to use the proximal operator of the corresponding objective term.
Then, we customize our methods to solve different convex problems, and lead to new variants.
As a byproduct, these algorithms preserve the same convergence guarantees as in our main algorithms.
We verify our theoretical development via different numerical examples and compare our methods with some existing state-of-the-art algorithms.
\end{abstract}


\beforesec
\section{Introduction}\label{sec:intro}
\aftersec
\noindent\textbf{Problem statement:}
We develop novel numerical methods to solve the following generic and possibly nonsmooth constrained convex optimization problem:
\begin{equation}\label{eq:constr_cvx}
F^{\star} := 
\left\{\begin{array}{ll}
\displaystyle\min_{\zb := (\xb, \yb)\in\R^{p}} & \Big\{ F(\zb) := f(\xb) + g(\yb) \Big\} \vspace{0.5ex}\\
\mathrm{s.t.} & \Ab\xb + \Bb\yb - \cb \in \Kc,
 \end{array}\right.
\end{equation}
where $f : \R^{p_1}\to\Rext$ and $g : \R^{p_2}\to\Rext$ are two proper, closed, and convex functions; $p := p_1 + p_2$;
$\Ab\in\R^{n\times p_1}$, $\Bb\in\R^{n\times p_2}$, and $\cb\in\R^n$ are given; and $\Kc\subseteq\R^n$ is a nonempty, closed, and convex subset.

Problem \eqref{eq:constr_cvx}, on the one hand, covers a wide range class of classical constrained convex optimization problems in practice including conic programming (e.g., linear, convex quadratic, second-order cone, and semidefinite programming), convex optimization over graphs and networks, geometric programming, monotropic convex programming, and model predictive controls (MPC) \cite{BenTal2001,Boyd2004,Nocedal2006}.
On the other hand, it can be used as a unified template to formulate many recent convex optimization models arising in signal and image processing, machine learning, and statistics ranging from unconstrained to constrained settings, see, e.g., \cite{Boyd2011,Nesterov2004,sra2012optimization}.
In the latter case, the underlying convex problems obtained from these applications are often challenging to solve due to their high-dimensionality and nonsmoothness.
Therefore, classical optimization methods such as sequential quadratic programming, and interior-point methods are no longer efficient to solve them \cite{Nocedal2006}.
This fundamental challenge has opened a door for the use of first-order methods \cite{Beck2009,Chambolle2011,Nesterov2004}.
Various first-order methods have been proposed to solve large-scale instances of \eqref{eq:constr_cvx} including [proximal] gradient and fast gradient, primal-dual, splitting, conditional gradient, mirror descent, coordinate descent, and stochastic gradient-type methods, see, e.g, \cite{Bauschke2011,Chambolle2011,du2017selective,Jaggi2013,johnson2013accelerating,kiwiel1999proximal,Nesterov2004,Nesterov2012}.
While discussing them all is out of scope of this paper, we focus on some strategies such as penalty, alternating direction, and primal-dual methods which most relate to our work in this paper.

\noindent\textbf{Our approach and related work:}
The approach in this paper relies on a novel combination of the quadratic penalty \cite{Fletcher1987,Nocedal2006}, alternating miminization \cite{Bauschke2011,kiwiel1999proximal,Tseng1991a}, adaptive strategy for parameters \cite{TranDinh2015b}, and Nesterov's accelerated methods \cite{auslender2006interior,Nesterov2004,tseng2008accelerated}.
The quadratic penalty method is a classical optimization framework to handle constrained problems, and can be found in classical text books, e.g., \cite{Fletcher1987,Nocedal2006}. 
It is often used in nonlinear optimization, and has recently been studied in first-order convex optimization methods, see \cite{Lan2013,necoara2015complexity}.
This method is often inefficient if it stands alone. 
In this paper, we combine it with other ideas and show that it is indeed useful. 
Our second idea is to use the alternating strategy dated back from the work of J. von Neumann \cite{Boyd2011}, but has recently become extremely popular, see, e.g., \cite{Boyd2011,Esser2010,Goldstein2012,He2012,shefi2016rate,Ouyang2014}. 
We exploit this old technique to split the coupling constraint $\Ab\xb + \Bb\yb -\cb \in\Kc$ and the proximal operator of $f + g$ into each individual one on $\xb$ and $\yb$.
Note that the alternating idea has been widely used in many papers including \cite{du2017selective,Goldfarb2012,kiwiel1999proximal} but often for unconstrained settings.
However, the key idea in our approach is perhaps Nesterov's acceleration scheme \cite{Nesterov2004} and the adaptive strategy for parameters in \cite{TranDinh2015b} that allow us to accelerate the convergence rate of our methods as well as to automatically update the penalty and other parameters without tuning.

In the context of primal-dual frameworks, our algorithms work on the primal problem \eqref{eq:constr_cvx} and also have convergence guarantees on this problem in terms of objective residual and feasibility violation.
Hence, they are different from primal-dual methods such as Chambolle-Pock's scheme \cite{Chambolle2011,chambolle2016ergodic}, alternating minimization (AMA) \cite{Goldstein2012,Tseng1991a}, and alternating direction methods of multipliers (ADMM) \cite{Eckstein1992,Chen1994,Boyd2011,Goldstein2012,Ouyang2014}. 
Note that primal-dual algorithms, AMA, and ADMM are classical methods and their convergence guarantees were proved in many early works, e.g., \cite{Eckstein1992,Chen1994,Tseng1991a}. 
Nevertheless, their convergence rate and iteration-complexity have only recently been studied under different assumptions including strong convexity, Lipschitz gradient continuity, and error bound-type conditions, see, e.g., \cite{Chambolle2011,Davis2014,Davis2014a,Davis2014b,Goldstein2012,He2012,shefi2016rate} and the references quoted therein.

Existing state-of-the-art primal-dual methods often achieve the best-known $\BigO{\frac{1}{k}}$-rate without  strong convexity and Lipschitz gradient continuity, where $k$ is the iteration counter.
However, such a rate is often obtained via an ergodic sense or a weighted averaging sequence \cite{Chambolle2011,Davis2014,Davis2014a,Davis2014b,He2012,shefi2016rate,Ouyang2014}.  
Under a stronger condition such as either strong convexity or Lipschitz gradient continuity, one can achieve the best-known $\BigO{\frac{1}{k^2}}$-convergence rate as shown in, e.g., \cite{Chambolle2011,Davis2014,Davis2014a,Davis2014b,Ouyang2014}.\footnote{A recent work in \cite{attouch2017rate} showed an ${o}\left(\frac{1}{k}\right)$ or ${o}\left(\frac{1}{k^2}\right)$ rate depending on problem structures.}  
A recent work by Xu \cite{xu2017accelerated} showed that ADMM methods can achieve the $\BigO{\frac{1}{k^2}}$ convergence rate requiring only the strong convexity on one objective term (either $f$ or $g$).
Such a rate is achieved via weighted averaging sequences. 
This is fundamentally different from the fast ADMM variant studied in \cite{Goldstein2012}.
Note that the $\BigO{\frac{1}{k^2}}$ rate is also attained in AMA methods \cite{Goldstein2012} under the same assumption. 
Nevertheless, this rate is on the dual problem, and can be viewed as FISTA \cite{Beck2009} applying to the dual problem of \eqref{eq:constr_cvx}. 
To the best of our knowledge, the  $\BigO{\frac{1}{k^2}}$ on the primal problem  has not been shown yet.
Recently, we proposed two algorithms in \cite{TranDinh2015b} to solve \eqref{eq:constr_cvx} that achieve $\BigO{\frac{1}{k}}$ convergence rate without any strong convexity or Lipschitz gradient continuity.
Moreover, our guarantee for the first algorithm, Algorithm \ref{alg:A1}, is given in a non-ergodic sequence, i.e., in the last iterate.

Despite of using the quadratic penalty method as in \cite{Lan2013,necoara2015complexity} to handle the constraints, our approach in this paper is fundamentally different from  \cite{Lan2013}, where we apply the alternating scheme to decouple the joint variable $\zb = (\xb,\yb)$ and treat them alternatively between $\xb$ and $\yb$.
We also exploit the homotopy strategy in \cite{TranDinh2015b} to automatically update the penalty parameter instead of fixing or tuning as in \cite{Lan2013,necoara2015complexity}.
In terms of theoretical guarantee, \cite{Lan2013} characterized the iteration-complexity by appropriately choosing a set of parameters depending on the desired accuracy and the feasible set diameters, while \cite{necoara2015complexity} assumed that the subproblem could be solved by Nesterov's schemes up to a certain accuracy.
Our guarantee does not use any of these techniques, which avoids their drawbacks.
Our methods are also different from AMA or ADMM where we do not require Lagrange multipliers, but rather stay in the primal space of \eqref{eq:constr_cvx}.
{In fact, our idea is closely related to \cite{Goldfarb2012,kiwiel1999proximal},  but is still essentially different. 
The methods in \cite{kiwiel1999proximal} are originated from the bundle method and also do not require the smoothness of the objective functions.
The algorithms in \cite{Goldfarb2012} are  alternating linearization-type methods which entail the smoothness of the objective functions.}
We handle the constrained problem \eqref{eq:constr_cvx} directly and  update the penalty parameter. 
We also do not assume the smoothness of $f$ and $g$. 
Our algorithm is also different from the dual smoothing methods in \cite{Necoara2009} and \cite{Becker2011a}, where they simply added a proximity function to the primal objective function to obtain a Lipschitz gradient continuous dual function, and applied Nesterov's accelerated schemes. 
These methods accelerate on the dual space.

In terms of structure assumption, our first algorithm achieves the same $\BigO{\tfrac{1}{k}}$-rate as in \cite{Chambolle2011,Davis2014,Davis2014a,Davis2014b,He2012,shefi2016rate,Ouyang2014} without any assumption except for the existence of a saddle point. Moreover, the rate of convergence is on the last iterate, which is important for sparse and low-rank optimization (since averaging essentially destroys the sparsity or low-rankness of the approximate solutions).
Under a semi-strong convexity, i.e., either $f$ or $g$ is strongly convex, our second method can accelerate up to the $\BigO{\frac{1}{k^2}}$-convergence rate aka \cite{xu2017accelerated}, but it has certain advantages compared to \cite{xu2017accelerated}.
First, it is a primal method without Lagrange multipliers. 
Second, it linearizes the penalty term in the $y$-subproblem (see Algorithm~\ref{alg:A2} for details), which reduces the per-iteration complexity.
Third, it either takes averaging only on the $y$-sequence or uses its last iterate with one additional proximal operator of $g$. 
Finally, the $y$-averaging sequence is weighted.

\vspace{0.25ex}
\noindent\textbf{Our contribution:}
Our contribution can be summarized as follows:
\begin{itemize}
\vspace{-1.5ex}
\item[(a)] We propose a new proximal alternating penalty algorithm called PAPA to solve the generic constrained convex problem \eqref{eq:constr_cvx}.
We show that, under the existence of a saddle point, our method achieves the best-known $\BigO{\frac{1}{k}}$ convergence rate on both the objective residual and the feasibility violation on \eqref{eq:constr_cvx} without strong convexity, Lipschitz gradient continuity, and the boundedness of the domain of $f$ and $g$.
Moreover, our guarantee is on the last iterate of the primal variable instead of its averaging sequence.
In addition, we allow one to linearize the penalty term in the second subproblem of $\yb$ (see Step~\ref{eq:alter_scheme} of Algorithm~\ref{alg:A1} below) that significantly reduces the per-iteration complexity.
We also flexibly update all the algorithmic parameters using analytical update rules.

\item[(b)] If one objective term of \eqref{eq:constr_cvx} is strongly convex (i.e., either $f$ or $g$ is strongly convex), then we propose a new variant that combines both Nesterov's optimal scheme (or FISTA) \cite{Beck2009,Nesterov1983} and Tseng's variant \cite{auslender2006interior,tseng2008accelerated} to solve \eqref{eq:constr_cvx}.
We prove that this variant can achieve up to $\BigO{\frac{1}{k^2}}$-convergence rate on both the objective residual and the feasibility violation.
Such a rate is attained by either averaging only on the $\xb/\yb$-sequence or using one additional proximal operator of $f/g$.

\item[(c)] We customize our algorithms to obtain new variants for solving \eqref{eq:constr_cvx} and their extensions and special cases including the sum of three objective terms, and unconstrained composite convex problems.
Some of these variants are new.
We also interpret our algorithms as new variants of the primal-dual first-order method.
As a byproduct, these variants preserve the same convergence rate as in the proposed algorithms. 
We also discuss restarting strategies for our methods to significantly improve their practical performance.
The convergence guarantee of this strategy will be presented in our forthcoming work \cite{TranDinh2017f}.
\vspace{-1ex}
\end{itemize}

Let us clarify the following points of our contribution.
First, by utilizing the results in \cite{woodworth2016tight}, one can show that under only the convexity (respectively,  the semi-strong convexity) and the existence of a saddle point of \eqref{eq:constr_cvx}, our convergence rate $\BigO{\frac{1}{k}}$ (respectively, $\BigO{\frac{1}{k^2}}$) is optimal in the sense of black-box models for optimization complexity theory \cite{Nemirovskii1983} as shown in \cite{Tran-Dinh2018}.
The non-ergodic rate is very important for sparse and low-rank optimization since averaging often destroys the sparsity or low-rankness as we previously mentioned.
It is also important in image processing to preserve image sharpness. 
Second, the linearization of the $y$-subproblem in Algorithm~\ref{alg:A1} and Algorithm~\ref{alg:A2} is useful when $\Ab$ is an orthogonal operator. 
This allows us to only use the proximal operator of both $f$ and $g$ and significantly reduces the per-iteration complexity compared to classical AMA and ADMM.  
Third, when applying our method to a composite convex problem, we obtain new variants which are different from existing works.
Finally, we allow one to handle general constraints in $\Kc$ without shifting the problem into linear equality constraints.
This is very convenient to handle inequality constraints, convex cones, or boxed constraints as long as the projection onto $\Kc$ is efficient to compute, see Subsection \ref{subsubsec:conic} for a concrete conic programming example.

\vspace{0.25ex}
\noindent\textbf{Paper organization:}
The rest of this paper is organized as follows.
Section~\ref{sec:prelim_results} recalls the dual problem of \eqref{eq:constr_cvx}, and states a fundamental assumption and the optimality condition.
It also defines the quadratic penalty function for  \eqref{eq:constr_cvx} and proves a key lemma.
Section~\ref{sec:papa_algs} presents the main contribution with two algorithms and their convergence analysis.
Section~\ref{sec:variants_extensions} deals with some extensions and variants of the two proposed methods.
Section~\ref{sec:num_examples} provides several numerical examples to illustrate our theoretical development and compares with existing methods.
For clarity of exposition, all technical proofs are deferred to Appendix \ref{sec:appendix}.

\beforesec
\section{Preliminaries: Duality,  optimality condition, and quadratic penalty}\label{sec:prelim_results}
\aftersec
We first define the dual problem of \eqref{eq:constr_cvx} and recall its optimality condition. 
Then, we define the quadratic penalty function for \eqref{eq:constr_cvx} and prove  a key lemma on the objective residual and the feasibility violation.

\beforesubsec
\subsection{Basic notation}
\aftersubsec
We work on finite dimensional Euclidean spaces, $\R^p$ and $\R^n$, equipped with a standard inner product $\iprods{\cdot,\cdot}$ and Euclidean norm $\norm{\cdot} := \iprods{\cdot, \cdot}^{1/2}$.
Given a nonempty, closed, and convex set $\Kc\subseteq\R^n$, we use $\Nc_{\Kc}(\xb)$ for its normal cone at $\xb$, $\relint{\Kc}$ for its relative interior, and define $\Kc^{\circ} := \set{\xb\in\R^n \mid \iprods{\xb, \ub} \leq 1, ~\ub\in\Kc}$ for its polar set; 
we also use $\delta_{\Kc}(\cdot)$ and $s_{\Kc}(\cdot)$  to denote its indicator and support functions, respectively.
If $\Kc$ is a cone, then $\Kc^{\ast} := \set{\xb\in\R^n \mid \iprods{\xb, \ub} \geq 0, ~\ub\in\Kc}$ stands for its dual cone.
Given a proper, closed, and convex function $f$, $\dom{f}$ denotes its domain, $\partial{f}(\cdot)$ is its subdifferential, $f^{\ast}(\yb) := \sup_{\xb}\set{\iprods{\yb,\xb} - f(\xb)}$ is its Fenchel conjugate, and 
\begin{equation}\label{eq:prox_oper}
\kprox{\gamma f}{\xb} := \argmin_{\ub\in\R^p}\set{ f(\ub) + \tfrac{1}{2\gamma}\norms{\ub - \xb}^2}
\end{equation}
is called its the proximal operator, where $\gamma > 0$.
In this case, we have
\begin{equation}\label{eq:Moreau_identity}
\kprox{\gamma f}{\xb} + \gamma \kprox{f^{\ast}/\gamma}{\xb/\gamma} = \xb,
\end{equation}
which is known as Moreau's identity.
We say that $f$ is $L_f$-Lipschitz gradient continuous if it is differentiable, and its gradient $\nabla{f}$ is Lipschitz continuous on its domain with the Lipschitz constant $L_f \in [0, +\infty)$.
We say that $f$ is $\mu_f$-strongly convex if $f(\cdot) - \frac{\mu_f}{2}\norms{\cdot}^2$ is convex, where $\mu_f > 0$ is its strong convexity parameter.
Without loss of generality, we assume that $f$ is $\mu_f$-strongly convex with $\mu_f \geq 0$ to cover also convex functions.
For more details, we refer the reader to \cite{Bauschke2011,Rockafellar1970}.
The notation ``$:=$'' stands for ``is defined as''.

\beforesubsec
\subsection{Dual problem, fundamental assumption, and KKT condition} 
\aftersubsec
Let us define the Lagrange function associated with \eqref{eq:constr_cvx} as
\begin{equation*}
\Lc(\xb, \yb, \rb, \lbd) := f(\xb) + g(\yb) - \iprods{\Ab\xb + \Bb\yb - \rb - \cb, \lbd},
\end{equation*}
where $\lbd$ is the vector of Lagrange multipliers, and $r \in\Kc$.
The dual function is defined as
\begin{equation*}
d(\lbd) :=  {\!\!\!} \displaystyle\max_{(\xb,\yb)\in\dom{F}}{\!\!\!}\Big\{\iprods{\Ab\xb + \Bb\yb - \cb, \lbd} - f(\xb) - g(\yb) \Big\} = f^{\ast}(A^{\top}\lbd) + g^{\ast}(\Bb^{\top}\lbd) - \iprods{\cb,\lbd},
\end{equation*}
where $\dom{F} := \dom{f}\times\dom{g}$, and $f^{\ast}$ and $g^{\ast}$ are the Fenchel conjugates of $f$ and $g$, respectively.
The dual problem of \eqref{eq:constr_cvx} is
\begin{equation}\label{eq:dual_prob}
{\!\!\!}D^{\star} := {\!\!\!}\min_{\lbd \in \R^n} \set{ D(\lbd) := d(\lbd) + s_{\Kc}(-\lbd) \equiv f^{\ast}(A^{\top}\lbd) + g^{\ast}(\Bb^{\top}\lbd) - \iprods{\cb,\lbd} + s_{\Kc}(-\lbd)},{\!\!\!}
\end{equation}
where $s_{\Kc}(\vb) := \sup\set{\iprods{\vb,\rb} \mid \rb\in\Kc}$ is the support function of $\Kc$.
If $\Kc$ is a nonempty, closed, and convex cone, then \eqref{eq:dual_prob} reduces to
\begin{equation*}
D^{\star} := \min_{\lbd\in -\Kc^{\ast}}\set{ D(\lbd) :=  f^{\ast}(A^{\top}\lbd) + g^{\ast}(\Bb^{\top}\lbd) - \iprods{\cb,\lbd}},
\end{equation*}
where $\Kc^{\ast}$ is the dual cone of $\Kc$.

We say that a point $(\xb^{\star},\yb^{\star}, \rb^{\star}, \lbd^{\star}) \in \dom{f}\times\dom{g}\times\Kc\times\R^n$ is a saddle point of the Lagrange function $\Lc$ if for $(\xb,\yb)\in\dom{F}$, $\rb\in\Kc$ and $\lbd\in\R^n$, one has
\begin{equation}\label{eq:saddle_point}
\Lc(\xb^{\star}, \yb^{\star}, \rb^{\star}, \lbd) \leq \Lc(\xb^{\star}, \yb^{\star},\rb^{\star}, \lbd^{\star}) \leq \Lc(\xb, \yb, \rb, \lbd^{\star}).
\end{equation}
We denote by $\Sc^{\star} := \set{ (\xb^{\star}, \yb^{\star}, \rb^{\star}, \lbd^{\star})}$ the set of saddle points of $\Lc$, by $\Zc^{\star} := \set{(\xb^{\star}, \yb^{\star})}$, 
and by $\Lbd^{\star} := \set{\lbd^{\star}}$ the set of the optimal multipliers $\lbd^{\star}$.

In this paper, we rely on the following assumption.

\begin{assumption}\label{as:A1}
Both functions $f$ and $g$ are proper, closed, and convex, and $\Kc$ is a nonempty, closed, and convex set in $\R^n$.
The set of saddle points $\Sc^{\star}$ of $\Lc$ is nonempty, and the optimal value $F^{\star}$ is finite and attainable at some $(\xb^{\star},\yb^{\star}) \in \Zc^{\star}$.
\end{assumption}

We assume that Assumption~\ref{as:A1} holds throughout this paper without recalling it in the sequel.
Under this assumption, the optimality condition (or the KKT condition) of \eqref{eq:constr_cvx} can be written as
\begin{equation}\label{eq:opt_cond}
0 \in \partial{f}(\xopt) - \Ab^{\top}\lbdopt, ~~0 \in \partial{g}(\yopt) - \Bb^{\top}\lbdopt, ~~ \lbd^{\star} \in \Nc_{\Kc}(\Ab\xopt + \Bb\yopt - \cb),
\end{equation}
where $\Nc_{\Kc}(\cdot)$ is the normal cone of $\Kc$.
Let us assume that the following Slater condition holds:
\begin{equation*}
\relint{\dom{F}}\cap\set{(\xb, \yb) \mid \Ab\xb + \Bb\yb - \cb \in \relint{\Kc}} \neq\emptyset.
\end{equation*}
Then the optimality condition \eqref{eq:opt_cond} is necessary and sufficient for the strong duality of \eqref{eq:constr_cvx} and \eqref{eq:dual_prob} to hold, i.e., $F^{\star} + D^{\star} = 0$, and the dual solution is attainable and $\Lambda^{\star}$ is bounded, see, e.g., \cite{Bertsekas1999}.

\vspace{-1ex}
\beforesubsec
\subsection{Quadratic penalty function and its properties}
\aftersubsec
Let us define the quadratic penalty function $\Phi_{\rho}$ for the constrained problem \eqref{eq:constr_cvx} as
\begin{equation}\label{eq:Phi_func}
\Phi_{\rho}(\zb) := f(\xb) + g(\yb) + \rho\psi(\xb, \yb),~~~\text{where}~~\psi(\xb, \yb) := \tfrac{1}{2}\kdist{\Kc}{\Ab\xb + \Bb\yb - \cb}^2,
\end{equation}
$\zb := (\xb, \yb)$, $\rho > 0$ is a penalty parameter, and $\kdist{\Kc}{u}$ is the Euclidean distance from $u$ to $\Kc$.
Let us denote by $\proj_{\Kc}(\cdot)$ the projection operator onto $\Kc$. 
Then, we can write $\psi(\cdot)$ in \eqref{eq:Phi_func} as
\begin{equation*}
\psi(\xb, \yb) := \tfrac{1}{2}\min_{\rb\in\Kc}\norms{\rb - (\Ab\xb + \Bb\yb - \cb)}^2 = \tfrac{1}{2}\norms{\Ab\xb + \Bb\yb - \cb - \kproj{\Kc}{\Ab\xb + \Bb\yb - \cb}}^2.
\end{equation*}
From the definition of $\Phi_{\rho}$, we have the following result, whose proof is similar to \cite[Lemma 1]{TranDinh2015b}; however, we provide here a short proof for completeness.

\begin{lemma}\label{le:approx_opt_cond}
Let $\Phi_{\rho}(\cdot)$ be the quadratic penalty function defined by \eqref{eq:Phi_func}, and $S_{\rho}(\zb) := \Phi_{\rho}(\zb) - F^{\star}$.
Then, for any $\zb = (\xb, \yb) \in\dom{F}$, and $\lbd^{\star}\in\Lbd^{\star}$, we have
\begin{equation}\label{eq:approx_opt_cond}
\begin{cases}
-\norms{\lambda^{\star}}\kdist{\Kc}{\Ab\xb+\Bb\yb - \cb} &\leq F(\zb) - F^{\star}\leq S_{\rho}(\zb) - \frac{\rho}{2}\kdist{\Kc}{\Ab\xb + \Bb\yb - \cb}^2, \vspace{1.5ex}\\
\kdist{\Kc}{\Ab\xb+\Bb\yb - \cb} & \leq \tfrac{1}{\rho}\left[\norms{\lbd^{\star}} + \sqrt{\norms{\lbd^{\star}}^2 + 2\rho S_{\rho}(\zb)}\right],
\end{cases}
\end{equation}
where $\norms{\lbd^{\star}}^2 + 2\rho S_{\rho}(\zb) \geq  \tfrac{\rho^2}{2}\norms{\Ab\xb + \Bb\yb - \cb - \kproj{\Kc}{\Ab\xb + \Bb\yb - \cb } + \tfrac{1}{\rho}\lbd^{\star}}^2  \geq 0$.
\end{lemma}

\begin{proof}
Since \eqref{eq:saddle_point} holds, we have the following inequality for any $\rb\in\Kc$:
\begin{equation*}
F(\zb^{\star}) = \Lc(\zb^{\star}, \rb^{\star}, \lbd^{\star}) \leq \Lc(\zb, \rb, \lbd^{\star}) = F(\zb) - \iprods{\lambda^{\star}, \Ab\xb + \Bb\yb - \rb - \cb}.
\end{equation*}
Therefore, using this, $\rb = \rb^{\ast} = \kproj{\Kc}{\Ab\xb + \Bb\yb - \cb} \in \Kc$, and  $S_{\rho}(\cdot)$, we obtain
\begin{equation}\label{eq:lm21_est1}
\begin{array}{ll}
S_{\rho}(\zb) - \frac{\rho}{2}\kdist{\Kc}{\Ab\xb + \Bb\yb - \cb}^2 &=  F(\zb) - F(\zb^{\star}) \geq \iprods{\lambda^{\star}, \Ab\xb + \Bb\yb - \rb^{\ast} - \cb} \vspace{1ex}\\
&\geq -\norms{\lbd^{\star}}\norms{\Ab\xb + \Bb\yb - \cb - \rb^{\ast}} \vspace{1ex}\\
& = -\norms{\lbd^{\star}}\kdist{\Kc}{\Ab\xb + \Bb\yb - \cb},
\end{array}
\end{equation}
which is the first inequality of \eqref{eq:approx_opt_cond}.
Next, since $\frac{\rho}{2}\norms{\ub - \rb^{\ast}}^2 + \frac{1}{2\rho}\norm{\lbd^{\star}}^2 + \iprods{\lbd^{\star}, \ub} \geq \frac{\rho}{2}\norms{\ub - \rb^{\ast} + \frac{1}{\rho}\lbd^{\star}} \geq 0$ for $\ub = \Ab\xb + \Bb\yb - \cb$, we obtain
\begin{equation*}
\tfrac{\rho}{2}\kdist{\Kc}{\Ab\xb + \Bb\yb - \cb}^2 + \tfrac{1}{2\rho}\norms{\lbd^{\star}}^2 + \iprods{\lbd^{\star}, \Ab\xb + \Bb\yb \!-\! \cb \!-\! \rb^{\ast}} \!=\! \tfrac{\rho}{2}\norms{\Ab\xb + \Bb\yb - \cb - \rb^{\ast} + \tfrac{1}{\rho}\lbd^{\star}}^2 \geq 0.
\end{equation*}
Summing up this estimate and the first inequality of \eqref{eq:lm21_est1}, we obtain 
\begin{equation*}
S_{\rho}(\zb) + \tfrac{1}{2\rho}\norms{\lbd^{\star}}^2 \geq  \tfrac{\rho}{2}\norms{\Ab\xb + \Bb\yb - \cb - \kproj{\Kc}{\Ab\xb + \Bb\yb - \cb } + \tfrac{1}{\rho}\lbd^{\star}}^2  \geq 0.
\end{equation*}
The second inequality of \eqref{eq:approx_opt_cond} is a consequence of the first one by solving the following quadratic inequation $\rho t^2 - 2\norms{\lbd^{\star}}t - 2S_{\rho}(\zb) \leq 0$ in $t$ with $t \geq 0$.
\Eproof
\end{proof}

\vspace{-2.5ex}
\beforesec
\section{Proximal Alternating Penalty Algorithms}\label{sec:papa_algs}
\aftersec
\vspace{-0.5ex}
Our algorithms rely on an alternating strategy applying to the quadratic penalty function $\Phi_{\rho}(\cdot)$ defined by \eqref{eq:Phi_func}, Nesterov's accelerated scheme \cite{Nesterov2004}, and the adaptive strategy for parameters in \cite{TranDinh2015b}.
We present our first algorithm for non-strongly convex objective functions in Subsection \ref{subsec:nonstrongly_convex_case}, and then describe the second algorithm for semi-strongly convex objective function in Subsection~\ref{subsec:papa_for_scvx}.

\beforesubsec
\subsection{PAPA for non-strongly convex problems}\label{subsec:nonstrongly_convex_case}
\aftersubsec
At each iteration $k\geq 0$ of our algorithm, given $\hat{\zb}^k := (\hat{\xb}^k, \hat{\yb}^k) \in \dom{F}$ and $\gamma_k \geq 0$, we need to solve the following $x$-subproblem: 
\begin{equation}\label{eq:x_cvx_subprob}
\xb^{k+1} \in \Sc_{\gamma_k}(\hat{\xb}^k,\hat{\yb}^k;\rho_k) := \argmin_{\xb\in\R^{p_1}}\set{ f(\xb) + \rho_k\psi(\xb, \hat{\yb}^k) + \tfrac{\gamma_k}{2}\norms{\xb - \hat{\xb}^k}^2 }.
\end{equation}
When $\gamma_k = 0$, $\Sc_{\gamma_k}(\cdot)$ can be a multivalued mapping.
Since  $\Sc_{\gamma_k}(\cdot) \neq\emptyset$ for $\gamma_k > 0$, without loss of generality, we assume that $\Sc_{\gamma_k}(\cdot)$ is nonempty for any $\gamma_k \geq 0$. 

We consider the case where both $f$ and $g$ in \eqref{eq:constr_cvx} are non-strongly convex (i.e., both $\mu_f$ and $\mu_g$ are zero) and not Lipschitz gradient continuous.

\vspace{-1ex}
\beforesubsubsec
\subsubsection{The algorithm} 
\aftersubsubsec
We present our first algorithm to solve \eqref{eq:constr_cvx} in Algorithm \ref{alg:A1}, where we name it by the ``Proximal Alternating Penalty Algorithm'' (shortly, PAPA).

\begin{algorithm}[hpt!]\caption{{\!}(\textit{\textbf{P}roximal \textbf{A}lternating \textbf{P}enalty \textbf{A}lgorithm} - Nonstrong convexity){\!\!\!\!}}\label{alg:A1}
\begin{normalsize}
\begin{algorithmic}[1]
	\State {\hskip0ex}\textbf{Initialization:} 
	Choose an initial point $(x^0, y^0) \in \dom{F}$, and two initial values $\rho_0 > 0$ and $\gamma_0 \geq 0$. 
	Set $\hat{x}^0 := x^0$, and $\hat{\yb}^0 := \yb^0$.
	\vspace{1ex}
	\State \textbf{For $k := 0$ to $k_{\max}$ perform}
		\vspace{1ex}
		\State{\hskip2ex}\label{eq:alter_scheme}Update 
		$\left\{\begin{array}{ll}
                  \xb^{k+1} & \in \Sc_{\gamma_k}(\hat{\xb}^k,\hat{\yb}^k; \rho_k), \vspace{1ex}\\ 
                 \yb^{k+1}  &:=  \kprox{g/(\rho_k\norms{\Bb}^2)}{\hat{\yb}^k - \tfrac{1}{\norms{\Bb}^2}\nabla_y{\psi}(\xb^{k+1}, \hat{\yb}^k)}, \vspace{1ex}\\
                 (\hat{\xb}^{k+1}, \hat{\yb}^{k+1} ) &:=  (\xb^{k+1}, \yb^{k+1}) + \tfrac{k}{k+2}(\xb^{k+1} - \xb^k, \yb^{k+1} - \yb^k).
                 \end{array}\right.$
                 \vspace{1ex}
                 \State{\hskip2ex}\label{step:A1_param_update}Update $\rho_{k+1} :=  (k+2)\rho_0$~~and~~$\gamma_{k+1} := \gamma_0(k+2)$.
	\State\textbf{End~for}
\end{algorithmic}
\end{normalsize}
\end{algorithm}
 
\noindent 
Algorithm \ref{alg:A1} looks rather simple with four lines in the main loop.
Before analyzing its convergence, we make the following comments:

$\mathrm{(a)}$~Firstly, Algorithm~\ref{alg:A1} adopts the idea of Nesterov's first accelerated method in \cite{Beck2009,Nesterov1983} to accelerate the penalized problem $\min_{\xb,\yb}\Phi_{\rho}(\xb,\yb)$ studied, e.g., in \cite{Lan2013,Nocedal2006}. 
However, it first alternates between $\xb$ and $\yb$ to decouple the quadratic penalty term $\psi(\xb,\yb)$ compared to \cite{Lan2013}.
Next, it linearizes the second subproblem in $\yb$ to use $\prox_g$.
Finally, it is combined with the adaptive strategy for parameters in \cite{TranDinh2015b} to update the penalty parameter $\rho$ so that its last iterate sequence $\sets{(\xb^k, \yb^k)}$ converges to a solution $(\xb^{\star},\yb^{\star})$ of the original problem \eqref{eq:constr_cvx}.

$\mathrm{(b)}$~Secondly, if the $\xb$-subproblem \eqref{eq:x_cvx_subprob} with $\gamma_k = 0$, i.e.:
\begin{equation}\label{eq:unreg_subprob_x}
\xb^{k+1} \in \mathcal{S}(\hat{\yb}^k;\rho_k) := \argmin_{\xb}\big\{ f(\xb) + \rho_k\psi(\xb, \hat{\yb}^k) \big\}
\end{equation}
is solvable (not necessarily unique, e.g., when $\dom{f}$ is compact or $\Ab$ is orthogonal), then the main step, Step~\ref{eq:alter_scheme}, in Algorithm~\ref{alg:A1} reduces to
\begin{equation}\label{eq:alter_scheme_1b}
\left\{\begin{array}{ll}
\xb^{k+1} & \in \mathcal{S}(\hat{\yb}^k;\rho_k), \vspace{1ex}\\
\yb^{k+1} & := \kprox{g/(\rho_k\norms{\Bb}^2)}{\hat{\yb}^k - \tfrac{1}{\norms{\Bb}^2}\nabla_y{\psi}(\xb^{k+1}, \hat{\yb}^k)}, \vspace{1ex}\\
\hat{\yb}^{k+1}  &:=  \yb^{k+1} + \tfrac{k}{k+2}(\yb^{k+1} - \yb^k).
\end{array}\right.
\end{equation}
In this case, only one parameter $\rho_k$ is involved in \eqref{eq:alter_scheme_1b}, and the term $\norms{\xb^0 - \xb^{\star}}^2$ disappears in the bounds of Theorem~\ref{th:convergence1} below.
If $\Ab = \Id$, the identity operator, then the two first steps of \eqref{eq:alter_scheme_1b} becomes
\begin{equation*}
\xb^{k+1} :=  \kprox{f/\rho_k}{\cb - \Bb\hat{\yb}^k}~~\text{and}~~\yb^{k+1} :=  \kprox{g/(\rho_k\norms{\Bb}^2)}{\hat{\yb}^k - \tfrac{1}{\norms{\Bb}^2}\nabla_y{\psi}(\xb^{k+1}, \hat{\yb}^k)},
\end{equation*}
which only require the proximal operator of $f$ and $g$.

$\mathrm{(c)}$~Thirdly, the gradient $\nabla_y{\psi}(\xb^{k+1}, \hat{\yb}^k)$ is computed explicitly as 
\begin{equation*}
\nabla_y{\psi}(\xb^{k+1}, \hat{\yb}^k) = \Bb^{\top}(\hat{\ub}^k  - \proj_{\Kc}(\hat{\ub}^k)),~~~\text{where}~~~\hat{\ub}^k :=  \Ab\xb^{k+1} + \Bb\hat{\yb}^k - \cb,
\end{equation*}
which requires matrix-vector products $\Ab\xb$, $\Bb\yb$, and $\Bb^{\top}\ub$ each, and one projection onto $\Kc$.
When $\Kc$ is a simple set (e.g., box, cone, or simplex), the cost of computing $\proj_{\Kc}$ is minor.

$\mathrm{(d)}$~Fourthly, the convergence guarantee in Theorem~\ref{th:convergence1} is on the last iterate $(\xb^k, \yb^k)$ (i.e., without averaging) compared to, e.g., \cite{Chambolle2011,Davis2014,He2012,shefi2016rate}. 

$\mathrm{(e)}$~Fifthly, the update rule of $\rho_k$ and $\gamma_k$ at Step~\ref{step:A1_param_update} is not heuristically tuned.
The choice of $\rho_0$ trades-off the feasibility and the objective residual in the bound \eqref{eq:convergence1} below.
In our implementation, we choose $\rho_0 :=  \frac{1}{\norms{\Bb}}$ by default.

$\mathrm{(f)}$~Finally, for Algorithm~\ref{alg:A1}, we can also linearize the subproblem~\eqref{eq:x_cvx_subprob} at Step~\ref{eq:alter_scheme} to obtain the following closed form solution using the proximal operator of $f$:
\begin{equation*}
\xb^{k+1} := \kprox{f/\hat{\gamma}_k}{\hat{\xb}^k - \tfrac{\rho_k}{\hat{\gamma}_k}\nabla_x{\psi}(\hat{\xb}^k,\hat{\yb}^k)}.
\end{equation*}
In this case, we modify $\yb^{k+1} :=  \kprox{g/(\rho_k\norms{\Bb}^2)}{\hat{\yb}^k - \tfrac{1}{\norms{\Bb}^2}\nabla_y{\psi}(\hat{\xb}^{k}, \hat{\yb}^k)}$, which is no longer alternating.
Therefore, we can compute $\xb^{k+1}$ and $\yb^{k+1}$ in parallel.
The analysis of this variant is similar to \cite[Theorem 3]{TranDinh2015b}, and we omit the details.

We highlight that Algorithm~\ref{alg:A1} is different from alternating linearization method in \cite{Goldfarb2012}, alternating minimization (AMA) \cite{Tseng1991a}, and alternating direction methods of multipliers (ADMM) in the literature \cite{He2012,Ouyang2014,xu2017accelerated} as discussed in the introduction.

\beforesubsubsec
\subsubsection{Convergence analysis}
\aftersubsubsec
The convergence of Algorithm~\ref{alg:A1} is presented as follows.
\begin{theorem}\label{th:convergence1}
Let $\sets{(\xb^k, \yb^k)}$ be the sequence generated by Algorithm~\ref{alg:A1} for solving \eqref{eq:constr_cvx}. 
Then, for $k\geq 1$, we have
\begin{equation}\label{eq:convergence1}
\begin{cases}
\absc{F(\zb^{k}) - F^{\star}} &\leq ~~\dfrac{\max\sets{ \rho_0 R_p^2, ~2\norms{\lbd^{\star}} R_d }}{2\rho_0 k}, \vspace{1ex}\\
\kdist{\Kc}{\Ab\xb^{k} + \Bb\yb^{k} - \cb} &\leq \dfrac{R_d}{\rho_0 k },
\end{cases} 
\end{equation}
where $R_p^2 := \gamma_0\norms{\xb^0 - \xb^{\star}}^2 + \rho_0\norms{\Bb}^2\norms{\yb^0 - \yb^{\star}}^2$ and $R_d := \norms{\lbd^{\star}} + \sqrt{\norms{\lbd^{\star}}^2 + \rho_0R_p^2}$.
Consequently, the non-ergodic convergence rate of Algorithm~\ref{alg:A1} is $\BigO{\tfrac{1}{k}}$, i.e., $\vert F(\zb^k) - F^{\star}\vert \leq \BigO{\tfrac{1}{k}}$ and  $\mathrm{dist}_{\Kc} \big(\Ab\xb^k + \Bb\yb^k - \cb\big) \leq \BigO{\tfrac{1}{k}}$.
\end{theorem}

The proof of Theorem~\ref{th:convergence1} requires the following key lemma, whose proof can be found in Appendix~\ref{apdx:le:key_estimate}.

\begin{lemma}\label{le:key_estimate}
Let $\sets{(\xb^k, \yb^k, \hat{\xb}^k, \hat{\yb}^k)}$ be the sequence generated by Algorithm~\ref{alg:A1}.
Then, $(\hat{\xb}^k, \hat{\yb}^k)$ can be interpreted as 
\begin{equation}\label{eq:update_xtilde}
\begin{array}{lll}
&(\hat{\xb}^k, \hat{\yb}^k) &= (1-\tau_k)(\xb^k, \yb^k) + \tau_k(\tilde{\xb}^k,\tilde{\yb}^k), \vspace{1.25ex}\\
\text{with}~~&(\tilde{\xb}^{k+1}, \tilde{\yb}^{k+1}) &:= (\tilde{\xb}^k, \tilde{\yb}^k) + \tfrac{1}{\tau_k}(\xb^{k+1} - \hat{\xb}^k, \yb^{k+1} - \hat{\yb}^k),
\end{array}
\end{equation}
and $(\tilde{\xb}^0,\tilde{\yb}^0) := (\xb^0,\yb^0)$, where $\tau_k := \frac{1}{k+1} \in (0, 1]$.
Moreover, $\Phi_{\rho}$ defined by \eqref{eq:Phi_func} satisfies 
\begin{align}\label{eq:key_est1}
{\!\!\!\!\!\!}\begin{array}{ll}
\Phi_{\rho_k}(\zb^{k+1}) & \leq  (1-\tau_k)\Phi_{\rho_{k-1}}(\zb^k)  + \tau_kF(\zb^{\star}) + \tfrac{\gamma_k\tau_k^2}{2}\norms{\tilde{\xb}^k - \xb^{\star}}^2 \vspace{1ex}\\
& -  \tfrac{\gamma_k\tau_k^2}{2}\norms{\tilde{\xb}^{k+1} {\!\!} -\! \xb^{\star}}^2 + \tfrac{\rho_k\tau_k^2\norms{\Bb}^2}{2}\norms{\tilde{\yb}^k \!-\! \yb^{\star}}^2 - \tfrac{\rho_k\tau_k^2\norms{\Bb}^2}{2}\norms{\tilde{\yb}^{k+1} {\!\!}- \yb^{\star}}^2{\!\!\!} \vspace{1ex}\\
& - \tfrac{(1-\tau_k)}{2}\left[\rho_{k-1} - \rho_k(1-\tau_k)\right]\norms{\sb^k}^2,
\end{array}{\!\!\!\!}
\end{align}
where $\sb^k := \Ab\xb^k + \Bb\yb^k - \cb - \proj_{\Kc}\big(\Ab\xb^k + \Bb\yb^k - \cb\big)$.
\end{lemma}

\begin{proof}[The proof of Theorem~\ref{th:convergence1}]
The update rules $\tau_k := \frac{1}{k+1}$ and $\rho_k := \rho_0(k+1)$ from Algorithm~\ref{alg:A1} show that
\begin{equation*} 
\tau_0 := 1,~~~~\frac{(1-\tau_k)}{\rho_k\tau_k^2} = \frac{1}{\rho_{k-1}\tau_{k-1}^2}, ~~~\text{and}~~~\rho_{k-1} = (1-\tau_k)\rho_k. 
\end{equation*}
If we update $\gamma_k$ as $\gamma_{k+1} = \frac{\gamma_k(k+2)}{(k+1)} = \gamma_0(k+2)$, then $\frac{\gamma_{k+1}}{k+2} = \frac{\gamma_k}{k+1}$.

Let $S_{\rho}(\zb) := \Phi_{\rho}(\zb) - F^{\star}$. 
Using these equalities, we obtain from \eqref{eq:key_est1} that
\begin{equation*}
{\!\!\!\!}\begin{array}{ll}
(k+1)S_{\rho_k}(\zb^{k+1}) &+ \frac{\gamma_0}{2}\norms{\tilde{\xb}^{k+1} - \xb^{\star}}^2 +  \frac{\rho_0\norms{\Bb}^2}{2}\norms{\tilde{\yb}^{k+1} - \yb^{\star}}^2 \leq  kS_{\rho_{k-1}}(\zb^k) \vspace{1ex}\\
& +  \frac{\gamma_0}{2}\norms{\tilde{\xb}^k - \xb^{\star}}^2 + \frac{\rho_0\norms{\Bb}^2}{2}\norms{\tilde{\yb}^k - \yb^{\star}}^2.
\end{array}{\!\!\!\!}
\end{equation*} 
Let us denote by $a_k := \frac{\gamma_0}{2}\norms{\tilde{\xb}^k \!-\! \xb^{\star}}^2 \!+\! \frac{\rho_0\norms{\Bb}^2}{2}\norms{\tilde{\yb}^k \!-\! \yb^{\star}}^2$.
Then, the last estimate can be simplified as
\begin{equation*} 
\begin{array}{ll}
(k+1)S_{\rho_k}(\zb^{k+1}) + a_{k+1} \leq kS_{\rho_{k-1}}(\zb^k) + a_k.
\end{array}
\end{equation*} 
By induction and $a_k\geq 0$, we can show that $S_{\rho_k}(\zb^{k+1}) \leq \frac{a_0}{k+1}$, which leads to
\begin{equation}\label{eq:key_est10}
S_{\rho_k}(\zb^{k+1}) \leq \tfrac{1}{2(k+1)}\left[\gamma_0\norms{\tilde{\xb}^0 - \xb^{\star}}^2 + \rho_0\norms{\Bb}^2\norms{\tilde{\yb}^0 - \yb^{\star}}^2\right].
\end{equation}
Using this estimate into Lemma~\ref{le:approx_opt_cond} and note that $\rho_k = \rho_0(k+1)$, $\tilde{\xb}^0 = \hat{\xb}^0 = \xb^0$, and $\tilde{\yb}^0 = \hat{\yb}^0 = \yb^0$, we obtain \eqref{eq:convergence1}.
\Eproof
\end{proof}


\vspace{-1.5ex}
\begin{remark}[Non-acceleration]\label{re:non_acceleration}
Algorithm~\ref{alg:A1} adopts the Nesterov acceleration method to achieve $\BigO{\frac{1}{k}}$-rate. 
If we remove the acceleration step (i.e., set $\hat{\zb}^k = \zb^k$ and $\tilde{\zb}^k = \zb^k$ at all iterations), then one can show that the convergence rate of the non-acceleration variant of Algorithm~\ref{alg:A1} reduces to $\BigO{\frac{1}{\sqrt{k}}}$, i.e.,  $\vert F(\zb^k) - F^{\star}\vert \leq \BigO{\tfrac{1}{\sqrt{k}}}$ and  $\mathrm{dist}_{\Kc} \big(\Ab\xb^k + \Bb\yb^k - \cb\big) \leq \BigO{\tfrac{1}{\sqrt{k}}}$.
The proof of this result can be derived from Lemma \ref{le:key_descent_lemma}, and we omit its details in this paper. 
\end{remark}

\beforesubsec
\subsection{PAPA for the semi-strong convexity case}\label{subsec:papa_for_scvx}
\aftersubsec
In Algorithm~\ref{alg:A1}, we have not been able to prove a better convergence rate than $\BigO{\frac{1}{k}}$ when one objective term $f$ or $g$ is strongly convex.
Without loss of generality, we can assume that $g$ is $\mu_g$-strongly convex with  $\mu_g > 0$.

In this subsection, we propose a new algorithm that allows us to exploit the strong convexity of $g$ in order to improve the convergence rate from $\BigO{\frac{1}{k}}$ to $\BigO{\frac{1}{k^2}}$.
This algorithm can be viewed as a hybrid variant between Tseng's accelerated proximal gradient \cite{tseng2008accelerated} and Nesterov's  scheme in \cite{Nesterov1983}.

\beforesubsubsec
\subsubsection{The algorithm}
\aftersubsubsec
The details of the algorithm are presented in Algorithm~\ref{alg:A2}.

\begin{algorithm}[htp!]\caption{(\textit{\textbf{P}roximal \textbf{A}lternating \textbf{P}enalty \textbf{A}lgorithm} - Semi-strong convexity)}\label{alg:A2}
\begin{normalsize}
\begin{algorithmic}[1]
	\State {\hskip0ex}\textbf{Initialization:} 
	Choose an initial point $(x^0, y^0)\in\dom{F}$, and two initial values $\rho_0 \in \left(0, \frac{\mu_g}{2\norms{\Bb}^2}\right]$ and $\gamma_0 \geq 0$. 
	Set $\tau_0 := 1$,  ~$\hat{\xb}^0 := \xb^0$, and $\tilde{\yb}^0 := \yb^0$.
	\vspace{1ex}
	\State \textbf{For $k := 0$ to $k_{\max}$ perform}
		\vspace{1ex}
		\State{\hskip2ex}\label{step:A2_step6}Update $\tau_{k+1} :=  \frac{\tau_k}{2}\big( (\tau_k^2 + 4)^{1/2} - \tau_k \big)$.
	        \vspace{1ex}
		\State{\hskip2ex}\label{step:x_cvx_subprob}Update 
		$\left\{\begin{array}{ll}
		\hat{\yb}^k &:= (1-\tau_k)\yb^k + \tau_k\tilde{\yb}^k, \vspace{1ex}\\
		\xb^{k+1} &\in \Sc_{\gamma_0}(\hat{\xb}^k,\hat{\yb}^k;\rho_k), \vspace{1ex}\\
		\hat{\xb}^{k+1} &:= \xb^{k+1} + \tfrac{\tau_{k+1}(1-\tau_k)}{\tau_k}(\xb^{k+1} - \xb^k), \vspace{1ex}\\
		\tilde{\yb}^{k+1} &:= \kprox{g/(\tau_k\rho_k\norms{\Bb}^2)}{\tilde{\yb}^k - \tfrac{1}{\tau_k\norms{\Bb}^2}\nabla_y{\psi}(\xb^{k+1}, \hat{\yb}^k)}.
		\end{array}\right.$
		\vspace{1ex}
	        \State{\hskip2ex}\label{step:A2_step5} Perform \textbf{one} of the following \textbf{two steps}:
	        \begin{equation*}
	        \begin{array}{llll}
	          \textbf{Option 1:} & \yb^{k+1} &:=  (1-\tau_k)\yb^k + \tau_k\tilde{\yb}^{k+1} &~\text{(Averaging step)}. \vspace{1.5ex}\\
	          \textbf{Option 2:} & \yb^{k+1} &:= \prox_{g/(\rho_k\norms{\Bb}^2)}\Big( \hat{\yb}^k - \tfrac{1}{\norms{\Bb}^2}\nabla_y{\psi}(\xb^{k+1}, \hat{\yb}^k)\Big) &~\text{(Proximal step)}.
	          \end{array}
	         \end{equation*}
	        \State{\hskip2ex}\label{step:A2_step7}Update $\rho_{k+1} :=  \frac{\rho_k}{1 - \tau_{k+1}}$.
	        \vspace{1ex}
	\State\textbf{End~for}
\end{algorithmic}
\end{normalsize}
\end{algorithm}

\noindent Before analyzing the convergence of Algorithm~\ref{alg:A2}, we make the following remarks.

$\mathrm{(a)}$~Similar to Algorithm~\ref{alg:A1}, when the $x$-subproblem \eqref{eq:unreg_subprob_x} is solvable, we do not need to add the regularization term $\frac{\gamma_0}{2}\norms{\xb - \hat{\xb}^k}^2$.
In this case, there are only two parameters $\tau_k$ and $\rho_k$ involved in Algorithm~\ref{alg:A2}, and the term $\norms{\xb^0 - \xb^{\star}}^2$ also disappears in the convergence bound \eqref{eq:convergence2} of Theorem~\ref{th:convergence2} below.

$\mathrm{(b)}$~The update of $\tau_k$ at Step~\ref{step:A2_step6} is standard in accelerated methods. 
Indeed, if we define $t_k := \frac{1}{\tau_k}$, then we obtain the well-known Nesterov update rule \cite{Nesterov1983} for $t_k$ as $t_{k+1} = \frac{1}{2}\big(1 + (1 + 4t_k^2)^{1/2}\big)$ with $t_0 := 1$.
However, as shown in our proof below, we  can update $\tau_k$ and $\rho_k$ based on the following tighter conditions:
\begin{equation*}
\frac{\rho_k\tau_k^2\norms{\Bb}^2}{1-\tau_k} = \rho_{k-1}\tau_{k-1}^2\norms{\Bb}^2 + \mu_g \tau_{k-1} ~~~~\text{and}~~~~ \rho_k = \frac{\rho_{k-1}}{1 - \tau_k}.
\end{equation*}
These conditions lead to a new update rule for $\rho_k$ and $\tau_k$ as
\vspace{-0.75ex}
\begin{equation*}
\tau_k :=  \frac{\big(\tau_{k-1}^2 + \kappa\tau_{k-1}/\rho_{k-1}\big)^{1/2}}{1 + \big(\tau_{k-1}^2 + \kappa\tau_{k-1}/\rho_{k-1}\big)^{1/2}},~~~\text{and}~~~\rho_k :=  \frac{\rho_{k-1}}{1 - \tau_k},
\vspace{-0.75ex}
\end{equation*}
where $\kappa := \frac{\mu_g}{\norms{\Bb}^2}$.
This update requires $\mu_g$.  
In this case, we still have the same guarantee as in Theorem~\ref{th:convergence2}.
The update of $\rho_k$ is the same as in Algorithm~\ref{alg:A1}, i.e. $\rho_k :=  \frac{\rho_{k-1}}{1 - \tau_k}$, but $\gamma_0$ is fixed for all $k\geq 0$.  

$\mathrm{(c)}$~To achieve $\BigO{\frac{1}{k^2}}$-convergence rate, we only require the strong convexity on one objective term, i.e., $\mu_g > 0$.
In addition, we can compute $\sets{\yb^k}$ with averaging as in \textbf{Option 1} or with one additional proximal operator $\prox_g$ of $g$ as in \textbf{Option 2}.
For \textbf{Option 1}, the weighted averaging sequence is only on $\sets{\yb^k}$ but not on $\sets{\xb^k}$. 
This is different from a recent work in \cite{xu2017accelerated}, where the same convergence rate of ADMM is obtained for $\mu_g > 0$.
We emphasize that Algorithm~\ref{alg:A2} is fundamentally different from \cite{xu2017accelerated} as stated in the introduction.
The $\BigO{\frac{1}{k^2}}$ rate was also known for AMA  \cite{Goldstein2012}, but the guarantee is on the dual problem \eqref{eq:dual_prob}.
To achieve the same rate on \eqref{eq:constr_cvx}, an extra step is required, see \cite{tran2015construction}.

$\mathrm{(d)}$~The strong convexity of $g$ can be  relaxed to a quasi-strong convexity as studied in \cite{necoara2015linear}, where we assume that there exists $\mu_g > 0$ such that
\begin{equation*}
g(\yb) + \iprods{\nabla{g}(\yb), y^{\star} - \yb} + \tfrac{\mu_g}{2}\norms{\yb - \yb^{\star}}^2 \leq g(\yb^{\star}),~~~\forall \yb\in\dom{g}, ~\yb^{\star} \in \Yc^{\star},
\end{equation*}
where $\Yc^{\star}$ is the projection of the primal solution set $\Zc^{\star}$ onto $y$, and $\nabla{g}(\yb) \in\partial{g}(\yb)$.
As shown in \cite{necoara2015linear}, this condition is weaker than the strong convexity of $g$.

\beforesubsubsec
\subsubsection{Convergence analysis} 
\aftersubsubsec
We prove the following convergence result for Algorithm~\ref{alg:A2}.

\begin{theorem}\label{th:convergence2}
Let $\sets{(\xb^k, \yb^k)}$ be generated by Algorithm~\ref{alg:A2} for solving \eqref{eq:constr_cvx}. 
Then
\begin{equation}\label{eq:convergence2}
\begin{cases}
\vert F(\zb^k) - F^{\star} \vert &\leq \dfrac{2\max\set{ \rho_0 R_p^2, ~2\norms{\lbd^{\star}}R_d}}{\rho_0(k+1)^2}, \vspace{1ex}\\
\kdist{\Kc}{\Ab\xb^k + \Bb\yb^k - \cb} &\leq \dfrac{4R_d}{\rho_0(k+1)^2},
\end{cases}
\end{equation}
where $R_p^2 := \gamma_0\norms{\xb^0 \!\!-\! \xb^{\star}}^2 + \rho_0\norms{\Bb}^2\norms{\yb^0 \!\!-\! \yb^{\star}}^2$ and $R_d := \norms{\lbd^{\star}} + \sqrt{\norms{\lbd^{\star}}^2 + \rho_0R_p^2}$.
Consequently, the convergence rate of Algorithm~\ref{alg:A2} is $\BigO{\tfrac{1}{k^2}}$, i.e., $\vert F(\zb^k) - F^{\star}\vert \leq \BigO{\tfrac{1}{k^2}}$ and  $\mathrm{dist}_{\Kc} \big(\Ab\xb^k + \Bb\yb^k - \cb\big) \leq \BigO{\tfrac{1}{k^2}}$ either in semi-ergodic sense $($\textbf{Option 1}$)$ $($i.e., non-ergodic in $x$ and ergodic in $y$$)$ or in non-ergodic sense $($\textbf{Option 2}$)$.
\end{theorem}

To prove  Theorem~\ref{th:convergence2} we need  the following lemma (\emph{cf.} Appendix~\ref{apdx:le:tseng_key_est}).

\begin{lemma}\label{le:tseng_key_est}
Let $\sets{(\xb^k, \hat{\xb}^k,  \yb^k, \hat{\yb}^k, \tilde{\yb}^k)}$ be the sequence generated by Algorithm~\ref{alg:A2}. 
Then, $\hat{\xb}^k$  can be interpreted  as 
\begin{equation}\label{eq:x_hat}
\hat{\xb}^k = (1-\tau_k)\xb^k + \tau_k\tilde{\xb}^k,~~\text{with}~~\tilde{\xb}^0 := \xb^0,~~\text{and}~~\tilde{\xb}^{k+1} := \tilde{\xb}^k + \tfrac{1}{\tau_k}(\xb^{k+1} - \hat{\xb}^k).
\end{equation}
Moreover, the following estimate holds:
\begin{equation}\label{eq:key_est3} 
{\!\!\!\!\!}\begin{array}{ll}
\Phi_{\rho_k}(\zb^{k+1}) &{\!\!}\leq (1-\tau_k)\Phi_{\rho_{k-1}}(\zb^k)  + \tau_kF(\zb^{\star})  + \frac{\gamma_0\tau_k^2}{2}\norms{\tilde{\xb}^k {\!} - \xb^{\star}}^2  -  \frac{\gamma_0\tau_k^2}{2}\norms{\tilde{\xb}^{k+1} {\!\!}- \xb^{\star}}^2 \vspace{1ex}\\
& + \frac{\rho_k\tau_k^2\norms{\Bb}^2}{2}\norms{\tilde{\yb}^k - \yb^{\star}}^2  -  \frac{\rho_k\tau_k^2\norms{\Bb}^2 + \mu_g\tau_k}{2}\norms{\tilde{\yb}^{k+1} - \yb^{\star}}^2  \vspace{1ex}\\
& - \frac{(1-\tau_k)}{2}\left[\rho_{k-1} - \rho_k(1-\tau_k)\right]\norms{\sb^k}^2,
\end{array}{\!\!\!\!\!}
\end{equation}
where $\sb^k := \Ab\xb^k + \Bb\yb^k - \cb - \proj_{\Kc}\big(\Ab\xb^k + \Bb\yb^k - \cb\big)$.
\end{lemma}

\begin{proof}[The proof of Theorem~\ref{th:convergence2}]
For simplicity of notation, we denote by $S_k := S_{\rho_{k-1}}(\zb^k) = \Phi_{\rho_{k-1}}(\zb^k) - F^{\star}$.
Since $\rho_k$ is updated by $\rho_{k-1} = \rho_k(1-\tau_k)$, we can simplify \eqref{eq:key_est3} as follows:
\begin{equation*}
\begin{array}{ll}
S_{k+1} &+ \frac{\gamma_0\tau_k^2}{2}\norms{\tilde{\xb}^{k+1} - \xb^{\star}}^2 + \frac{\left(\rho_k\tau_k^2\norms{\Bb}^2 + \mu_g\tau_k\right)}{2}\norms{\tilde{\yb}^{k+1} - \yb^{\star}}^2 \leq  (1-\tau_k)S_k \vspace{1ex}\\
&+ \frac{\gamma_0\tau_k^2}{2}\norms{\tilde{\xb}^k - \xb^{\star}}^2 + \frac{\rho_k\tau_k^2\norms{\Bb}^2}{2}\norms{\tilde{\yb}^k - \yb^{\star}}^2.
\end{array}
\end{equation*}
Let us assume that the parameters $\tau_k$ and $\rho_k$ are updated such that
\begin{equation}\label{eq:key_cond4}
\frac{\rho_k\tau_k^2\norms{\Bb}^2}{1-\tau_k} \leq \rho_{k-1}\tau_{k-1}^2\norms{\Bb}^2 + \mu_g \tau_{k-1} ~~~~\text{and}~~~~ \frac{\gamma_0\tau_k^2}{1-\tau_k} \leq \gamma_0\tau_{k-1}^2.
\end{equation}
If we define $A_k := S_k +  \frac{\gamma_0\tau_{k-1}^2}{2}\norms{\tilde{\xb}^k - \xb^{\star}}^2 + \frac{\left(\norms{\Bb}^2\rho_{k-1}\tau_{k-1}^2 + \mu_g\tau_{k-1}\right)}{2}\norms{\tilde{\yb}^k - \yb^{\star}}^2$, then the above inequality implies
\begin{equation*}
A_{k+1} \leq (1-\tau_k)A_k.
\end{equation*}
By induction, we obtain
\begin{equation}\label{eq:proof3_est1}
A_{k+1} \leq  \omega_k \left[(1-\tau_0)S_0 + \tfrac{\gamma_0\tau_0^2}{2}\norms{\tilde{\xb}^0 - \xb^{\star}}^2 + \tfrac{\norms{\Bb}^2\rho_0\tau_0^2}{2}\norms{\tilde{\yb}^0 - \yb^{\star}}^2\right],
\end{equation}
where $\omega_k := \prod_{i=1}^k(1-\tau_i)$. 

Now, we write the update of $\tau_k$ and $\rho_k$ as follows
\begin{equation*} 
\tau_0 = 1,~~~\tau_k := \frac{\tau_{k-1}}{2}\left(\sqrt{\tau_{k-1}^2 + 4} - \tau_{k-1}\right),~\text{and}~\rho_{k} := \frac{\rho_{k-1}}{1-\tau_k} = \frac{\rho_{k-1}\tau_{k-1}^2}{\tau_k^2}.
\end{equation*}
This update leads to $1 - \tau_k = \frac{\tau_k^2}{\tau_{k-1}^2}$.
By induction and  $\tau_0 = 1$, we can show that $\frac{1}{k+1} \leq \tau_k \leq \frac{2}{k+2}$.
Moreover, we also have $\omega_k =  \prod_{i=1}^k(1-\tau_i) =  \prod_{i=1}^k\frac{\tau_i^2}{\tau_{i-1}^2} = \frac{\tau_k^2}{\tau_0^2} = \tau_k^2$.
Since $\rho_k = \frac{\rho_{k-1}}{1-\tau_k}$, by induction, we obtain $\rho_k = \frac{\rho_0}{\tau_k^2}$.

Next, we find the condition on $\rho_0$ such that the first condition of \eqref{eq:key_cond4} holds.
Indeed, using $1-\tau_k = \frac{\tau_k^2}{\tau_{k-1}^2}$ and $\rho_k = \frac{\rho_0}{\tau_k^2}$, this condition is equivalent to
\begin{equation*}
\rho_0\norms{\Bb}^2\left(\frac{\tau_{k-1}}{\tau_k}\right) \leq \mu_g.
\end{equation*}
Clearly, since $1 \leq \frac{\tau_{k-1}}{\tau_k} \leq 2$,  if $2\rho_0\norms{\Bb}^2 \leq \mu_g$, then $\rho_0\norms{\Bb}^2\left(\frac{\tau_{k-1}}{\tau_k}\right) \leq \mu_g$ holds.
The condition $2\rho_0\norms{\Bb}^2 \leq \mu_g$ is equivalent to $\rho_0 \leq \frac{\mu_g}{2\norms{\Bb}^2}$.

The second condition of \eqref{eq:key_cond4} automatically holds due to the update rule of $\tau_k$.
In this case, since $\tau_0 = 1$, $\tilde{\xb}^0 = \xb^0$, and $\tilde{\yb}^0 = \yb^0$, \eqref{eq:proof3_est1} leads to
\begin{equation*} 
S_{\rho_k}(\zb^{k+1}) \leq \tfrac{\tau_k^2}{2}\left[\gamma_0\norms{\xb^0 - \xb^{\star}}^2 + \rho_0\norms{\Bb}^2\norms{\yb^0 - \yb^{\star}}^2\right].
\end{equation*}
Using this, $\rho_k = \frac{\rho_0}{\tau_k^2}$, and $\frac{1}{k+1} \leq \tau_k \leq \frac{2}{k+2}$ into Lemma~\ref{le:approx_opt_cond}, we obtain \eqref{eq:convergence2}.
\Eproof
\end{proof}

\beforesec
\section{Variants and extensions}\label{sec:variants_extensions}
\aftersec
Algorithms \ref{alg:A1} and \ref{alg:A2} can be customized to obtain different variants. 
Let us provide some examples on how to customize these algorithms to handle instances of \eqref{eq:constr_cvx}.

\beforesubsec
\subsection{Application to composite convex minimization}\label{subsec:composite_cvx1}
\aftersubsec
Let us consider the following composite convex problem
\begin{equation}\label{eq:composite_prob}
P^{\star} := \min_{\yb\in\R^p}\Big\{ P(\yb) := f(\yb) + g(\yb) \Big\},
\end{equation}
where $f : \R^p\to\Rext$ and $g :\R^p\to\Rext$ are proper, closed, and convex.
Let us introduce $\xb = \yb$ and write \eqref{eq:composite_prob} as \eqref{eq:constr_cvx} with $F(\zb) := f(\xb) + g(\yb)$ and $\xb - \yb = 0$.

Now we apply Algorithm~\ref{alg:A1} to solve the resulting problem, and obtain
\begin{equation*}
\xb^{k+1} :=  \prox_{f/\rho_k}(\hat{\yb}^k)~~\text{and}~~\yb^{k+1} :=  \kprox{g/\rho_k}{\xb^{k+1}}.
\end{equation*}
We can combine these two steps to obtain the following scheme to solve \eqref{eq:composite_prob}:
\begin{equation}\label{eq:DR_method2}
\left\{\begin{array}{ll}
\yb^{k+1} &:=  \kprox{g/\rho_k}{\prox_{f/\rho_k}(\hat{\yb}^k)} ~~\text{with}~\rho_k :=  \rho_0(k+1),\vspace{1ex}\\
\hat{\yb}^{k+1} &:=  \yb^{k+1} + \frac{k}{k+2}\big(\yb^{k+1} - \yb^k\big).  
\end{array}\right.
\end{equation}
Here, $\rho_0 > 0$ is an initial value.
This scheme was studied in \cite{TranDinh2015d}.

Similarly, when $g$ is $\mu_g$-strongly convex with $\mu_g > 0$, we can apply Algorithm~\ref{alg:A2} to solve \eqref{eq:composite_prob}.
Let us consider \textbf{Option 1}. Then, after eliminating $\sets{\xb^k}$, we obtain
\begin{equation}\label{eq:DR_method3}
\left\{\begin{array}{ll}
\hat{\yb}^k &:=  (1-\tau_k)\yb^k + \tau_k\tilde{\yb}^k, \vspace{1ex}\\
\tilde{\yb}^{k+1} &:=  \prox_{g/(\tau_k\rho_k)}\Big(\tfrac{1}{\tau_k}\prox_{f/\rho_k}(\hat{\yb}^k) - \tfrac{(1 - \tau_k)}{\tau_k}\yb^k \Big), \vspace{1ex}\\
\yb^{k+1} &:=  (1-\tau_k)\yb^k + \tau_k\tilde{\yb}^{k+1}.
\end{array}\right.
\end{equation}
Here, $\rho_k := \frac{\rho_0}{\tau_k^2}$ for $\rho_0 \in (0, \tfrac{\mu_g}{2}]$, and $\tau_0 := 1$ and $\tau_{k+1} := 0.5\tau_k\big( (\tau_k^2 + 4)^{1/2} - \tau_k\big)$.
The following corollary provides the convergence rate of these two variants, whose proof can be found in Appendix~\ref{apdx:co:com_cvx_convergence}.

\begin{corollary}\label{co:com_cvx_convergence}
Assume that $f$ is Lipschitz continuous with the Lipschitz constant $L_f \in [0,+\infty)$, i.e., $\vert f(\yb) - f(\hat{\yb})\vert \leq L_f\norms{\yb - \hat{\yb}}$ for all $\yb, \hat{\yb}\in\dom{f}$.
Let $\sets{\yb^k}$ be generated by \eqref{eq:DR_method2} to solve \eqref{eq:composite_prob}. 
Then, we have
\begin{equation}\label{eq:co1_conv_rate1} 
P(\yb^k) - P^{\star} \leq \frac{\rho_0^2\norms{\yb^0 - \yb^{\star}}^2 + 4L_f^2+ 2L_f\rho_0\norms{\yb^0-\yb^{\star}}}{2\rho_0k}.
\end{equation}
If $f$ is $L_f$-Lipschitz continuous on $\dom{f}$ and $g$ is $\mu_g$-strongly convex, then  $\sets{\yb^k}$ generated by \eqref{eq:DR_method3} to solve \eqref{eq:composite_prob} satisfies
\begin{equation}\label{eq:co1_conv_rate2} 
P(\yb^k) - P^{\star} \leq \frac{2\rho_0\norms{\yb^0 - \yb^{\star}}^2}{(k+1)^2} + \frac{8\left(L_f^2 +  L_f\rho_0\norms{\yb^0 - \yb^{\star}}\right)}{\rho_0(k + 1)^2}.
\end{equation}
\end{corollary}
Note that we can use \textbf{Option 2} to replace the averaging step on $\yb^k$ by $\prox_g$.
In this case, we still have the same guarantee as in \eqref{eq:co1_conv_rate2}, but the scheme \eqref{eq:DR_method3} is slightly changed.
We can also eliminate $\yb^k$ in Algorithm~\ref{alg:A2} instead of $\xb^k$. 
In this case, the convergence guarantee is on $\sets{\xb^k}$ and it requires $g$ to be $L_g$-Lipschitz continuous instead of $f$.
This convergence rate is non-ergodic. 
The proof is rather similar and we skip its details.

\beforesubsec
\subsection{Application to composite convex minimization with linear operator}\label{subsec:composite_cvx2}
\aftersubsec
We tackle a more general form of \eqref{eq:composite_prob} by considering the following problem:
\begin{equation}\label{eq:composite_prob2}
P^{\star} := \min_{\yb\in\R^p}\Big\{ P(\yb) := f(\Bb\yb) + g(\yb) \Big\},
\end{equation}
where $f : \R^p\to\Rext$ and $g :\R^n\to\Rext$ are proper, closed, and convex, and $\Bb\in\R^{n\times p}$ is a linear operator.

Using the same trick by introducing $\xb = \Bb\yb$, we obtain $F(\zb) = f(\xb) + g(\yb)$ and a linear constraint $\xb - \Bb\yb = 0$.
Now we apply Algorithm~\ref{alg:A1} to solve the resulting problem, and obtain
\begin{equation*}
\xb^{k+1} :=  \kprox{f/\rho_k}{\Bb\hat{\yb}^k}~~\text{and}~~\yb^{k+1} :=  \prox_{g/\rho_k\norms{\Bb}^2}\Big(\hat{\yb}^k - \tfrac{1}{\norms{\Bb}^2}\Bb^{\top}(\Bb\hat{\yb}^k - \xb^{k+1})\Big).
\end{equation*}
Plugging the first expression into the second one, and adding $\hat{y}$-step, we get
\begin{equation}\label{eq:algA1_v2}
\left\{\begin{array}{ll}
\yb^{k+1} &:=  \prox_{g/(\norms{\Bb}^2\rho_k)}\Big( \big(\Id-\frac{1}{\norms{\Bb}^2}\Bb^{\top}\Bb \big)\hat{\yb}^k + \tfrac{1}{\norms{\Bb}^2}\Bb^{\top}\kprox{f/\rho_k}{\Bb\hat{\yb}^k} \Big), \vspace{1ex}\\
\hat{\yb}^{k+1} &:= \yb^{k+1} + \frac{k}{k+2}(\yb^{k+1} - \yb^k).
\end{array}\right.
\end{equation}
Similarly, we can also customize Algorithm~\ref{alg:A2} to solve \eqref{eq:composite_prob2} when $g$ is $\mu_g$-strongly convex as
\begin{equation}\label{eq:algA2_v2}
\left\{\begin{array}{ll}
\tilde{\yb}^{k+1} &:= \prox_{g/(\tau_k\rho_k\norms{\Bb}^2)}\Big( \tilde{\yb}^k - \frac{1}{\tau_k\norms{\Bb}^2}\Bb^{\top}\left(\Bb\hat{\yb}^k - \kprox{f/\rho_k}{\Bb\hat{\yb}^k}\right) \Big), \vspace{1ex}\\
\yb^{k+1} &:= (1-\tau_k)\yb^k + \tau_k\tilde{\yb}^{k+1}, \vspace{1ex}\\
\hat{\yb}^{k+1} &:= \yb^{k+1} + \frac{\tau_{k+1}(1-\tau_k)}{\tau_k}(\yb^{k+1} - \yb^k).
\end{array}\right.
\end{equation}
The convergence of \eqref{eq:algA1_v2} and \eqref{eq:algA2_v2} can be proved as in Corollary \ref{co:com_cvx_convergence} under the Lipschitz continuity of $f$. We omit the details here.

\beforesubsec
\subsection{A primal-dual interpretation of Algorithm \ref{alg:A1} and Algorithm \ref{alg:A2}}\label{subsec:CP_method}
\aftersubsec
We show that Algorithms~\ref{alg:A1} and \ref{alg:A2} can be interpreted as  primal-dual methods for solving \eqref{eq:composite_prob2}.
We consider the $x$-subproblem \eqref{eq:x_cvx_subprob} with $\gamma = 0$ as
\begin{equation}\label{eq:x_subprob3}
\begin{array}{ll}
\xb^{k+1} &:= \kprox{f/\rho_k}{\Bb\hat{\yb}^k}  \overset{\tiny\eqref{eq:Moreau_identity}}{=}  \Bb\hat{\yb}^k - \frac{1}{\rho_k}\kprox{ \rho_k f^{\ast}}{ \rho_k\Bb\hat{\yb}^k}.
\end{array}
\end{equation}
Let $\bar{\xb}^{k+1} := \kprox{ \rho_k f^{\ast}}{\rho_k\Bb\hat{\yb}^k}$.
Then, by using \eqref{eq:x_subprob3} and a notation $\dot{\xb} := \boldsymbol{0}^{p_1}$, we can rewrite Algorithm~\ref{alg:A1} for solving \eqref{eq:composite_prob2} as
\begin{equation}\label{eq:PD_variant_for_A1} 
\left\{\begin{array}{ll}
\bar{\xb}^{k+1} &:= \kprox{\rho_k f^{\ast}}{\dot{\xb} + \rho_k\Bb\hat{\yb}^k}, \vspace{1ex}\\
\yb^{k+1} &:= \kprox{g/\rho_k\norms{\Bb}^2}{\hat{\yb}^k - \tfrac{1}{\rho_k\norms{\Bb}^2}\Bb^{\top}\bar{\xb}^{k+1}}, \vspace{1ex}\\
\hat{\yb}^{k+1} &:=  \yb^{k+1} + \frac{k}{k+2}(\yb^{k+1} - \yb^k).
\end{array}\right.
\end{equation}
This scheme can be considered as a new primal-dual method for solving \eqref{eq:composite_prob2}, and it is different from existing primal-dual methods in the literature.

Similarly, we can also interpret Algorithm~\ref{alg:A2} with \textbf{Option 1} as a primal-dual variant.
Using the same idea as above, we arrive at
\begin{equation}\label{eq:PD_variant_for_A2} 
\left\{\begin{array}{ll}
\bar{\xb}^{k+1} &:= \kprox{\rho_k f^{\ast}}{\dot{\xb} + \rho_k\Bb\hat{\yb}^k}, \vspace{1ex}\\
\tilde{\yb}^{k+1} &:= \kprox{g/(\tau_k\rho_k\norms{\Bb}^2)}{\tilde{\yb}^k - \tfrac{1}{\tau_k\rho_k\norms{\Bb}^2}\Bb^{\top}\bar{\xb}^{k+1}}, \vspace{1ex}\\
\yb^{k+1} &:=  (1-\tau_k)\yb^k + \tau_k\tilde{\yb}^{k+1},\vspace{1ex}\\
\hat{\yb}^{k+1} &:=  \yb^{k+1} + \frac{\tau_{k+1}(1-\tau_k)}{\tau_k}(\yb^{k+1} - \yb^k).
\end{array}\right.
\end{equation}
The convergence guarantee of both schemes \eqref{eq:PD_variant_for_A1}  and \eqref{eq:PD_variant_for_A2} can be proved as in Corollary~\ref{co:com_cvx_convergence} under the $L_f$-Lipschitz continuity assumption of $f$.
We again omit the detailed analysis here.

\beforesubsec
\subsection{Extension to the sum of three objective functions}\label{subsec:three_objs}
\aftersubsec
Let us consider the following constrained convex optimization problem:
\begin{equation}\label{eq:constr_cvx2}
F^{\star} := \min_{\zb := [\xb,\yb]}\Big\{ F(\zb) := f(\xb) + g(\yb) + h(\yb) \mid \Ab\xb + \Bb\yb - \cb \in \Kc \Big\},
\end{equation}
where $f$, $g$, $\Ab$, $\Bb$, $\cb$ and $\Kc$ are defined as in \eqref{eq:constr_cvx}, and $h : \R^{p_2}\to\R$ is convex and Lipschitz gradient continuous with the Lipschitz constant $L_h > 0$.
In this case, we can modify the $\yb$-subproblem in Algorithm~\ref{alg:A1} as
\begin{equation}\label{eq:y_prob2}
\yb^{k+1} := \prox_{g/\hat{\beta}_k}\left( \hat{\yb}^k - \tfrac{1}{\hat{\beta}_k}\big(\nabla{h}(\hat{\yb}^k) + \rho_k\nabla_y{\psi}(\xb^{k+1}, \hat{\yb}^k)\big) \right),
\end{equation}
where $\hat{\beta}_k := \norms{\Bb}^2\rho_k + L_h$.
Other steps remain as in Algorithm~\ref{alg:A1}.

When either $g$ is $\mu_g$-strongly convex or $h$ is $\mu_h$-strongly convex such that $\mu_g + \mu_h > 0$, we can applied Algorithm~\ref{alg:A2} to solve \eqref{eq:constr_cvx2}.
In this case, the $\yb$-subproblem in Algorithm~\ref{alg:A2} becomes
\begin{equation}\label{eq:y_prob2b}
\tilde{\yb}^{k+1} := \prox_{g/(\tau_k\hat{\beta}_k)}\left( \tilde{\yb}^k - \tfrac{1}{\tau_k\hat{\beta}_k}\big(\nabla{h}(\bar{\yb}^k) + \rho_k\nabla_y{\psi}(\xb^{k+1}, \hat{\yb}^k)\big) \right).
\end{equation}
Here, $\bar{\yb}^k$ is chosen such that $\bar{\yb}^k := \tilde{\yb}^k$ if $\mu_g + 2\mu_h - L_h > 0$, or $\bar{\yb}^k := \hat{\yb}^k$ if $\mu_g > 0$.
In addition, if we use \textbf{Option 2}, then we compute $\yb^{k+1}$ as
\begin{equation}\label{eq:choice2}
\yb^{k+1} :=  \prox_{g/\breve{\beta}_k}\left(\hat{\yb}^k - \tfrac{1}{\breve{\beta}_k}\big(\nabla{h}(\hat{\yb}^k) + \rho_k\nabla_y{\psi}(\xb^{k+1}, \hat{\yb}^k)\big)\right),
\end{equation}
where $\breve{\beta}_k := \rho_k\norms{B}^2 + L_h$.
Other steps remain the same as in Algorithm~\ref{alg:A2}.

The following theorem shows the convergence of these variants, whose proof can be found in Appendix~\ref{apdx:co:three_objs}.
\begin{theorem}\label{co:three_objs}
Let $\sets{(\xb^k,\yb^k)}$ be the sequence generated by Algorithm~\ref{alg:A1} to solve \eqref{eq:constr_cvx2} using \eqref{eq:y_prob2} for $\yb^k$ and $\hat{\beta}_k := \rho_k\norms{\Bb}^2 + L_h$.
Then the bound \eqref{eq:convergence1} in Theorem~\ref{th:convergence1} still holds with $R_p^2 := \gamma_0\norms{\xb^0 - \xb^{\star}}^2 + (L_h + \rho_0\norms{\Bb}^2)\norms{\yb^0 - \yb^{\star}}^2$.

Assume that either $g$ is $\mu_g$-strongly convex or $h$ is $\mu_h$-strongly convex such that $\mu_g + \mu_h > 0$.
Let $\sets{(\xb^k,\yb^k)}$ be the sequence generated by Algorithm~\ref{alg:A2} to solve \eqref{eq:constr_cvx2} using \eqref{eq:y_prob2b} for $\yb^k$ such that:
\begin{itemize}
\item[$\mathrm{(i)}$] If $\mu_g > 0$, then we choose $\bar{\yb}^k := \hat{\yb}^k$ and $0 < \rho_0 \leq \frac{\mu_g}{2\norms{\Bb}^2}$, and update $\hat{\beta}_k := \rho_k\norms{\Bb}^2 + L_h$.
\item[$\mathrm{(ii)}$] If $L_h < 2\mu_h$, then we choose $\bar{\yb}^k := \tilde{\yb}^k$ and $0 < \rho_0 \leq \frac{\mu_g + 2\mu_h - L_h}{2\norms{\Bb}^2}$, and update $\hat{\beta}_k := \rho_k\norms{\Bb}^2 + \tfrac{L_h}{\tau_k}$.
\end{itemize}
Then the bound \eqref{eq:convergence2} in Theorem~\ref{th:convergence2} still holds with $R_p^2 := \gamma_0\norms{\xb^0 - \xb^{\star}}^2 + (L_h + \rho_0\norms{\Bb}^2)\norms{\yb^0 - \yb^{\star}}^2$.
\end{theorem}
Note that we can extend  Algorithms~\ref{alg:A1} and \ref{alg:A2} to handle the case where $f(x)$ is replaced by $f(\xb) + h(\xb)$, where $h$ is convex and $L_h$-Lipschitz gradient continuous.

\beforesubsec
\subsection{Application to conic programming}\label{subsubsec:conic}
\aftersubsec
The setting \eqref{eq:constr_cvx} is sufficiently general to cope with many classes of convex problems.
As a specific example, we illustrate how to use Algorithm~\ref{alg:A1} to solve the following conic program:
\begin{equation}\label{eq:conic_prog}
\min_{\xb, \yb}\Big\{ \iprods{\qb, \yb} ~\mid~ \Bc(\yb) + \xb = \cb, ~\xb \in\Cc \Big\},
\end{equation}
where $\qb$ and $\cb$ are given vectors, $\Bc$ is a bounded linear operator, and $\Cc$ is a nonempty, closed, pointed, and convex cone.
This formulation is often  referred to as a dual problem in, e.g., linear, second-order cone, or semidefinite programming.
Clearly, \eqref{eq:conic_prog} can be cast into \eqref{eq:constr_cvx} with $f(\xb) := \delta_{\Cc}(\xb)$, the indicator function of the cone $\Cc$, $g(\yb) := \iprods{\qb, \yb}$, $\mathcal{A} = \Id$, the identity operator, and $\Kc = \set{\boldsymbol{0}}$.
Therefore, the core steps of Algorithm~\ref{alg:A1} for solving \eqref{eq:conic_prog} become
\begin{equation*}
\left\{\begin{array}{ll}
\xb^{k+1} &:= \kproj{\Cc}{\cb - \Bc(\hat{\yb}^k)}, \vspace{1ex}\\
\yb^{k+1} &:= \hat{\yb}^k - \frac{1}{\rho_k\norms{\Bc}^2}\left(\qb + \rho_k\Bc^{\ast}\left(\xb^{k+1} + \Bc(\hat{\yb}^k) - \cb\right)\right), \vspace{1ex}\\
(\hat{\xb}_{k+1}, \hat{\yb}_{k+1}) &:= (\xb^{k+1}, \yb^{k+1}) + \tfrac{k}{k+2}(\xb^{k+1} - \xb^k, \yb^{k+1} - \yb^k).
\end{array}\right.
\end{equation*}
Here, $\Bc^{\ast}$ is the adjoint operator of $\Bc$, and the parameter $\rho_k$ is updated as in Algorithm~\ref{alg:A1}.
This variant is rather simple, it requires one operation $\Bc(\cdot)$, one adjoint operation $\Bc^{\ast}(\cdot)$, and one projection onto the cone $\Cc$.

\beforesubsec
\subsection{Shifting the initial dual variable and restarting}\label{subsubsec:restarting}
\aftersubsec
As we can see from  \eqref{eq:approx_opt_cond} of Lemma~\ref{le:approx_opt_cond} that the bound on $\mathrm{dist}_{\Kc}(\Ab\xb + \Bb\yb - \cb)$ depends on $\norms{\lbd^{\star}}$ instead of $\norms{\lbd^{\star} - \lbd^0}$ from an initial dual variable $\lbd^0$.
We use the idea of ``restarting the prox-center point'' from \cite{Tran-Dinh2014a,TranDinh2014b} to adaptively update $\lambda^0$.
This idea has been recently used in  \cite{van2017smoothing,TranDinh2015b} as a restarting strategy and it has significantly improved the performance of the algorithms.

The main idea is to replace $\varphi$ defined by \eqref{eq:phi_minmax} by 
\begin{equation*}
\varphi_{\rho}(\ub;\lambda^0) := \max_{\lbd\in\R^n}\min_{\rb\in\Kc}\set{\iprods{\ub - \rb,\lbd} - \tfrac{\rho}{2}\norms{\lbd - \lbd^0}^2} = \tfrac{\rho}{2}\kdist{\Kc}{u + \tfrac{1}{\rho}\lbd}.
\end{equation*}
and redefine $\psi(\cdot,\cdot)$ in \eqref{eq:Phi_func} by $\psi_{\rho}(\xb, \yb;\lbd^0) := \varphi_{\rho}(\Ab\xb + \Bb\yb - \cb;\lbd^0 ) = \frac{\rho}{2}\kdist{\Kc}{\Ab\xb + \Bb\yb - \cb + \tfrac{1}{\rho}\lbd^0}^2$.
Then, the main steps of Algorithm~\ref{alg:A1} or Algorithm~\ref{alg:A2} become
\begin{equation}\label{eq:restart}
{\!\!\!\!\!}\begin{array}{lll}
&~~\xb^{k\!+\!1} {\!\!\!}&~ \in \argmin_{\xb}\set{ f(\xb) + \psi_{\rho_k}(\xb, \hat{\yb}^k; \lbd^0) }, \vspace{1ex}\\
&~~\yb^{k\!+\!1} {\!\!\!}&\ := \argmin_{\yb}\set{ g(\yb) + \iprods{\nabla_y{\psi_{\rho_k}}(\xb^{k\!+\!1}, \hat{\yb}^k;  \lbd^0), \yb \!-\! \hat{\yb}^k} + \tfrac{\rho_k\norms{\Bb}^2}{2}\norms{\yb - \hat{\yb}^k}^2 }, \vspace{1ex}\\
\text{or}&~\tilde{\yb}^{k\!+\!1} {\!\!}&{\!\!\!} := \argmin_{\yb}\set{ g(\yb) + \iprods{\nabla_y{\psi_{\rho_k}}(\xb^{k\!+\!1}, \hat{\yb}^k; \lbd^0), \yb - \hat{\yb}^k} + \tfrac{\rho_k\tau_k\norms{\Bb}^2}{2}\norms{\yb - \tilde{\yb}^k}^2 }.
\end{array}{\!\!\!\!\!}
\end{equation}
Our strategy is to frequently update $\lbd^0$ and restart the algorithms as follows.
We perform $k_s$ steps (e.g., $k_s = 100$) starting from $k :=  0$ to $k :=  k_s-1$, and restart the variables by resetting:
\vspace{-0.5ex}
\begin{equation*}
\rho_{k_s} :=  \rho_0, ~~~ \tau_{k_s} :=  1, ~~~ \hat{\yb}^{k_s} :=  \yb^{k_s}, ~\text{and}~~\lbd^0 :=  \lbd^0 + \nabla{\varphi_{\rho_{k_s}}}(\Ab\xb^{k_s+1} + \Bb\hat{\yb}^{k_s} - \cb;\lbd^0),
\vspace{-0.5ex}
\end{equation*}
where $\nabla{\varphi_{\rho}}$ is given by \eqref{eq:grad_varphi}.
Since proving the convergence of this variant is out of scope of this paper, we refer to our forthcoming work \cite{TranDinh2017f} for the full theory of restarting.

\beforesec
\section{Numerical experiments}\label{sec:num_examples}
\aftersec
In the following numerical examples, we focus on the following problem template:
\vspace{-0.5ex}
\begin{equation}\label{eq:unconstr_cvx}
F^{\star} := \min_{\yb\in\R^p}\Big\{ F(\yb) :=  f(\Bb\yb) + g(\yb) + h(\yb) \Big\},
\vspace{-0.5ex}
\end{equation}
where $f$ and $g$ are convex and possibly nonsmooth, $h$ is  convex and $L_h$-Lipschitz gradient continuous, and $\Bb$ is a linear operator.
If we introduce $\xb := \Bb\yb$ and let $h=0$, then the objective of \eqref{eq:unconstr_cvx} becomes $F(\zb) := f(\xb) + g(\yb)$ with an additional constraint $-\xb + \Bb\yb = 0$.
Hence, \eqref{eq:unconstr_cvx} can be converted into \eqref{eq:constr_cvx}.
Otherwise, it becomes \eqref{eq:constr_cvx2}.

\noindent We implement $9$ algorithms to solve \eqref{eq:unconstr_cvx} as follows:
\begin{itemize}
\item Algorithm \ref{alg:A1}, denoted by \texttt{PAPA}, and its restarting variant, called \texttt{PAPA-rs}.
\item Algorithm \ref{alg:A2}, denoted by \texttt{scvx-PAPA} and its restarting variant, called \texttt{scvx-PAPA-rs}.
\item Algorithm 1 in \cite{TranDinh2015b}, \texttt{ASGARD}, and its restarting variant, denoted by \texttt{ASGARD-rs}.
\item The Chambolle-Pock algorithm in \cite{Chambolle2011} and Vu-Condat's method in \cite{Condat2013,vu2013splitting}.
\item The accelerated proximal gradient method, denoted by \texttt{AcProxGrad}, in \cite{Beck2009,Nesterov2007}.
\end{itemize}
These algorithms are implemented in Matlab (R2014b), running on a MacBook Pro. Laptop with 2.7 GHz Intel Core i5, and 16GB memory.
Note that the per-iteration complexity of Algorithm \ref{alg:A1}, Algorithm \ref{alg:A2}, \texttt{ASGARD}, Chambolle-Pock's algorithm, and Vu-Condat's algorithm is essentially the same.
For a thorough comparison to between \texttt{ASGARD} and other methods, including ADMM, we refer to \cite{TranDinh2015b}.

For configuration of Algorithms~\ref{alg:A1} and \ref{alg:A2}, we choose $\rho_0 := \tfrac{1}{\norm{\Bb}}$ in Algorithm~\ref{alg:A1} and its variants.
We choose $\rho_0 := \frac{\mu_g}{2\norm{\Bb}^2}$ in Algorithm~\ref{alg:A2} and its variants. 
However, if $\mu_g$ is unknown (e.g., problem may not be strongly convex, but quasi-strongly convex), we examine and choose $\mu_g := 0.1$.
Since we consider the case $\Ab = \Id$, we set $\gamma_0 := 0$ in all variants of \texttt{PAPA}.
For restarting variants, we restart \texttt{PAPA} after each $50$ iterations and \texttt{scvx-PAPA} after each $100$ iterations as described in Subsection \ref{subsubsec:restarting}.
For \texttt{ASGARD}, we use the same setting as in  \cite{TranDinh2015b}, and for Chambolle-Pock's and Vu-Condat's algorithm, we choose the parameters as suggested in \cite{Chambolle2011,Condat2013,vu2013splitting} for both the strongly and nonstrongly convex cases.
We also restart \texttt{ASGARD} after every each $50$ iterations.
Our Matlab code is available online at \url{https://github.com/quoctd/PAPA-1.0}.

\beforesubsec
\subsection{Dense convex quadratic programs}
\aftersubsec
We consider the following convex quadratic programming problem:
\vspace{-0.5ex}
\begin{equation}\label{eq:qp_exam}
g^{\star} := \min_{\yb\in\R^{p_2}}\set{ g(\yb) := \tfrac{1}{2}\yb^{\top}\Qb\yb + \qb^{\top}\yb \mid \ab \leq \Bb\yb \leq \bb }, 
\vspace{-0.5ex}
\end{equation}
where  $\Qb\in\R^{p_2\times p_2}$ is a symmetric positive [semi]definite matrix, $\qb\in\R^{p_2}$,  $\Bb\in\R^{n\times p_2}$ and $\ab,\bb\in\R^n$ such that $\ab \leq \bb$. 
We assume that both $\Qb$ and $\Bb$ are dense.

This problem can be reformulated into \eqref{eq:constr_cvx} by introducing a new variable $\xb := \Bb\yb$ to form the linear constraint $\xb - \Bb\yb = 0$ and an additional objective term $f(\xb) := \delta_{[\ab,\bb]}(\xb)$.
In this case, we have $\Kc = \set{ \boldsymbol{0} }$.

The main step of both Algorithms~\ref{alg:A1} and \ref{alg:A2} is to solve two subproblems at Step~\ref{eq:alter_scheme}.
For \eqref{eq:qp_exam}, these two problems can be solved explicitly as
\begin{equation*}
\left\{\begin{array}{ll}
\xb^{k+1} &: =  \proj_{[\ab,\bb]}\big(\Bb\hat{\yb}^k\big), \vspace{1ex}\\
\yb^{k+1} &: =  (\rho_k\norms{\Bb}^2\Id + \Qb)^{-1}\Big(\rho_k\norms{\Bb}^2\hat{\yb}^k - \rho_k\Bb^{\top}(\Bb\hat{\yb}^k - \xb^{k+1}) - \qb\Big).
\end{array}\right.
\end{equation*}
For Algorithm~\ref{alg:A2}, we change from $\yb^{k+1}$ to $\tilde{\yb}^{k+1}$, from $\hat{\yb}^k$ to $\tilde{\yb}^k$, and from $\rho_k\norms{\Bb}^2$ to $\tau_k\rho_k\norms{\Bb}^2$ in the second line.

Note that we can write \eqref{eq:qp_exam} into the following form 
\begin{equation*}
g^{\star} := \min_{\yb}\set{ G(\yb) := \tfrac{1}{2}\yb^{\top}\Qb\yb + \qb^{\top}\yb + f(\Bb\yb)},
\end{equation*}
where $f(\xb) := \delta_{[\ab,\bb]}(\xb)$ is the indicator function of the box $[\ab, \bb]$.
Hence, we can apply the Chambolle-Pock primal-dual algorithm \cite{Chambolle2011} to solve \eqref{eq:qp_exam}.

We test the first $7$ algorithms mentioned above on some synthetic data generated as follows.
We randomly generate $R \in\R^{p_2\times m}$, $q\in\R^{p_2}$, and $B\in\R^{n\times p_2}$ using the standard Gaussian distribution, where $m = \lfloor p_2/2\rfloor + 1$.
To avoid large magnitudes, we normalize $R$ by $\frac{1}{\sqrt{m}}R$, and $B$ by $\frac{1}{\sqrt{n}}B$.
We then define $Q := RR^{\top} + \mu_g\Id$, where $\mu_g = 0$ for the nonstrongly convex case and $\mu_g = 1$ for the strongly convex case.
We generate a random vector $\yb^{\natural}$ using again the standard Gaussian distribution, and define $\ab := \Bb\yb^{\natural} - \texttt{rand}(n,1)$ and $\bb := \Bb\yb^{\natural} + \texttt{rand}(n,1)$ to make sure that the problem is feasible, where $\texttt{rand}(n,1)$ is a uniform random vector in $(0, 1)^n$.

Figure \ref{fig:scvxQp1000} shows the convergence of $7$ algorithms on a strongly convex instance of \eqref{eq:qp_exam}, where $p_2 = 2000$ and $n = 2000$.
The left-plot shows the convergence of the relative objective residual $\frac{\abs{g(\yb^k) - g^{\star}}}{\abs{g^{\star}}}$, where $g^{\star}$ is computed by CVX \cite{Grant2006} using Mosek with the best accuracy.
The right-plot reveals the relative feasibility violation $\frac{\norms{\max\set{\Bb\yb^k - \bb, 0}} + \norms{\min\set{\Bb\yb^k - \ab, 0}}}{\max\set{\norm{\ab}, \norm{\bb}}}$.

\begin{figure}[htp!]
\begin{center}
\vspace{-4ex}
\includegraphics[width=1\linewidth]{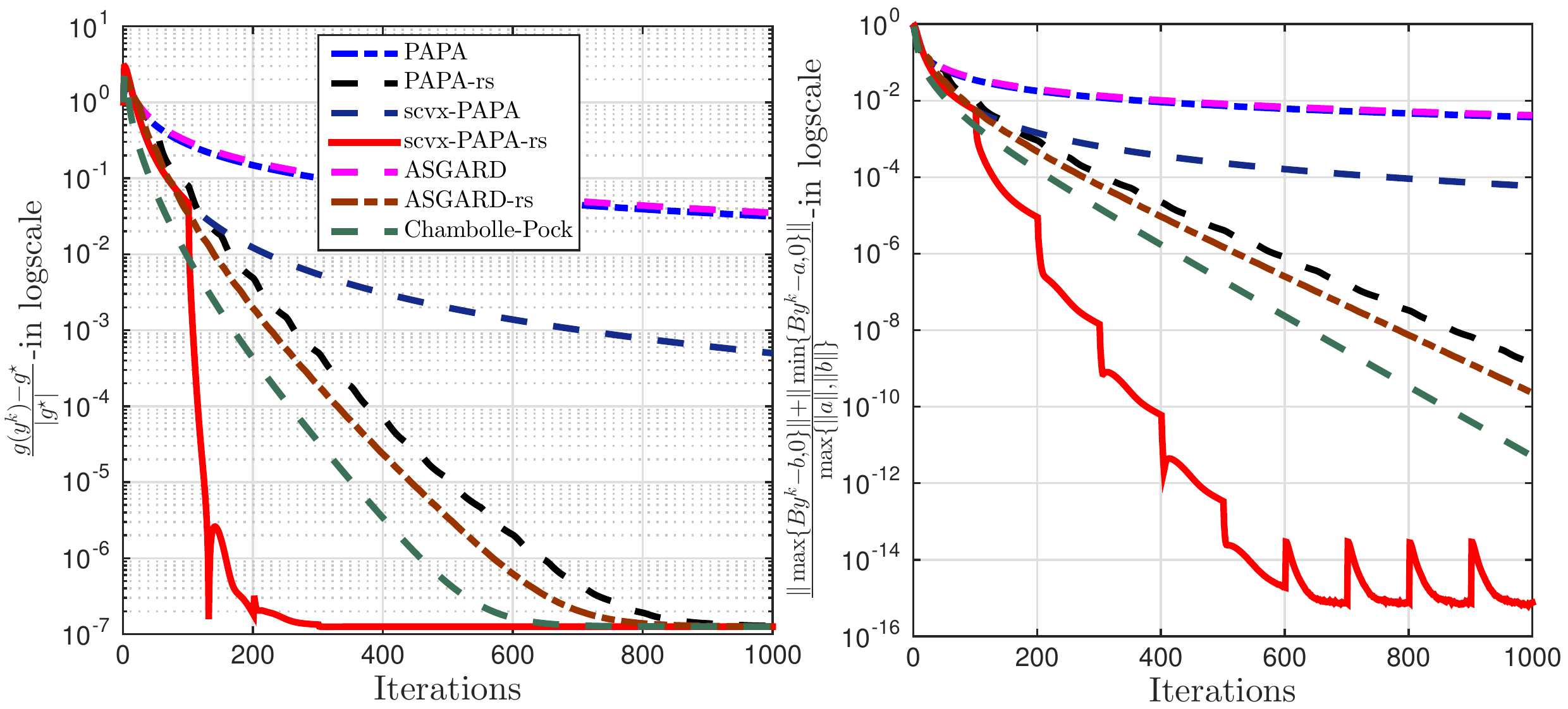}
\vspace{-3ex}
\caption{A comparison of $7$ algorithms on a \textbf{strongly convex} problem instance of \eqref{eq:qp_exam} after $1000$ iterations. 
The problem size is $(p_2 = 2000, n = 2000)$. 
Left: The relative objective residual, Right: The relative feasibility violation.
Due to Mosek's solution, the relative objective residual is saturated at a $10^{-7}$ accuracy, while the relative feasibility can reach a $10^{-15}$ accuracy. 
}\label{fig:scvxQp1000}
\vspace{-6ex}
\end{center}
\end{figure}

Since the problem is strongly convex, Algorithm~\ref{alg:A2} shows its $\BigO{\tfrac{1}{k^2}}$ convergence rate as predicted by the theory (Theorem~\ref{th:convergence2}), while Algorithm~\ref{alg:A1} and \texttt{ASGARD} still show their $\BigO{\tfrac{1}{k}}$ convergence rate.
The Chambolle-Pock algorithm using strong convexity works really well and exhibits beyond the theoretical $\BigO{\tfrac{1}{k^2}}$-rate.
The restarting variant of  Algorithm~\ref{alg:A2} completely outperforms the other methods, although the restarting variants of \texttt{PAPA} as well as \texttt{ASGARD} work well.

Next, we test these algorithms on a nonstrongly convex instance of \eqref{eq:qp_exam} by setting $\mu_g := 0$.
The convergence behavior of these algorithms is plotted in Figure~\ref{fig:nscvxQp1000}.

\begin{figure}[htp!]
\begin{center}
\vspace{-4ex}
\includegraphics[width=1\linewidth]{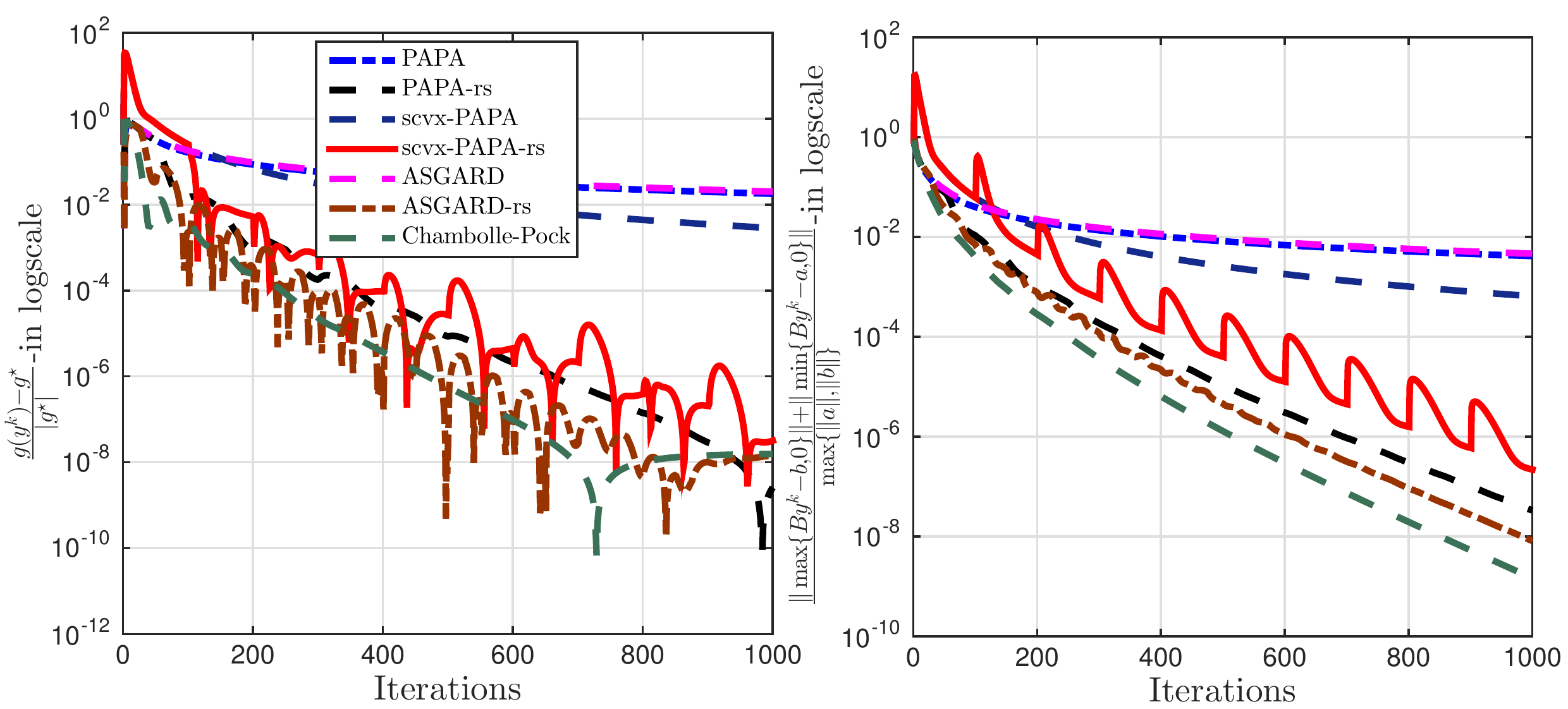}
\vspace{-3ex}
\caption{A comparison of $7$ algorithms on a \textbf{nonstrongly convex} problem instance of \eqref{eq:qp_exam} after $1000$ iterations. 
The problem size is $(p_2 = 2000, n = 2000)$. 
Left: The relative objective residual, Right: The relative feasibility violation.
}\label{fig:nscvxQp1000}
\vspace{-6ex}
\end{center}
\end{figure}

Since the problem is no longer strongly convex, Algorithm~\ref{alg:A2} does not guarantee its $\BigO{\tfrac{1}{k^2}}$-rate, but Algorithm~\ref{alg:A1} still has its $\BigO{\tfrac{1}{k}}$-rate.
The restarting variant of Algorithm~\ref{alg:A2} still improves its theoretical performance, but becomes worse than other restart variants and the Chambolle-Pock method.
In this particular instance, \texttt{ASGARD} with restarting still works well.

Finally, we verify the restarting variants on the strongly convex problem instance of \eqref{eq:qp_exam} by choosing different frequencies: $s = 50$ and $s= 100$.
The convergence result of this run is plotted in Figure~\ref{fig:scvxQp1000_restart}.
\begin{figure}[H]
\begin{center}
\vspace{-4ex}
\includegraphics[width=1.0\linewidth]{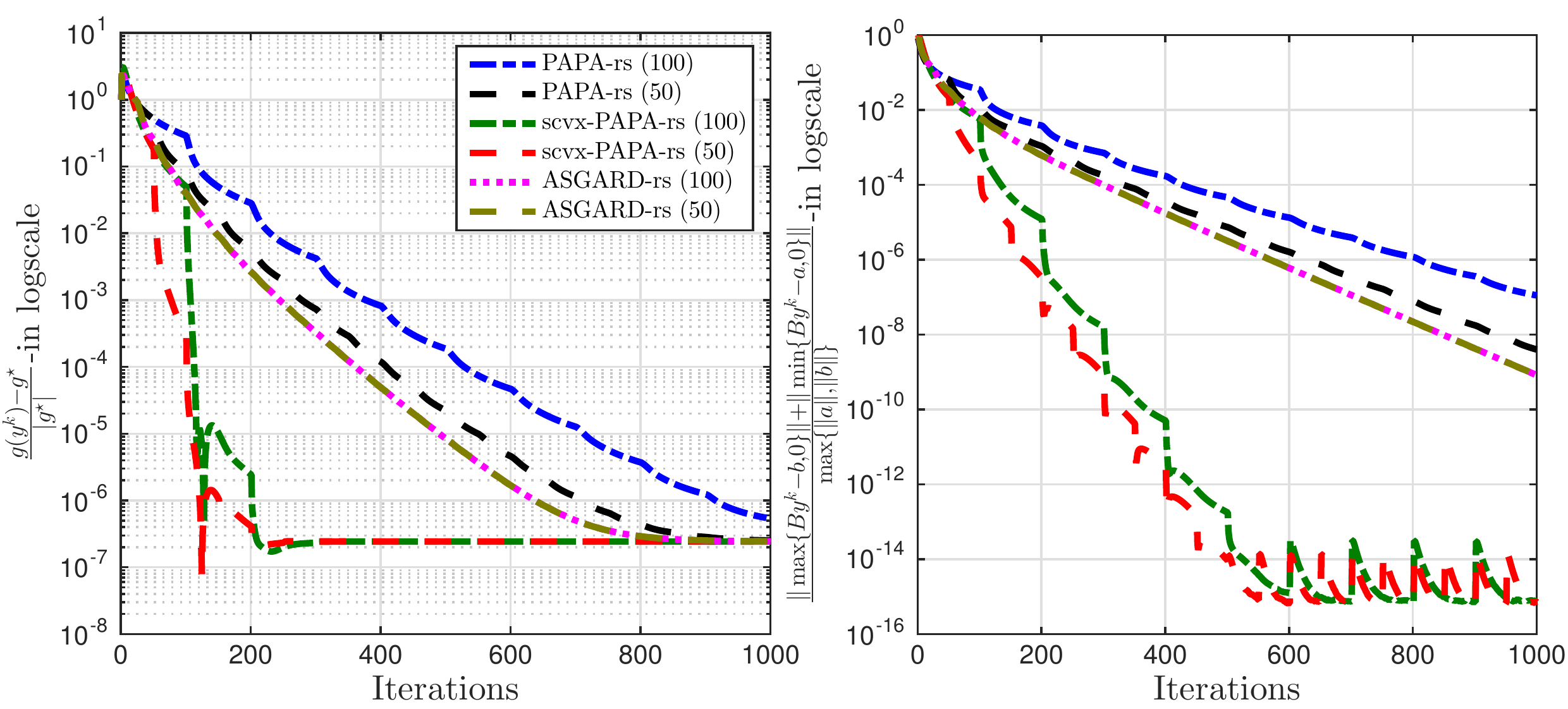}
\vspace{-3ex}
\caption{A comparison of  the restarting variants with two different frequencies on a \textbf{strongly convex} problem instance of \eqref{eq:qp_exam} after $1000$ iterations. 
The problem size is $(p_2 = 2000, n = 2000)$. 
Left: The relative objective residual, Right: The relative feasibility violation.
}\label{fig:scvxQp1000_restart}
\vspace{-6ex}
\end{center}
\end{figure}

Figure~\ref{fig:scvxQp1000_restart} shows that these two frequencies seem not significantly affecting the performance of the restarting algorithms.
For \texttt{PAPA}, $s=50$ slightly works better than $s = 100$.
However, we have observed that if we set the frequency $s$ too small, e.g., $s=10$, then the restarting variants are highly oscillated. 
If we set it too big, then it does not improve the performance and we need to run with a large number of iterations.
In \cite{TranDinh2017f}, we provide a full theory on how to adaptively choose the frequency to guarantee the convergence of the restarting \texttt{ASGARD} methods, which can also be applied to \texttt{PAPA}.

\beforesubsec
\subsection{The elastic-net problem with square-root loss}
\aftersubsec
In this example, we consider the common elastic-net LASSO problem studied in \cite{zou2005regularization} but with a square-root loss as follows:
\begin{equation}\label{eq:elastic_lasso_exam}
F^{\star} := \min_{\yb\in\R^{p_2}}\set{ F(\yb) := \norms{\Bb\yb - \cb}_2 + \tfrac{\kappa_1}{2}\norms{\yb}^2 + \kappa_2\norms{\yb}_1},
\end{equation}
where $\kappa_1 > 0$ and $\kappa_2 > 0$ are two regularization parameters.
Due to the nonsmoothness of the square-root loss $\norms{\Bb\yb - \cb}_2$, this problem is harder to solve than the standard elastic-net in \cite{zou2005regularization}, and algorithms such as FISTA \cite{Beck2009} are not applicable. 

By introducing $\xb := \Bb\yb - \cb$, we can reformulate \eqref{eq:elastic_lasso_exam} into \eqref{eq:constr_cvx} as
\begin{equation*}
F^{\star} := \min_{\zb := (\xb, \yb)}\set{ F(\zb) := \norms{\xb}_2 + \tfrac{\kappa_1}{2}\norms{\yb}^2 + \kappa_2\norms{\yb}_1 ~\mid~ -\xb + \Bb\yb  = \cb}.
\end{equation*}
Since $g(\yb) :=  \tfrac{\kappa_1}{2}\norms{\yb}^2 + \kappa_2\norms{\yb}_1$ is strongly convex, we can apply Algorithm~\ref{alg:A2} to solve it.
By choosing $\gamma_0 = 0$, the two subproblems at Step~\ref{step:x_cvx_subprob} of Algorithm~\ref{alg:A2} become:
\begin{equation*}
\xb^{k+1} := \kprox{\norms{\cdot}_2/\rho_k}{\Bb\hat{\yb}^k - \cb} ~~\text{and}~~~\tilde{\yb}^{k+1} := \kprox{\sigma_k\norms{\cdot}_1}{\ub^k},
\end{equation*}
where $\sigma_k := \frac{\kappa_2}{\kappa_1 + \rho_k\norms{\Bb}^2}$ and $\ub^k := \frac{\norms{\Bb}^2\hat{\yb}^k - \Bb^{\top}(\Bb\hat{\yb}^k - \xb^{k+1} -\cb)}{\kappa_1/\rho_k + \norms{\Bb}^2}$.

In order to apply the Chambolle-Pock method in \cite[Algorithm 2]{Chambolle2011}, we define $F(\Bb\yb) := \norms{\Bb\yb - \cb}_2$ and $G(\yb) :=  \tfrac{\kappa_1}{2}\norms{\yb}^2 + \kappa_2\norms{\yb}_1$.
In this case, $G$ is strongly convex with the parameter $\mu_g = \kappa_2$.
Hence, we choose the parameters as suggested in \cite[Algorithm 2]{Chambolle2011}.
When $\kappa_2 = 0$, i.e., $G$ is non-strongly convex, we use again \cite[Algorithm 1]{Chambolle2011} with the parameters $\sigma = \tau = \frac{1}{2\norm{\Bb}^2}$ and $\theta = 1$.

We compare again the first $7$ algorithms discussed above to solve \eqref{eq:elastic_lasso_exam}.
We generate the data as follows. Matrix $\Bb \in \R^{n\times p_2}$ is generated randomly using standard Gaussian distribution $\Nc(0, 1)$, and then is normalized by $\frac{1}{\sqrt{n}}$, i.e., $\Bb := \tfrac{1}{\sqrt{n}}\texttt{randn}(n, p_2)$.
We generate a sparse vector $\yb^{\natural}$ with $s$-nonzero entries sampling from the standard Gaussian distribution as the true parameter vector.  
Then, we generate the observed measurement as $\cb = \Bb\yb^{\natural} + \bar{\sigma}\Nc(0, 1)$, where $\bar{\sigma} = 0$ in the noiseless case, and $\bar{\sigma} = 10^{-3}$ in the noisy case.
We choose $\kappa_1 = 0.1$ and $\kappa_2 = 0.01$ for our test. In this case, we obtain solutions with approximately $2\%$ sparsity.

\begin{figure}[htp!]
\begin{center}
\vspace{-4ex}
\includegraphics[width=1.0\linewidth]{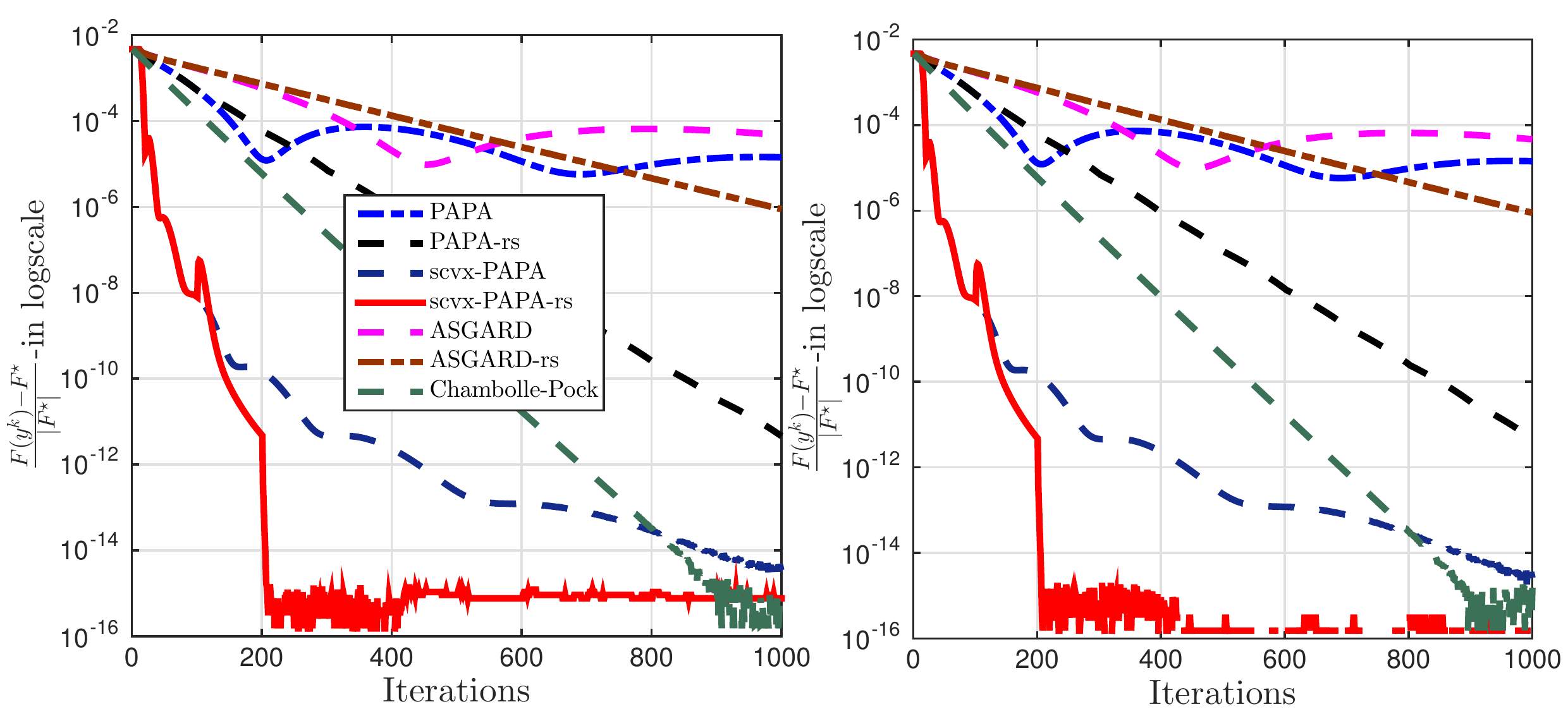}
\vspace{-3ex}
\caption{A comparison of $7$ algorithms on the original objective residual $\tfrac{F(\yb^k) - F^{\star}}{\vert F^{\star}\vert}$ of \eqref{eq:elastic_lasso_exam} after $1000$ iterations. 
The problem size is $(p_2 = 5000, n = 1750, s = 500)$.
Left: \textbf{without noise};  Right: \textbf{with Gaussian noise} (with variance $\bar{\sigma} = 10^{-3}$).
}\label{fig:elastic5000}
\vspace{-6ex}
\end{center}
\end{figure}

Figure \ref{fig:elastic5000} shows the actual convergence behavior of  two instances of \eqref{eq:elastic_lasso_exam} with noise and without noise respectively, in terms of the relative objective residual $\frac{F(\yb^k) - F^{\star}}{\vert F^{\star}\vert}$ of \eqref{eq:elastic_lasso_exam}, where the optimal value $F^{\star}$ is computed via CVX \cite{Grant2006} using Mosek with the best precision.

The theoretical algorithms, i.e., \texttt{PAPA} and \texttt{ASGARD} \cite{TranDinh2015b}, still show the $\BigO{\frac{1}{k}}$-rate on the original objective residual.
But their restarting variants  exhibit a much better convergence rate without employing the strong convexity.
\texttt{ASGARD} with restart performs worse than \texttt{PAPA-rs} in this example.
If we exploit the convexity as in Algorithm \ref{alg:A2}, then this algorithm and its restart variant completely outperform other methods.
The theoretical version of Algorithm~\ref{alg:A2} performs significantly well in this example, beyond the theoretical $\BigO{\tfrac{1}{k^2}}$-rate. 
It even performs better than Chambolle-Pock's method with the strong convexity \cite[Algorithm 2]{Chambolle2011}.
The restarting variant requires approximately $200$ iterations to achieve up to the $10^{-15}$ accuracy level.

\beforesubsec
\subsection{Square-root LASSO}
\aftersubsec
We now show that Algorithm~\ref{alg:A2} still works well even when the problem is not strongly convex using again \eqref{eq:elastic_lasso_exam}.
In this test, we set $\kappa_1 = 0$, and problem \eqref{eq:elastic_lasso_exam} reduces to the common square-root LASSO problem \cite{Belloni2011}.
We test $3$ algorithms as above on a new instance of \eqref{eq:elastic_lasso_exam} with the size $(p_2 = 5000, n = 1750, s = 500)$, and noise.
Since $\kappa_1 = 0$, we do not know if the problem is strongly convex or not. 
Hence, we select three different values of $\mu_g$ in Algorithm~\ref{alg:A2} and the Chambolle-Pock method as $\mu_g = 1$, $\mu_g = 0.1$ and $\mu_g = 0.01$.
Figure \ref{fig:sqrt_lasso} shows the result of this test when we restart at every $100$ iterations (left), and $50$ iterations (right).

\begin{figure}[H]
\begin{center}
\vspace{-4ex}
\includegraphics[width=1.0\linewidth]{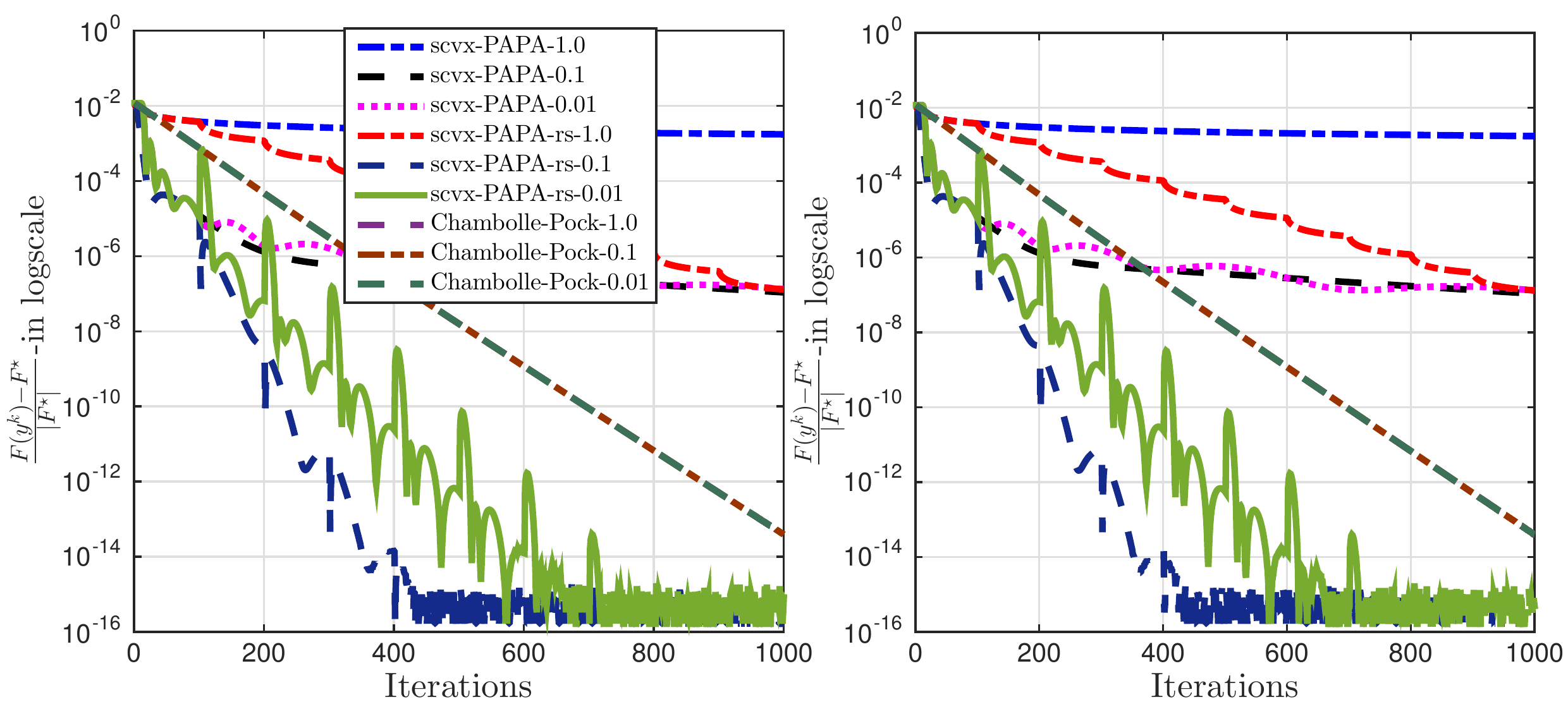}
\vspace{-3ex}
\caption{A comparison of $9$ algorithmic variants on the square-root LASSO problem \eqref{eq:elastic_lasso_exam} (i.e., $\kappa_1 = 0$) after $1000$ iterations. 
The problem size is $(p_2 = 5000, n = 1750, s = 500)$.
Left: Restarting after each $100$ iterations; Right: Restarting after each $50$ iterations.
}\label{fig:sqrt_lasso}
\vspace{-6ex}
\end{center}
\end{figure}

Figure \ref{fig:sqrt_lasso} shows that Algorithm~\ref{alg:A2} still has the $\BigO{\tfrac{1}{k^2}}$-rate.
The restarting Algorithm~\ref{alg:A2} with $\mu_g = 0.1$ still outperforms Algorithm~\ref{alg:A1}, and the Chambolle-Pock method with strong convexity.
When $\mu_g = 0.01$, it still performs well compared to the Chambolle-Pock method, but if $\mu_g = 1$, then it becomes worse.
This is affected by the choice of the initial value $\rho_0 = \frac{\mu_g}{2\norms{\Bb}^2}$, which is inappropriate.

\beforesubsec
\subsection{Image reconstruction with low sampling rate}\label{subsec:image}
\aftersubsec
We consider an image reconstruction problem using  low sampling rates as:
\begin{equation}\label{eq:image_exam}
F^{\star} := \min_{Y\in\R^{m_1\times m_2}} \set{ F(Y) := \tfrac{1}{2}\norms{\Ac(Y) - \bb}_F^2 + \kappa\norms{Y}_{\mathrm{TV}} },
\end{equation}
where $\Ac$ is a linear operator, $\bb$ is a measurement vector, $\norms{\cdot}_F$ is the Frobenius norm, $\kappa > 0$ is a regularization parameter, and $\norms{Y}_{\mathrm{TV}}$ is the total-variation norm.

To apply our methods, we use $\norms{Y}_{\mathrm{TV}} = \norm{D(Y)}_1$ and reformulate this problem into \eqref{eq:constr_cvx2} as
\begin{equation*}
\min_{X, Y}\set{ \kappa\norms{X}_1 + \tfrac{1}{2}\norms{\Ac(Y) - \bb}_F^2 \mid X - D(Y) = 0},
\end{equation*}
where we choose $f(X) := \kappa\norms{X}_1$, $g(Y) := 0$, and $h(Y) := \tfrac{1}{2}\norms{\Ac(Y) - \bb}_F^2$ which is Lipschitz gradient continuous with $L_h := \norms{\Ac^{\ast}\Ac}$.
Although $h$ and $g$ may not be quasi-strongly convex, we still apply Algorithm~\ref{alg:A2}(b) in Subsection~\ref{subsec:three_objs} to solve it.

We also implement Vu-Condat's algorithm \cite{Condat2013,vu2013splitting} and FISTA \cite{Beck2009,Nesterov2007} to directly solve  \eqref{eq:image_exam}.  
For Vu-Condat's algorithm, we implement the following scheme:
\begin{equation}\label{eq:vu_condat}
\left\{\begin{array}{ll}
\tilde{\Xb}^k &:=  \kprox{\tau g}{\Xb^k - \tau(\nabla{h}(\Xb^k) + \Db^{\ast}(\Xb^k))} \vspace{1ex}\\
\tilde{\Xb}^k &:=  \kprox{\sigma f}{\Yb^k + \sigma\Db( 2\tilde{\Xb}^k - \Xb^k)} \vspace{1ex}\\
(\Xb^{k+1},\Yb^{k+1}) &:=  (1-\theta)(\Xb^k, \Yb^k) + \theta(\tilde{\Xb}^k, \tilde{\Yb}^k),
\end{array}\right.
\end{equation}
where $\Db^{\ast}$ is the adjoint operator of $\Db$, $\theta :=  1$, and $\tau > 0$ and $\sigma > 0$ satisfying $\tfrac{1}{\tau} - \sigma\norms{\Db}^2 \geq \tfrac{L_h}{2}$.
The last condition leads $0 < \tau < \frac{2}{L_h}$ and $0 < \sigma \leq \tfrac{1}{\norms{\Db}^2}\left(\tfrac{1}{\tau} - \tfrac{L_h}{2}\right)$.
We test Vu-Condat's algorithm using $(\tau, \sigma) :=  \left(\tfrac{0.089}{L_h},  \tfrac{1}{\norms{\Db}^2}\left(\tfrac{1}{\tau} - \tfrac{L_h}{2}\right) \right)$ after carefully tuning these parameters.

For Algorithm~\ref{alg:A1}, we set $\rho_0 :=  \frac{1}{2\norm{\Db}}$, and for Algorithm~\ref{alg:A2}, we set $\rho_0 :=  \frac{1}{4\norm{\Db}^2}$.
We also implement the restarting variants of both algorithms with a frequency of $s = 20$ iterations.
For FISTA, we compute the proximal operator $\prox_{\norm{\cdot}_{\mathrm{TV}}}$ using a primal-dual method as in \cite{Chambolle2011} by setting the number of iterations at $25$ and $50$, respectively, and also use a fixed restarting strategy after each $50$ iterations \cite{Odonoghue2012}.

We test these algorithms on $6$ MRI images of different sizes downloaded from different websites.
We generate the observed measurement $\bb$ by using subsampling -FFT transform at the rate of $20\%$.
After tuning the regularization parameter $\kappa$, we fix it at $\kappa = 4.0912\times 10^{-4}$ for all the experiments.
Table \ref{tbl:image_reconstruction} shows the results of $8$ algorithms on these MRI images after $200$ iterations in terms of the objective values, computational time, and PSNR (Peak signal-to-noise ratio) \cite{Chambolle2011}.

\begin{table}[hpt!]
\newcommand{\cell}[1]{{\!\!}#1{\!\!}}
\vspace{-4ex}
\begin{center}
\begin{scriptsize}
\caption{The results and performance of $8$ algorithms on $6$ MRI images}\label{tbl:image_reconstruction}
\begin{tabular}{l | rrr | rrr | rrr} \toprule
\multicolumn{1}{c}{}  & \multicolumn{3}{|c}{\texttt{Hip}~$(798\times 802)$} & \multicolumn{3}{|c}{\texttt{Knee}~$(779\times 693)$} & \multicolumn{3}{|c}{\texttt{Brain-tomor}~$(650\times 650)$} \\  \midrule
\cell{Algorithms} & \cell{$F(Y^k)$} & \cell{PSNR} & \cell{Time[s]} & \cell{$F(Y^k)$} & \cell{PSNR} & \cell{Time[s]} & \cell{$F(Y^k)$} & \cell{PSNR} & \cell{Time[s]} \\  \midrule
\cell{\texttt{PAPA}} & 			\cell{0.01070} & \cell{81.56} & \cell{57.97} & 	\cell{0.00840} & \cell{79.62} & \cell{48.69} & 	\cell{0.01050} & \cell{77.82} & \cell{34.63}  \\
\cell{\texttt{PAPA-rs}} &		\cell{0.01056} & \cell{81.35} & \cell{57.57} & 	 \cell{0.00828} & \cell{79.51} & \cell{48.69} & 	\cell{0.01039} & \cell{77.72} & \cell{34.36}  \\
\cell{\texttt{scvx-PAPA}} &  	\cell{0.01034} & \cell{81.24} & \cell{64.34} &	\cell{0.00805} & \cell{79.44} & \cell{53.81} &  	\cell{0.01025} & \cell{77.70} & \cell{37.98} \\ 	
\cell{\texttt{scvx-PAPA-rs}} & 	\cell{0.01035} & \cell{81.23} & \cell{67.47} & 	\cell{0.00807} & \cell{79.45} & \cell{53.48} & 	\cell{0.01026} & \cell{77.69} & \cell{37.93} \\ 
\cell{\texttt{Vu-Condat-tuned}} &	\cell{0.01030} & \cell{81.23} & \cell{56.02} & 	\cell{0.00801} & \cell{79.45} & \cell{45.92} & 	\cell{0.01023} & \cell{77.70} & \cell{31.71} \\
\cell{\texttt{AcProxGrad-25}} &	\cell{0.01179} & \cell{82.25} & \cell{1055.94} & 	\cell{0.00917} & \cell{79.78} & \cell{844.47} & 	\cell{0.01133} & \cell{78.74} & \cell{674.08} \\ 
\cell{\texttt{AcProxGrad-rs}} & 	\cell{0.01179} & \cell{82.25} & \cell{1052.23} &	\cell{0.00917} & \cell{79.78} & \cell{860.97} &  	\cell{0.01133} & \cell{78.74} & \cell{652.91} \\
\cell{\texttt{AcProxGrad-50}} &	\cell{0.01104} & \cell{82.30} & \cell{2052.68} & 	\cell{0.00865} & \cell{79.83} & \cell{1652.81} &  \cell{0.01079} & \cell{78.80} & \cell{1264.33} \\ \bottomrule
\multicolumn{1}{c}{}  & \multicolumn{3}{|c}{\texttt{Body}~$(895 \times 320)$} & \multicolumn{3}{|c}{\texttt{Confocal}~$(370 \times 370)$} & \multicolumn{3}{|c}{\texttt{Leg}~$(588\times 418)$} \\  \midrule
\cell{Algorithms} & \cell{$F(Y^k)$} & \cell{PSNR} & \cell{Time[s]} & \cell{$F(Y^k)$} & \cell{PSNR} & \cell{Time[s]} & \cell{$F(Y^k)$} & \cell{PSNR} & \cell{Time[s]} \\ \midrule
\cell{\texttt{PAPA}} & 			 \cell{0.01674} & \cell{66.92} & \cell{22.80} & 	\cell{0.02539} & \cell{67.58} & \cell{12.12} & 	\cell{0.01050} & \cell{74.50} & \cell{22.30} \\ 
\cell{\texttt{PAPA-rs}} &		 \cell{0.01664} & \cell{66.96} & \cell{22.70} & 	\cell{0.02534} & \cell{67.60} & \cell{11.80} & 	\cell{0.01040} & \cell{74.37} & \cell{22.81} \\ 
\cell{\texttt{scvx-PAPA}} &  	 \cell{0.01653} & \cell{66.98} & \cell{25.09} & 	\cell{0.02528} & \cell{67.67} & \cell{13.28} & 	\cell{0.01030} & \cell{74.36} & \cell{25.22} \\
\cell{\texttt{scvx-PAPA-rs}} &      \cell{0.01664} & \cell{66.98} & \cell{25.15} & 	\cell{0.02529} & \cell{67.61} & \cell{13.41} & 	\cell{0.01030} & \cell{74.34} & \cell{25.83} \\ 
\cell{\texttt{Vu-Condat-tuned}} & \cell{0.01652} & \cell{66.99} & \cell{22.99} & 	\cell{0.02527} & \cell{67.74} & \cell{10.84} & 	\cell{0.01028} & \cell{74.38} & \cell{20.70}  \\ 
\cell{\texttt{AcProxGrad-25}} &   \cell{0.01728} & \cell{67.35} & \cell{400.36} & 	\cell{0.02652} & \cell{68.97} & \cell{136.46} & 	\cell{0.01104} & \cell{75.23} & \cell{361.63}  \\
\cell{\texttt{AcProxGrad-rs}} &    \cell{0.01728} & \cell{67.35} & \cell{431.07} & 	\cell{0.02652} & \cell{68.97} & \cell{132.13} & 	\cell{0.01104} & \cell{75.23} & \cell{366.83} \\
\cell{\texttt{AcProxGrad-50}} &  \cell{0.01697} & \cell{67.39} & \cell{817.63} & 	\cell{0.02639} & \cell{68.97} & \cell{256.97} & 	\cell{0.01074} & \cell{75.21} & \cell{700.66}  \\ 
\bottomrule
\end{tabular}
\end{scriptsize}
\vspace{-4ex}
\end{center}
\end{table}

Table \ref{tbl:image_reconstruction} shows that our algorithms and Vu-Condat's method outperform FISTA in terms of computational time. 
This is not surprised since FISTA requires to evaluate an expensive proximal operator of the TV-norm at each iteration.
However, it gives a slightly better PSNR while producing worse objective values than our methods.
Vu-Condat's algorithm with tuned parameters has a similar performance as our methods.
Unfortunately, the restarting variants with a fixed frequency, e.g., $s = 50$, do not significantly improve the performance of all methods in this example.
This happens perhaps due to the nonstrong convexity of the problem.

In order to observe the quality of reconstruction, we plot the result of $8$ algorithm in Figure~\ref{fig:im_reconstruction} for one MRI image (\texttt{Hip}) of the size $798\times 802$ (i.e., $p_2=639,996$).
\begin{figure}[hpt!]
\vspace{-3ex}
\begin{center}
\includegraphics[width=0.95\linewidth]{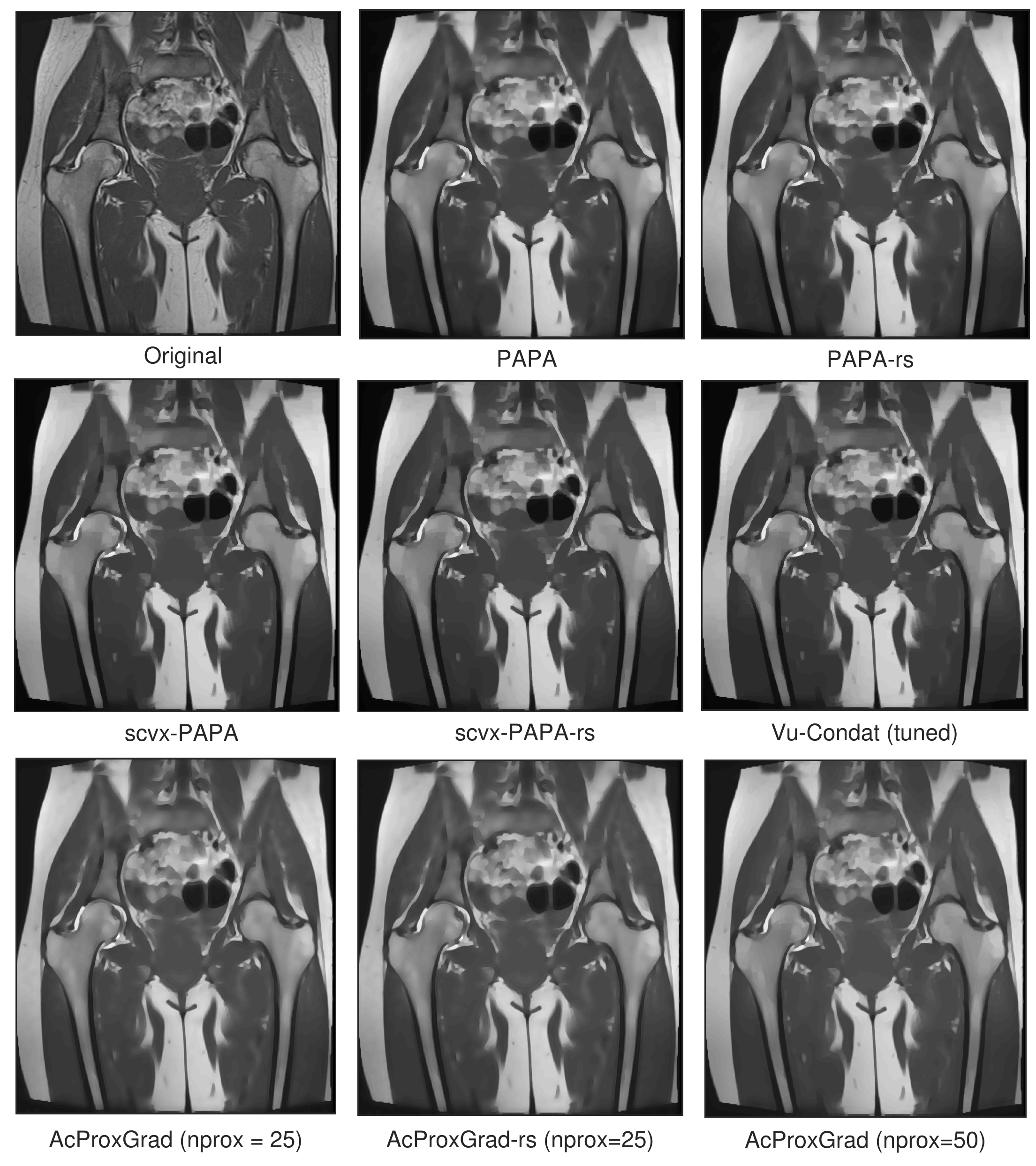}
\vspace{-1ex}
\caption{The original image, and its reconstructions from $8$ algorithms using $20\%$ of measurement. Here, \texttt{nprox} is the number of iterations required to evaluate the proximal operator of the TV-norm, and Vu-Condat (tuned) is Vu-Condat's method \cite{Condat2013,vu2013splitting} using tuned parameters.}\label{fig:im_reconstruction}
\vspace{-5ex}
\end{center}
\end{figure}
Clearly, we can see that the quality of the reconstruction is still acceptable with only $20\%$ of the measurement. 

\beforesubsec
\subsection{Low-rank matrix recovery with square-root loss}
\aftersubsec
We consider a low-rank matrix recovery problem with square-root loss, which can be considered as a penalized formulation of the model in \cite{Recht2010}:
\vspace{-1ex}
\begin{equation}\label{eq:low_rank_exam}
F^{\star} := \min_{Y\in\R^{m\times q}} \Big\{ F(Y) := \norms{\Bc(Y) - \cb}_2 + \lambda \norms{Y}_{\ast} \Big\},
\vspace{-1ex}
\end{equation}
where $\norms{\cdot}_{\ast}$ is a nuclear norm, $\Bc : \R^{m\times q} \to \R^n$ is a linear operator, $\cb\in\R^n$ is a given observed vector, and $\lambda > 0$ is a  penalty parameter.
By letting $\zb := (\xb, Y)$, $F(\zb) := \norms{\xb}_2 + \lambda \norms{Y}_{\ast}$ and $-\xb + \Bc(Y)  = \cb$, we can reformulate \eqref{eq:low_rank_exam} into \eqref{eq:constr_cvx}. 

Now, we apply Algorithm~\ref{alg:A1} and its restarting variant to solve \eqref{eq:low_rank_exam}.
Although $f$ and $g$ in problem \eqref{eq:low_rank_exam} are non-strongly convex, we still apply Algorithm \ref{alg:A2} to solve \eqref{eq:low_rank_exam}.
The main computation at each iteration of these algorithms consists of: $\prox_{\norms{\cdot}_{\ast}}$, $\Bc(Y)$ and $\Bc^{*}(\xb)$ which dominate the overall computational time.
We also compare these three algorithmic variants with ADMM.
Since ADMM often requires to solve two convex subproblems, we reformulate \eqref{eq:low_rank_exam} into 
\vspace{-1ex}
\begin{equation*}
\min_{\xb, Y, Z}\Big\{ \norms{\xb}_2 + \lambda\norm{Z}_{\ast} ~\mid~ -\xb + \Bc(\Yb) = \cb,~\Yb - \Zb = 0 \Big\},
\vspace{-1ex}
\end{equation*}
by introducing two auxiliary variables $x := \Bc(Y) - c$ and $Z := Y$.
The main computation at each iteration of ADMM includes $\prox_{\norms{\cdot}_{\ast}}$, $\Bc(Y)$, $\Bc^{*}(\xb)$, and the solution of $(\Id + \Bc^{\ast}\Bc)(Y) = \rb_k$, where $r_k$ is a residual term.
In this particular example, since $\Bc$ and $\Bc^{\ast}$ are given in operators, we apply a preconditioned conjugate gradient (PCG) method to solve it. We warm-start PCG and terminate it with a tolerance of $10^{-5}$ or  a maximum of $50$ iterations.
We tune the penalty parameter $\rho$ in ADMM for our test and find that $\rho = 0.25$ works best.

We test four algorithms on $5$ Logo images: MIT, UNC, EPFL, TUM and IBM.
The size of these images is  $256\times 256$, which shows that the number of variables are $65,536$.
We generate the observed measurement $\cb$ by using subsampling -FFT transform as in Subsection \ref{subsec:image} but with a rate of $35\%$.
We also add a Gaussian noise to $\cb$ as $\cb = \Bc(Y^{\natural}) + \Nc(0, 10^{-3}\max_{ij}\vert Y^{\natural}_{ij}\vert)$, where $Y^{\natural}$ is a clean low-rank image.
We tune the value of $\lambda$ for these five images and find that $\lambda \in \set{ 0.175, 0.125, 0.15, 0.125, 0.125 }$, respectively works well for these images.

We run four algorithms on five images up to $200$ iterations. 
The results and performance are reported in Table \ref{tbl:low_rank_recovery}.
Here, \texttt{Time} is the computational time in second, \texttt{Error} is the relative error $\frac{ \norms{Y^k - Y^{\natural} }_F}{\norms{Y^{\natural} }_F}$ between the approximate solution $Y^k$ and the true image, \texttt{PSNR} is the peak signal-to-noise ratio, \texttt{rank} is the rank of $Y^k$ after rounding up to $10^{-4}$, and \texttt{Res} is the relative residual $\norms{\Bc(Y^k) - \cb}_2/\norm{\cb}_2$.
\begin{table}[hpt!]
\vspace{-4ex}
\newcommand{\cell}[1]{{\!\!\!}#1{\!\!}}
\begin{center}
\begin{scriptsize}
\caption{The results and performance of $4$ algorithms on $5$ Logo images of size $256\times 256$.}\label{tbl:low_rank_recovery}
\begin{tabular}{l | rrrr rr | rrrr rr } \toprule
\multicolumn{1}{c}{}  & \multicolumn{6}{|c}{PAPA} & \multicolumn{6}{|c}{PAPA-rs}  \\  \midrule
\cell{Name} & \cell{Time} & \cell{Error} & \cell{$F(Y^k)$} & \cell{PSNR} & \cell{rank} & \cell{Res} & \cell{Time} & \cell{Error} & \cell{$F(Y^k)$} & \cell{PSNR} & \cell{rank} & \cell{Res} \\  \midrule
\cell{MIT} &  \cell{$7.05$} & \cell{$0.0510$} & \cell{0.34838} & \cell{74.026} & \cell{6} & \cell{0.103} & \cell{$7.38$} & \cell{$0.0511$} & \cell{0.34838} & \cell{74.014}  & \cell{6} & \cell{0.103} \\ 
\cell{UNC} &  \cell{$9.29$} & \cell{$0.0610$} & \cell{0.28197} & \cell{72.479} & \cell{42} & \cell{0.110} & \cell{$9.33$} & \cell{$0.0610$} & \cell{0.28199} & \cell{72.467}  & \cell{42} & \cell{0.110} \\ 
\cell{EPFL} & \cell{$9.27$} & \cell{$0.0823$} & \cell{0.41245} & \cell{69.896} & \cell{52} & \cell{0.107} & \cell{$9.41$} & \cell{$0.0822$} & \cell{0.41255} & \cell{69.885}  & \cell{52} & \cell{0.107} \\ 
\cell{TUM} & \cell{$8.07$} & \cell{$0.0374$} & \cell{0.26573} & \cell{76.711} & \cell{49} & \cell{0.087} & \cell{$7.10$} & \cell{$0.0377$} & \cell{0.26595} & \cell{76.649}  & \cell{49} & \cell{0.087} \\ 
\cell{IBM} & \cell{$9.24$} & \cell{$0.0627$} & \cell{0.29107} & \cell{72.229} & \cell{32} & \cell{0.107} & \cell{$8.24$} & \cell{$0.0629$} & \cell{0.29110} & \cell{72.212}  & \cell{32} & \cell{0.107} \\ 
\midrule
\multicolumn{1}{c}{}  & \multicolumn{6}{|c}{scvx-PAPA} & \multicolumn{6}{|c}{ADMM (tuned)}  \\  \midrule
\cell{Name} & \cell{Time} & \cell{Error} & \cell{$F(Y^k)$} & \cell{PSNR} & \cell{rank} & \cell{Res} & \cell{Time} & \cell{Error} & \cell{$F(Y^k)$} & \cell{PSNR} & \cell{rank} & \cell{Res} \\  \midrule
\cell{MIT} &  \cell{$7.45$} & \cell{$0.0510$} & \cell{0.34838} & \cell{74.030} & \cell{6} & \cell{0.103} & \cell{$14.18$} & \cell{$0.0510$} & \cell{0.34838} & \cell{74.022} & \cell{6} & \cell{0.103} \\ 
\cell{UNC} &  \cell{$8.90$} & \cell{$0.0609$} & \cell{0.28194} & \cell{72.492} & \cell{42} & \cell{0.110} & \cell{$14.91$} & \cell{$0.0609$} & \cell{0.28198} & \cell{72.476} & \cell{42} & \cell{0.110} \\ 
\cell{EPFL} &  \cell{$8.07$} & \cell{$0.0821$} & \cell{0.41240} & \cell{69.901} & \cell{52} & \cell{0.107} & \cell{$14.30$} & \cell{$0.0822$} & \cell{0.41249} & \cell{69.893} & \cell{53} & \cell{0.107} \\ 
\cell{TUM}  &  \cell{$7.57$} & \cell{$0.0374$} & \cell{0.26569} & \cell{76.730} & \cell{48} & \cell{0.086} & \cell{$14.51$} & \cell{$0.0375$} & \cell{0.26579} & \cell{76.687} & \cell{49} & \cell{0.087} \\ 
\cell{IBM} & \cell{$9.43$} & \cell{$0.0627$} & \cell{0.29105} & \cell{72.239} & \cell{32} & \cell{0.107} & \cell{$14.90$} & \cell{$0.0628$} & \cell{0.29108} & \cell{72.224} & \cell{32} & \cell{0.107} \\ 
\bottomrule
\end{tabular}
\end{scriptsize}
\vspace{-5ex}
\end{center}
\end{table}

As we can observed from Table \ref{tbl:low_rank_recovery} that four algorithms achieve almost similar results.
Since three variants of PAPA have the same per-iteration complexity, they have almost the same computational time in this test.
ADMM is slower since it requires to solve a linear system at each iteration with PCG.
We note that these algorithms give a low-rank solution compared to the size of $256\times 256$ of the images.
To see how the low-rankness is reflected in the final output, we plot three Logo images: MIT, UNC, and IBM in Figure \ref{fig:low_rank_recovery}.
Due to their low-rankness, MIT and IBM are clearer than UNC.
The quality of the recovered images is reflected through \texttt{PSNR} and \texttt{Error} in Table \ref{tbl:low_rank_recovery}.

\begin{figure}[hpt!]
\vspace{-2ex}
\begin{center}
\includegraphics[width=0.98\linewidth]{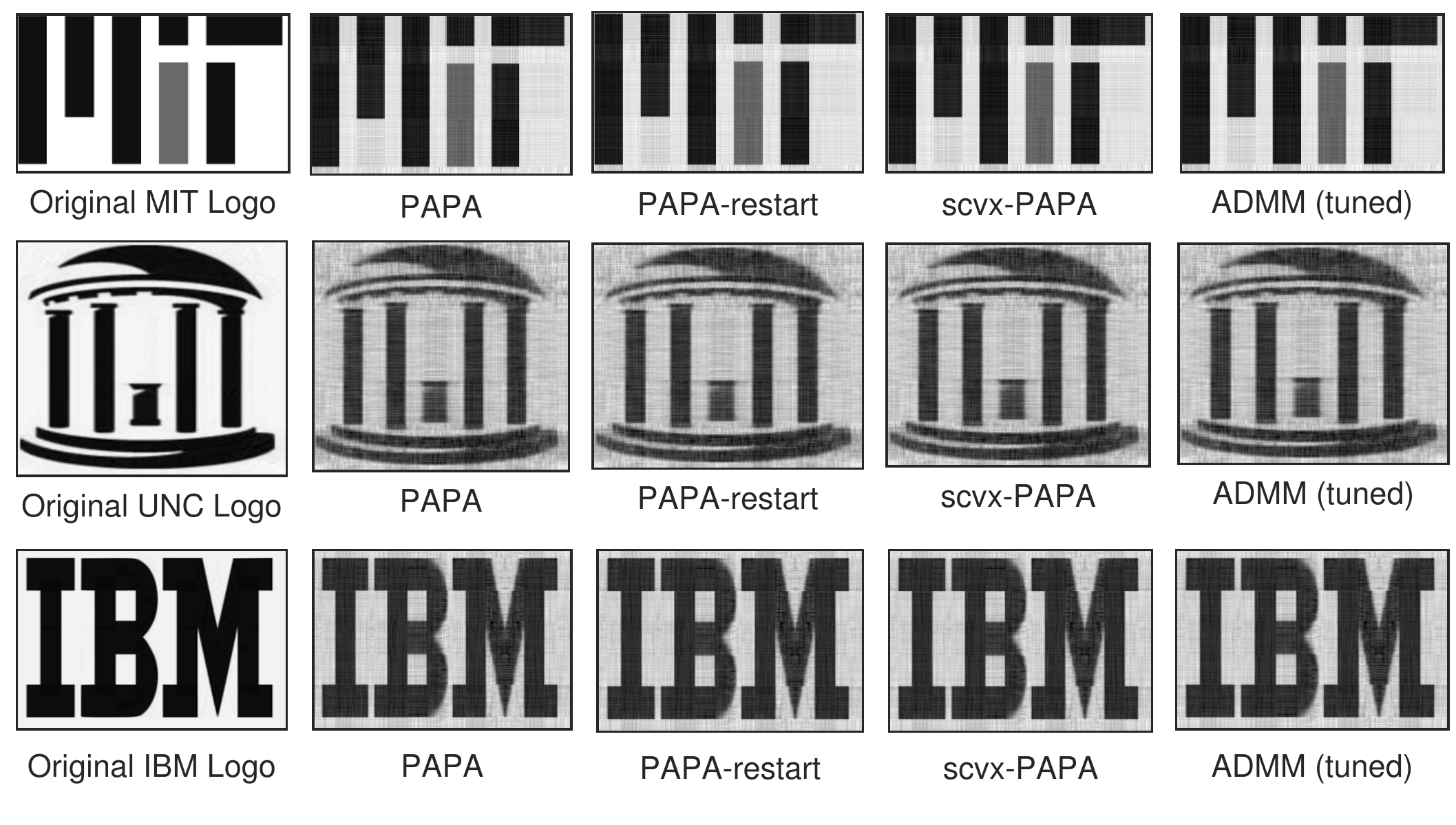}
\vspace{-1ex}
\caption{Three original Logo images and their recovered images from 4 algorithms.}\label{fig:low_rank_recovery}
\vspace{-5ex}
\end{center}
\end{figure}

\vspace{0.5ex}
\begin{footnotesize}
\noindent \textbf{Acknowledgments:} 
This work is partly supported by the NSF-grant, DMS-1619884, USA, and the Nafosted grant 101.01-2017.315 (Vietnam).
\end{footnotesize}

\appendix
\normalsize
\beforesec
\section{Appendix: The proof of technical results in the main text}\label{sec:appendix}
\aftersec
This appendix provides the full proof of the technical results in the main text.

\beforesubsec
\subsection{Properties of the distance function $\mathrm{dist}_{\Kc}(\cdot)$.}\label{apdx:proj_properties}
\aftersubsec
We investigate some necessary properties of $\psi$ defined by \eqref{eq:Phi_func} to analyze the convergence of Algorithms~\ref{alg:A1} and \ref{alg:A2}.
We first consider the following distance function:
\begin{equation}\label{eq:dist_func}
\varphi(\ub) := \tfrac{1}{2}\kdist{\Kc}{\ub}^2 = \displaystyle\min_{\rb\in\Kc}\tfrac{1}{2}\norms{\rb - \ub}^2  = \tfrac{1}{2}\norms{\rb^{\ast}(\ub) - \ub}^2 = \tfrac{1}{2}\norms{\kproj{\Kc}{\ub} - \ub}^2,
\end{equation}
where $\rb^{\ast}(\ub) := \kproj{\Kc}{\ub}$ is the projection of $\ub$ onto $\Kc$.
Clearly, \eqref{eq:dist_func} becomes
\begin{equation}\label{eq:phi_minmax}
\varphi(\ub) = \displaystyle\max_{\lbd\in\R^n}\displaystyle\min_{\rb\in\Kc}\set{\iprods{\ub - \rb, \lbd} - \tfrac{1}{2}\norm{\lbd}^2} = \displaystyle\max_{\lbd\in\R^n}\set{\iprods{\ub, \lbd} - s_{\Kc}(\lbd) - \tfrac{1}{2}\norms{\lbd}^2},
\end{equation}
where $s_{\Kc}(\lbd) := \sup_{\rb\in\Kc}\iprods{\lbd, \rb}$ is the support function of $\Kc$.

The function $\varphi$ is convex and differentiable. Its gradient is given by
\begin{equation}\label{eq:grad_varphi}
\nabla{\varphi}(\ub) = \ub - \kproj{\Kc}{\ub} = \nu^{-1}\kproj{\Kc^{\circ}}{\nu \ub},
\end{equation}
where $\Kc^{\circ} := \set{ \vb\in\R^n \mid \iprods{\ub, \vb} \leq 1, ~\ub\in\Kc}$ is the polar set of $\Kc$, and $\nu > 0$ solves $\nu = \iprods{\kproj{\Kc^{\circ}}{\nu\ub}, \nu u - \kproj{\Kc^{\circ}}{\nu\ub}}$.
If $\Kc$ is a cone, then $\nabla{\varphi}(\ub) = \kproj{\Kc^{\circ}}{\ub} = \kproj{-\Kc^{\ast}}{\ub}$, where $\Kc^{\ast} := \set{\vb\in\R^n \mid \iprods{\ub, \vb} \geq 0, ~\ub\in\Kc}$ is the dual cone of $\Kc$ \cite{Bauschke2011}.

By using the property of $\proj_{\Kc}(\cdot)$, it is easy to prove that $\nabla{\varphi}(\cdot)$ is Lipschitz continuous with the Lipschitz constant $L_{\varphi} = 1$.
Hence, for any $\ub, \vb\in\R^n$, we have (see \cite{Nesterov2004}):
\begin{equation}\label{eq:lipschitz_xy}
\begin{array}{ll}
\varphi(\ub) &+ \iprods{\nabla{\varphi}(\ub), \vb - \ub} + \tfrac{1}{2}\norms{\nabla{\varphi}(\vb) - \nabla{\varphi}(\ub)}^2 \leq \varphi(\vb), \vspace{1ex}\\
\varphi(\vb) & \leq \varphi(\ub) + \iprods{\nabla{\varphi}(\ub), \vb - \ub} + \tfrac{1}{2}\norms{\vb - \ub}^2.
\end{array}
\end{equation}
Let us recall $\psi$ defined by \eqref{eq:Phi_func} as
\begin{equation}\label{eq:psi_func_def}
\psi(\xb, \yb) := \varphi(\Ab\xb + \Bb\yb - \cb) = \tfrac{1}{2}\kdist{\Kc}{\Ab\xb + \Bb\yb - \cb}^2.
\end{equation}
Then, $\psi$ is also convex and differentiable, and its gradient is given by 
\begin{equation}\label{eq:psi_grad_def}
\begin{array}{ll}
\nabla_x\psi(\xb, \yb) &= \Ab^{\top}\left(\Ab\xb + \Bb\yb - \cb - \kproj{\Kc}{\Ab\xb + \Bb\yb - \cb}\right), \vspace{1ex}\\
\nabla_y\psi(\xb, \yb) &= \Bb^{\top}\left(\Ab\xb + \Bb\yb - \cb - \kproj{\Kc}{\Ab\xb + \Bb\yb - \cb}\right).
\end{array}
\end{equation}
For given $\xb^{k+1}\in\R^{p_1}$ and $\hat{\yb}^k\in\R^{p_2}$, let us define the following two functions:
\begin{equation}\label{eq:lin_func}
{\!\!\!}\begin{array}{ll}
 \Qc_k(\yb) &:= \psi(\xb^{k+1}, \hat{\yb}^k) + \iprods{\nabla_y\psi(\xb^{k+1}, \hat{\yb}^k), \yb - \hat{\yb}^k} + \tfrac{\norms{\Bb}^2}{2}\norms{\yb - \hat{\yb}^k}^2. \vspace{1ex}\\
\ell_k(\zb ) &:= \psi(\xb^{k+1}, \hat{\yb}^k) + \iprods{\nabla_x\psi(\xb^{k+1}, \hat{\yb}^k), \xb -  \xb^{k+1}} + \iprods{\nabla_y\psi(\xb^{k+1}, \hat{\yb}^k), \yb -  \hat{\yb}^k}.
\end{array}{\!\!}
\end{equation}
Then, the following lemma provides some properties of $\ell_k$ and $\Qc_k$. 

\begin{lemma}\label{le:property_of_lk}
Let $\zb^{\star} = (\xb^{\star}, \yb^{\star}) \in\R^{p}$ be such that $\Ab\xb^{\star} + \Bb\yb^{\star} - \cb\in\Kc$. 
Then, for $\ell_k$ defined by \eqref{eq:lin_func} and $\psi$ defined by \eqref{eq:psi_func_def}, we have
\begin{equation}\label{eq:lin_func_pro}
\ell_k(\zb^{\star}) \leq -\tfrac{1}{2}\norms{\hat{\sb}^{k+1}}^2 ~~~\text{and}~~~~\ell_k(\zb^k) \leq \psi(\xb^k, \yb^k) - \tfrac{1}{2}\norms{\sb^k - \hat{\sb}^{k+1}}^2,
\end{equation}
where $\hat{\sb}^{k+1} :=  \Ab\xb^{k+1} + \Bb\hat{\yb}^k - \cb - \proj_{\Kc}\big(\Ab\xb^{k+1} + \Bb\hat{\yb}^k - \cb\big)$ and $\sb^k := \Ab\xb^k + \Bb\yb^k - \cb - \proj_{\Kc}\big(\Ab\xb^k + \Bb\yb^k - \cb\big)$.
Moreover, we also have
\begin{equation}\label{eq:psi_grad_pro}
\psi(\xb^{k+1}, \yb) \leq   \Qc_k(\yb)~~\text{for all}~\yb\in\R^{p_2}.
\end{equation}
\end{lemma}

\begin{proof}
Since $\Ab\xb^{\star} + \Bb\yb^{\star} - \cb\in\Kc$, if we define $\rb^{\star} := \Ab\xb^{\star} + \Bb\yb^{\star} - \cb$, then $\rb^{\star} \in\Kc$.
Let $\hat{\ub}^{k} := \Ab\xb^{k+1} + \Bb\hat{\yb}^k - \cb \in \R^n$. 
We can derive
\begin{equation*}
\begin{array}{ll}
\ell_k(\zb^{\star}) &:= \psi(\xb^{k+1}, \hat{\yb}^k) + \iprods{\nabla_x\psi(\xb^{k+1}, \hat{\yb}^k), \xb^{\star} - \xb^{k+1}} + \iprods{\nabla_y\psi(\xb^{k+1}, \hat{\yb}^k), \yb^{\star} -  \hat{\yb}^k} \vspace{1ex}\\
& \overset{\tiny\eqref{eq:psi_func_def}}{=}  \iprods{\hat{\ub}^{k} -  \proj_{\Kc}(\hat{\ub}^{k}), \Ab(\xb^{\star} - \xb^{k+ 1}) + \Bb(\yb^{\star} - \hat{\yb}^k)}  + \tfrac{1}{2}\norms{\hat{\ub}^{k} - \proj_{\Kc}(\hat{\ub}^{k})}^2 \vspace{1.25ex}\\
& = \iprods{\hat{\ub}^{k} -  \proj_{\Kc}(\hat{\ub}^{k}), \rb^{\star} - \proj_{\Kc}(\hat{\ub}^{k})}  -  \tfrac{1}{2}\norms{\hat{\ub}^{k}  -  \proj_{\Kc}(\hat{\ub}^{k})}^2 \vspace{1.25ex}\\
&\leq -\tfrac{1}{2}\norms{\hat{\ub}^{k} - \proj_{\Kc}(\hat{\ub}^{k})}^2,
\end{array}
\end{equation*}
which is the first inequality of \eqref{eq:lin_func_pro}.
Here, we use the property $\iprods{\hat{\ub}^{k} - \proj_{\Kc}(\hat{\ub}^{k}), \rb^{\star} - \proj_{\Kc}(\hat{\ub}^{k})} \leq 0$ for any $\rb^{\star}\in\Kc$ of the projection $\proj_{\Kc}$.
The second inequality of \eqref{eq:lin_func_pro} follows directly from \eqref{eq:lipschitz_xy} and the definition of $\psi$ in \eqref{eq:psi_func_def}.
The proof of \eqref{eq:psi_grad_pro} can be found in \cite{Nesterov2004} due to the Lipschitz continuity of $\nabla_y\psi(\xb^{k+1},\cdot)$.
\Eproof
\end{proof}

\beforesubsec
\subsection{Descent property of the alternating scheme  in Algorithm~\ref{alg:A1} and Algorithm \ref{alg:A2}.}\label{apdx:le:key_descent_lemma}
\aftersubsec

\begin{lemma}\label{le:key_descent_lemma}
Let $\ell_k$ and $\Qc_k$ be defined by \eqref{eq:lin_func}, and $\Phi_{\rho}$ be defined by \eqref{eq:Phi_func}.
\begin{itemize}
\item[$\mathrm{(a)}$] 
Let $\zb^{k+1} := (\xb^{k+1}, \yb^{k+1})$  be generated by Step~\ref{eq:alter_scheme} of Algorithm~\ref{alg:A1}.
Then, for any $\zb := (\xb, \yb) \in\dom{F}$, we have
\begin{equation}\label{eq:key_descent_lemma}
\begin{array}{ll}
\Phi_{\rho_k}(\zb^{k+1}) &\leq F(\zb) + \rho_k\ell_k(\zb) + \gamma_k\iprods{\xb^{k+1} - \hat{\xb}^k, \xb - \hat{\xb}^k}  - \gamma_k\norms{\xb^{k+1} - \hat{\xb}^k}^2\vspace{1ex}\\
& + \rho_k\norms{\Bb}^2\iprods{\yb^{k + 1}  -  \hat{\yb}^k, \yb -\hat{\yb}^k}  - \frac{\rho_k\norm{\Bb}^2}{2}\norms{\yb^{k+ 1} -  \hat{\yb}^k}^2.
\end{array}
\end{equation}
\item[$\mathrm{(b)}$] 
Alternatively, let $\zb^{k+1} := (\xb^{k+1}, \yb^{k+1})$  be generated by Step~\ref{step:x_cvx_subprob} of Algorithm~\ref{alg:A2}, and $\breve{\yb}^{k+1} := (1-\tau_k)\yb^k + \tau_k\tilde{\yb}^{k+1}$.
Then,  for any $\zb := (\xb, \yb) \in \dom{F}$, we have
\begin{equation}\label{eq:key_descent_lemma2}
\begin{array}{ll}
\breve{\Phi}_{k+1} &:= f(\xb^{k+1}) + g(\breve{\yb}^{k+1}) + \rho_k\Qc_k(\breve{\yb}^{k+1}) \vspace{1ex}\\
& \leq (1-\tau_k) \big[ F(\zb^{k}) + \rho_k\ell_k(\zb^k) \big] + \tau_k\big[ F(\zb) + \rho_k\ell_k(\zb) \big]\vspace{1ex}\\ 
& + \tfrac{\gamma_0\tau_k^2}{2}\norms{\tilde{\xb}^k - \xb}^2 -  \tfrac{\gamma_0\tau_k^2}{2}\norms{\tilde{\xb}^{k+1} - \xb}^2 \vspace{1ex}\\
&  + \tfrac{\rho_k\tau_k^2\norms{\Bb}^2}{2}\norms{\tilde{\yb}^k - \yb}^2 - \tfrac{\left(\rho_k\tau_k^2\norms{\Bb}^2 + \mu_g\tau_k\right)}{2}\norms{\tilde{\yb}^{k+1} - \yb}^2.
\end{array}
\end{equation}
\end{itemize}
\end{lemma}

\begin{proof}
(a)~Combining the optimality condition of two subproblems at Step~\ref{eq:alter_scheme} of Algorithm~\ref{alg:A1}, and the convexity of $f$ and $g$, we can derive
\begin{equation}\label{eq:opt_cond_subprob_xy}
\left\{\begin{array}{lll}
f(\xb^{k+1})  &\leq f(\xb) + \iprods{\rho_k\nabla_x{\psi}(\xb^{k+1},\hat{\yb}^k) + \gamma_k(\xb^{k+1} - \hat{\xb}^k), \xb - \xb^{k+1}}, \vspace{1ex}\\
g(\yb^{k+1}) &\leq g(\yb) + \iprods{\rho_k\nabla_y{\psi}(\xb^{k+1},\hat{\yb}^k) + \rho_k\norms{\Bb}^2(\yb^{k+1} \!-\! \hat{\yb}^k), \yb - \yb^{k+1}}.
\end{array}\right.
\end{equation}
Using \eqref{eq:psi_grad_pro} with $\yb = \yb^{k+1}$, we have
\begin{equation*}
\psi(\xb^{k+1}, \yb^{k+1}) \leq \psi(\xb^{k+1}, \hat{\yb}^k) + \iprods{\nabla_y{\psi}(\xb^{k+1}, \hat{\yb}^k), \yb^{k+1} - \hat{\yb}^k} + \tfrac{\norms{\Bb}^2}{2}\norms{\yb^{k+1} - \hat{\yb}^k}^2.
\end{equation*}
Combining the last estimate and \eqref{eq:opt_cond_subprob_xy}, and then using \eqref{eq:Phi_func}, we can derive
\begin{equation*}
\begin{array}{ll}
\Phi_{\rho}(\zb^{k+1}) &\overset{\tiny \eqref{eq:Phi_func}}{=} f(\xb^{k+1}) + g(\yb^{k+1}) + \rho_k\psi(\xb^{k+1}, \yb^{k+1}) \vspace{1ex}\\
&\overset{\tiny\eqref{eq:opt_cond_subprob_xy}}{\leq} f(\xb) + g(\yb) + \rho_k\ell_k(\zb)  + \gamma_k\iprods{\hat{\xb}^k - \xb^{k+1}, \xb^{k+1} - \xb} \vspace{1ex}\\
& + \rho_k\norms{\Bb}^2\iprods{\hat{\yb}^k - \yb^{k+1}, \yb^{k+1} - \yb} +  \tfrac{\rho_k\norms{\Bb}^2}{2}\norms{\yb^{k+1}-\hat{\yb}^k}^2 \vspace{1ex}\\
&= f(\xb) + g(\yb) + \rho_k\ell_k(\zb)   + \gamma_k\iprods{\hat{\xb}^k - \xb^{k+1}, \hat{\xb}^k - \xb} \vspace{1ex}\\
& - \gamma_k\norms{\xb^{k+1} - \hat{\xb}^k}^2 + \rho_k\norms{\Bb}^2\iprods{\hat{\yb}^k - \yb^{k+1}, \hat{\yb}^k - \yb}  - \frac{\rho_k\norms{\Bb}^2}{2}\norms{\yb^{k+1}  - \hat{\yb}^k}^2,
\end{array}
\end{equation*}
which is exactly \eqref{eq:key_descent_lemma}.

(b)~First, from the definition of $\ell_k$ and $\Qc_k$ in  \eqref{eq:lin_func}, using $\breve{\yb}^{k+1} - \hat{\yb}^k = \tau_k(\tilde{\yb}^{k+1} - \tilde{\yb}^k)$ and $\xb^{k+1} - (1-\tau_k)\xb^k - \tau_k\tilde{\xb}^{k+1} = 0$, we can show that
\begin{align}\label{eq:proof2_est2}
\Qc_k(\breve{\yb}^{k\!+\!1}) &\overset{\tiny\eqref{eq:lin_func}}{=} (1-\tau_k)\ell_k(\zb^k) + \tau_k\ell_k(\tilde{\zb}^{k+1})   + \tfrac{\norms{\Bb}^2\tau_k^2}{2}\norms{\tilde{\yb}^{k+1} - \tilde{\yb}^k}^2.
\end{align}
By  the convexity of $f$, $\tau_k\tilde{\xb}^{k+1} = \xb^{k+1} - (1-\tau_k)\xb^k$ from \eqref{eq:x_hat}, and the optimality condition of the $x$-subproblem at Step~\ref{step:x_cvx_subprob} of Algorithm~\ref{alg:A2}, we can derive
\begin{equation}\label{eq:proof2_13a}
\begin{array}{ll}
f(\xb^{k+1}) &\leq (1 - \tau_k)f(\xb^k) +  \tau_k f(\xb)  + \tau_k \iprods{\nabla{f}(\xb^{k+1}), \tilde{\xb}^{k+1} - \xb} \vspace{1ex}\\
&= (1 - \tau_k)f(\xb^k) + \tau_k f(\xb)  + \rho_k\tau_k \iprods{\nabla_x{\psi}(\xb^{k+1},\hat{\yb}^k),\xb -  \tilde{\xb}^{k+1}}\vspace{1ex}\\
& + \gamma_0\tau_k\iprods{\xb^{k+1} - \hat{\xb}^k, \xb - \tilde{\xb}^{k+1}},
\end{array}
\end{equation}
for any $\xb\in\R^{p_1}$, where $\nabla{f}(\xb^{k\!+\!1})\in\partial{f}(\xb^{k\!+\!1})$.

By the $\mu_g$-strong convexity of $g$, $\breve{\yb}^{k+1} := (1-\tau_k)\yb^k + \tau_k\tilde{\yb}^{k+1}$, and the optimality condition of the $y$-subproblem at Step~\ref{step:x_cvx_subprob} of Algorithm~\ref{alg:A2}, one can also derive
\begin{equation}\label{eq:proof2_13b}
{\!\!\!\!}\begin{array}{ll}
g(\breve{\yb}^{k+1})  &\leq (1-\tau_k)g(\yb^k) + \tau_kg(\yb) + \tau_k\iprods{\nabla{g}(\tilde{\yb}^{k+1}), \tilde{\yb}^{k+1}{\!\!} - \yb} 
 - \frac{\tau_k\mu_g}{2}\norms{\tilde{\yb}^{k+1} {\!\!}- \yb}^2 \vspace{1ex}\\
 &= (1-\tau_k)g(\yb^k) + \tau_kg(\yb) + \rho_k\tau_k\iprods{\nabla_y{\psi}(\xb^{k+1},\hat{\yb}^k), \yb - \tilde{\yb}^{k+1}} \vspace{1ex}\\
 & + \rho_k\tau_k^2\norms{\Bb}^2\iprods{\tilde{\yb}^{k+1} - \tilde{\yb}^k, \yb - \tilde{\yb}^{k+1}} - \frac{\tau_k\mu_g}{2}\norms{\tilde{\yb}^{k+1} {\!\!}- \yb}^2,
 \end{array}{\!\!\!}
\end{equation}
for any $\yb\in\R^{p_2}$, where $\nabla{g}(\tilde{\yb}^{k+1}) \in \partial{g}(\tilde{\yb}^{k+1})$.

Combining this, \eqref{eq:proof2_est2}, \eqref{eq:proof2_13a} and \eqref{eq:proof2_13b} and then using $\breve{\Phi}_k$, we have
\begin{equation}\label{eq:proof2_est11}
{\!\!\!\!\!\!\!}\begin{array}{ll}
\breve{\Phi}_{k+1} & = f(\xb^{k+1}) + g(\breve{\yb}^{k+1}) + \rho_k\Qc_k(\breve{\yb}^{k+1}) \vspace{1ex}\\
&{\!\!} \overset{\tiny\eqref{eq:proof2_est2},\eqref{eq:proof2_13a},\eqref{eq:proof2_13b}}{\leq} (1-\tau_k)\big[ F(\zb^{k}) + \rho_k\ell_k(\zb^k) \big] + \tau_k\big[F(\zb) + \rho_k\ell_k(\zb) \big]\vspace{1ex}\\
&+ \gamma_0\tau_k\iprods{\xb^{k+1} - \hat{\xb}^k, \xb - \tilde{\xb}^{k+1}} + \rho_k\tau_k^2\norms{\Bb}^2\iprods{\tilde{\yb}^{k+1} - \tilde{\yb}^k, \yb - \tilde{\yb}^{k+1}} \vspace{1ex}\\
&{\!\!}  + \frac{1}{2}\rho_k\tau_k^2\norms{\Bb}^2\norms{\tilde{\yb}^{k+1} - \tilde{\yb}^k}^2 - \frac{\tau_k\mu_g}{2}\norms{\tilde{\yb}^{k+1} - \yb}^2.
\end{array}{\!\!\!\!\!\!\!}
\end{equation}
Next, using  \eqref{eq:x_hat}, for any $\zb = (\xb, \yb)\in\dom{F}$, we also have
\begin{equation}\label{eq:elem_estimate}
\begin{array}{ll}
2\tau_k\iprods{\hat{\xb}^k - \xb^{k+1}{\!\!}, \tilde{\xb}^{k} - \xb}&= \tau_k^2\norms{\tilde{\xb}^k - \xb}^2  -  \tau_k^2\norms{\tilde{\xb}^{k+1} - \xb}^2 + \norms{\xb^{k+1} - \hat{\xb}^k}^2, \vspace{1ex}\\
2\iprods{\tilde{\yb}^k - \tilde{\yb}^{k+1}, \tilde{\yb}^{k+1} - \yb} &=   \norms{\tilde{\yb}^k - \yb}^2  - \norms{\tilde{\yb}^{k+1} - \yb}^2 - \norms{\tilde{\yb}^{k+1} - \tilde{\yb}^k}^2.
\end{array}
\end{equation}
Substituting \eqref{eq:elem_estimate} into \eqref{eq:proof2_est11} we obtain
\begin{equation*} 
\begin{array}{ll}
\breve{\Phi}_{k+1} & \leq (1-\tau_k) \big[ F(\zb^{k}) + \rho_k\ell_k(\zb^k) \big] + \tau_k\big[ F(\zb) + \rho_k\ell_k(\zb) \big] \vspace{1ex}\\ 
& + \tfrac{\gamma_0\tau_k^2}{2}\norms{\tilde{\xb}^k - \xb}^2 -  \tfrac{\gamma_0\tau_k^2}{2}\norms{\tilde{\xb}^{k+1} - \xb}^2 - \tfrac{\gamma_0}{2} \norms{\xb^{k+1} - \hat{\xb}^k}^2 \vspace{1ex}\\
& + \tfrac{\rho_k\tau_k^2\norms{\Bb}^2}{2}\norms{\tilde{\yb}^k - \yb}^2 - \tfrac{\left(\rho_k\tau_k^2\norms{\Bb}^2 + \mu_g\tau_k \right)}{2}\norms{\tilde{\yb}^{k+1} - \yb}^2,
\end{array}
\end{equation*}
which is exactly \eqref{eq:key_descent_lemma2} by neglecting the term $-\frac{\gamma_0}{2}\norms{\xb^{k+1} - \hat{\xb}^k}^2$.
\Eproof
\end{proof}

\beforesubsec
\subsection{The proof of Lemma~\ref{le:key_estimate}: The key estimate of Algorithm~\ref{alg:A1}}\label{apdx:le:key_estimate}
\aftersubsec
Using the fact that $\tau_k = \frac{1}{k+1}$, we have $\frac{\tau_{k+1}(1-\tau_k)}{\tau_k} = \frac{k}{k+2}$.
Hence, the third line of Step \ref{eq:alter_scheme} of Algorithm~\ref{alg:A1} can be written as 
\begin{equation*}
(\hat{\xb}^{k\!+\!1}, \hat{\yb}^{k\!+\!1} ) =  (\xb^{k\!+\!1}, \yb^{k\!+\!1}) + \tfrac{\tau_{k+1}(1-\tau_k)}{\tau_k}(\xb^{k+1} - \xb^k, \yb^{k+1} - \yb^k).
\end{equation*} 
This step can be split into two steps with $(\hat{\xb}^k, \hat{\yb}^k)$ and $(\tilde{\xb}^k, \tilde{\yb}^k)$ as in \eqref{eq:update_xtilde}, which is standard in accelerated gradient methods \cite{Beck2009,Nesterov2004}. 
We omit the detailed derivation.

Next, we prove \eqref{eq:key_est1}. 
Using \eqref{eq:lin_func_pro}, we have
\begin{equation}\label{eq:lm31_est3}
\ell_k(\zb^k) \leq \psi(\xb^k, \yb^k) - \tfrac{1}{2}\norms{\sb^k  - \hat{\sb}^{k+1}}^2, ~~\text{and}~~\ell_k(\zb^{\star})  \leq -\tfrac{1}{2}\norms{\hat{\sb}^{k+1}}^2.
\end{equation}
Using \eqref{eq:key_descent_lemma} with $(\xb, \yb) = (\xb^k, \yb^k)$ and $(\xb, \yb) = (\xb^{\star}, \yb^{\star})$ respectively, we obtain
\begin{equation*}
\begin{array}{ll}
\Phi_{\rho_k}(\zb^{k+1}) &\overset{\tiny\eqref{eq:lm31_est3}}{\leq} \Phi_{\rho_k}(\zb^k)  + \gamma_k\iprods{\hat{\xb}^k - \xb^{k+1}, \hat{\xb}^{k} - \xb^k}  - \gamma_k\norms{\hat{\xb}^k - \xb^{k+1}}^2 \vspace{1ex}\\
& + \rho_k\norms{\Bb}^2\iprods{\hat{\yb}^k - \yb^{k+1}, \hat{\yb}^{k} - \yb^k} - \frac{\rho_k\norms{\Bb}^2}{2}\norms{\hat{\yb}^k - \yb^{k+1}}^2 - \frac{\rho_k}{2}\norms{\sb^k - \hat{\sb}^{k+1}}^2. \vspace{1ex}\\
\Phi_{\rho_k}(\zb^{k+1}) &\leq F(\zb^{\star})  - \frac{\rho_k}{2}\norms{\hat{\sb}^{k+1}}^2 + \gamma_k\iprods{\hat{\xb}^k - \xb^{k+1}, \hat{\xb}^{k} - \xb^{\star}}  - \gamma_k\norms{\hat{\xb}^k - \xb^{k+1}}^2 \vspace{1ex}\\
&+ \rho_k\norms{\Bb}^2\iprods{\hat{\yb}^k - \yb^{k+1}, \hat{\yb}^{k} - \yb^{\star}}  - \frac{\rho_k\norms{\Bb}^2}{2}\norms{\hat{\yb}^k - \yb^{k+1}}^2.
\end{array}
\end{equation*}
Multiplying the first inequality  by $1 - \tau_k \in [0, 1]$ and the second one by $\tau_k \in [0, 1]$, and summing up the results, then using  $\hat{\xb}^k - (1-\tau_k)\xb^k = \tau_k\tilde{\xb}^k$ and $\hat{\yb}^k - (1-\tau_k)\yb^k = \tau_k\tilde{\yb}^k$ from  \eqref{eq:update_xtilde}, we obtain
\begin{equation}\label{eq:lm31_est5}
{\!\!\!\!\!\!}\begin{array}{ll}
\Phi_{\rho_k}(\zb^{k+1}) &{\!\!}\leq  (1-\tau_k)\Phi_{\rho_k}(\zb^k)  + \tau_kF(\zb^{\star}) + \gamma_k\tau_k\iprods{\hat{\xb}^k - \xb^{k+1}, \tilde{\xb}^{k} - \xb^{\star}} \vspace{1ex}\\
& - \gamma_k\norms{\xb^{k\!+\!1} \!-\! \hat{\xb}^k}^2  \!+\! \rho_k\tau_k\norms{\Bb}^2\iprods{\hat{\yb}^k \!-\! \yb^{k\!+\!1}, \tilde{\yb}^{k} \!-\! \yb^{\star}} \!-\! \frac{\rho_k\norms{\Bb}^2}{2}\norms{\yb^{k\!+\!1}  \!-\! \hat{\yb}^k}^2 \vspace{1ex}\\
&  -~ \frac{(1-\tau_k)\rho_k}{2}\norms{s^k - \hat{s}^{k+1}}^2 - \frac{\tau_k\rho_k}{2}\norms{\hat{s}^{k+1}}^2.
\end{array}{\!\!\!\!\!\!\!\!}
\end{equation}
By the update rule in \eqref{eq:update_xtilde} we can show that
\begin{equation*}
\begin{array}{ll}
2\tau_k\iprods{\hat{\xb}^k - \xb^{k+1}, \tilde{\xb}^{k} - \xb^{\star}}  &= \tau_k^2\norms{\tilde{\xb}^k - \xb^{\star}}^2 -  \tau_k^2\norms{\tilde{\xb}^{k+1} - \xb^{\star}}^2 + \norms{\xb^{k+1} - \hat{\xb}^k}^2, \vspace{1ex}\\
2\tau_k\iprods{\hat{\yb}^k - \yb^{k+1}, \tilde{\yb}^k - \yb^{\star}} &= \tau_k^2\norms{\tilde{\yb}^k - \yb^{\star}}^2 -  \tau_k^2\norms{\tilde{\yb}^{k+1} - \yb^{\star}}^2 +  \norms{\yb^{k+1} - \hat{\yb}^k}^2.
\end{array}
\end{equation*}
Using this relation and $\Phi_{\rho_{k}}(\zb^k) = \Phi_{\rho_{k-1}}(\zb^k) + \frac{(\rho_k - \rho_{k-1})}{2}\norms{\sb^k}^2$ into \eqref{eq:lm31_est5}, we get
\begin{equation}\label{eq:lm31_est7} 
{\!\!\!\!\!\!}\begin{array}{ll}
\Phi_{\rho_k}(\zb^{k+1}) &{\!\!} \leq  (1-\tau_k)\Phi_{\rho_{k-1}}(\zb^k)  + \tau_kF(\zb^{\star}) + \gamma_k\tau_k^2\left[\norms{\tilde{\xb}^k - \xb^{\star}}^2 -  \norms{\tilde{\xb}^{k+1} - \xb^{\star}}^2\right]\vspace{1ex}\\
&{\!\!\!\!} - \tfrac{\gamma_k}{2}\norms{\hat{x}^k - \xb^{k+1}}^2 + \frac{\norms{\Bb}^2\rho_k\tau_k^2}{2}\left[\norms{\tilde{\yb}^k - \yb^{\star}}^2 - \norms{\tilde{\yb}^{k+1} - \yb^{\star}}^2\right] - R_k,
\end{array}{\!\!\!\!}
\end{equation}
where $R_k$ is defined as
\begin{equation}\label{eq:Rk_est}
\begin{array}{ll}
R_k &:= \frac{(1-\tau_k)\rho_k}{2}\norms{\sb^k - \hat{\sb}^{k+1}}^2 + \frac{\rho_k\tau_k}{2}\norms{\hat{\sb}^{k+1}}^2 - \frac{(1-\tau_k)(\rho_k - \rho_{k-1})}{2}\norms{\sb^k}^2 \vspace{1ex}\\
& \geq \tfrac{(1-\tau_k)}{2}\left[\rho_{k-1} - \rho_k(1-\tau_k)\right]\norms{\sb^k}^2.
\end{array}
\end{equation}
Using \eqref{eq:Rk_est} into \eqref{eq:lm31_est7} and ignoring $-\frac{\gamma_k}{2}\norms{\xb^{k+1} - \hat{\xb}^k}^2$, we obtain \eqref{eq:key_est1}.
\Eproof

\beforesubsec
\subsection{The proof of Lemma~\ref{le:tseng_key_est}: The key estimate of Algorithm~\ref{alg:A2}}\label{apdx:le:tseng_key_est}
\aftersubsec
The proof of \eqref{eq:x_hat} is similar to the proof of \eqref{eq:update_xtilde}, and we skip its details here.

Using $\zb = \zb^{\star}$ and \eqref{eq:lin_func_pro} into \eqref{eq:key_descent_lemma2}, we obtain
\begin{equation}\label{eq:key_descent_lemma2_star}
\begin{array}{ll}
\breve{\Phi}_{k+1} &:= f(\xb^{k+1}) + g(\breve{\yb}^{k+1}) + \rho_k\Qc_k(\breve{\yb}^{k+1}) \vspace{1ex}\\
& \overset{\tiny\eqref{eq:lin_func_pro}}{\leq} (1-\tau_k) \big[ F(\zb^{k}) + \rho_k\ell_k(\zb^k) \big] + \tau_kF(\zb^{\star})  - \frac{\rho_k\tau_k}{2}\norms{\hat{\sb}^{k+1}}^2 \vspace{1ex}\\ 
& +~ \tfrac{\gamma_0\tau_k^2}{2}\norms{\tilde{\xb}^k - \xb^{\star}}^2 -  \tfrac{\gamma_0\tau_k^2}{2}\norms{\tilde{\xb}^{k+1} - \xb^{\star}}^2 \vspace{1ex}\\
&  +~ \tfrac{\rho_k\tau_k^2\norms{\Bb}^2}{2}\norms{\tilde{\yb}^k - \yb^{\star}}^2 - \tfrac{\left(\rho_k\tau_k^2\norms{\Bb}^2 + \mu_g\tau_k\right)}{2}\norms{\tilde{\yb}^{k+1} - \yb^{\star}}^2.
\end{array}
\end{equation}
From the definition of $\psi$ in \eqref{eq:psi_func_def} and \eqref{eq:lin_func_pro}, we have
\begin{equation*}
\begin{array}{ll}
\Phi_{\rho_k}(\zb^k) = \Phi_{\rho_{k-1}}(\zb^k) + \frac{(\rho_k - \rho_{k-1})}{2}\norms{\sb^k}^2~~~\text{and}~~~\ell_k(\zb^k) \leq \psi(\xb^k, \yb^k) - \tfrac{1}{2}\norms{ \sb^k - \hat{\sb}^{k+1}}^2.
\end{array}
\end{equation*}
Using these expressions into \eqref{eq:key_descent_lemma2_star}, we obtain
\begin{equation}\label{eq:proof2_est11d}
{\!\!\!\!\!\!}\begin{array}{ll}
\breve{\Phi}_{k+1} &\leq (1-\tau_k)\Phi_{\rho_{k-1}}(\zb^k) + \tau_kF(\zb^{\star}) + \tfrac{\gamma_0\tau_k^2}{2}\norms{\tilde{\xb}^k - \xb^{\star}}^2  -  \tfrac{\gamma_0\tau_k^2}{2}\norms{\tilde{\xb}^{k+1} - \xb^{\star}}^2 \vspace{1ex}\\
&  + \tfrac{\norms{\Bb}^2\rho_k\tau_k^2}{2}\norms{\tilde{\yb}^k - \yb^{\star}}^2 - \tfrac{(\norms{\Bb}^2\rho_k\tau_k^2 + \mu_g\tau_k)}{2}\norms{\tilde{\yb}^{k+1} - \yb^{\star}}^2  - R_k,
\end{array}{\!\!\!\!\!\!\!\!}
\end{equation}
where $R_k$ is defined as \eqref{eq:Rk_est}.

Let us consider two cases:
\begin{itemize}
\item For \textbf{Option 1} at Step~\ref{step:A2_step5} of Algorithm~\ref{alg:A2}, we have $\yb^{k+1} = \breve{\yb}^{k+1}$.
Hence, using \eqref{eq:psi_grad_pro}, we get
\begin{equation}\label{eq:lower_bound_on_phik_a}
\Phi_{\rho_k}(\zb^{k+1}) =  f(\xb^{k+1}) + g(\yb^{k+1}) + \rho_k\psi(\xb^{k+1},\yb^{k+1}) \leq \breve{\Phi}_{k+1}.
\end{equation}
\item For \textbf{Option 2} at Step~\ref{step:A2_step5} of Algorithm~\ref{alg:A2}, we have
\begin{equation}\label{eq:lower_bound_on_phik_b}
\begin{array}{ll}
\Phi_{\rho_k}(\zb^{k+1}) &\leq f(\xb^{k+1}) + g(\yb^{k+1}) + \rho_k\Qc_k(\yb^{k+1}) \vspace{1ex}\\
& = f(\xb^{k+1}) + \displaystyle\min_{\yb\in\R^{p_2}}\Big\{ g(\yb) + \rho_k\Qc_k(\yb) \Big\} \vspace{1ex}\\
&\leq f(\xb^{k+1}) + g(\breve{\yb}^{k+1}) + \rho_k\Qc_k(\breve{\yb}^{k+1}) = \breve{\Phi}_{k+1}.
\end{array}
\end{equation}
\end{itemize}
Using either \eqref{eq:lower_bound_on_phik_a} or \eqref{eq:lower_bound_on_phik_b} into \eqref{eq:proof2_est11d}, we obtain
\begin{equation*} 
{\!\!\!\!\!\!}\begin{array}{ll}
\Phi_{\rho_k}(\zb^{k+1}) &\leq (1-\tau_k)\Phi_{\rho_{k-1}}(\zb^k) + \tau_kF(\zb^{\star}) + \tfrac{\gamma_0\tau_k^2}{2}\norms{\tilde{\xb}^k - \xb^{\star}}^2  -  \tfrac{\gamma_0\tau_k^2}{2}\norms{\tilde{\xb}^{k+1} - \xb^{\star}}^2 \vspace{1ex}\\
&  + \tfrac{\norms{\Bb}^2\rho_k\tau_k^2}{2}\norms{\tilde{\yb}^k - \yb^{\star}}^2 - \tfrac{(\norms{\Bb}^2\rho_k\tau_k^2 + \mu_g\tau_k)}{2}\norms{\tilde{\yb}^{k+1} - \yb^{\star}}^2 - R_k,
\end{array}{\!\!\!\!\!\!}
\end{equation*}
Using the lower bound \eqref{eq:Rk_est} of $R_k$ into this inequality, we  obtain \eqref{eq:key_est3}.
\Eproof
\beforesubsec
\subsection{The proof of Corollary \ref{co:com_cvx_convergence}: Application to composite convex minimization}\label{apdx:co:com_cvx_convergence}
\aftersubsec
By the $L_f$-Lipschitz continuity of $f$ and Lemma~\ref{le:approx_opt_cond}, we have 
\begin{equation}\label{eq:co1_est1}
{\!\!\!\!\!\!}\begin{array}{ll}
0 \leq P(\yb^k) - P^{\star} &= f(\yb^k) + g(\yb^k) - P^{\star} \leq f(\xb^k) + g(\yb^k) + \vert f(\yb^k) - f(\xb^k)\vert - P^{\star} \vspace{1ex}\\
&\leq f(\xb^k) + g(\yb^k) - P^{\star} + L_f\norms{\xb^k - \yb^k} \vspace{1ex}\\
&\overset{\tiny\eqref{eq:approx_opt_cond}}{\leq} S_{\rho_{k-1}}(\zb^k) - \frac{\rho_{k-1}}{2}\norms{\xb^k - \yb^k}^2 + L_f\norms{\xb^k - \yb^k},
\end{array}{\!\!\!\!\!\!}
\end{equation}
where $S_{\rho}(\zb) := \Phi_{\rho}(\zb) - P^{\star}$.
This inequality also leads to 
\begin{equation}\label{eq:co1_est2}
\begin{array}{ll}
\norms{\xb^k - \yb^k} &\leq \frac{1}{\rho_{k-1}}\left(L_f + \sqrt{L_f^2 + 2\rho_{k-1}S_{\rho_{k-1}}(\zb^k)}\right) \vspace{1ex}\\
& \leq \frac{1}{\rho_{k-1}}\left(2L_f + \sqrt{2\rho_{k-1}S_{\rho_{k-1}}(\zb^k)}\right).
\end{array}
\end{equation}
Since using \eqref{eq:DR_method2} is equivalent to applying Algorithm~\ref{alg:A1} to its constrained reformulation, by \eqref{eq:key_est10}, we have 
\begin{equation*}
S_{\rho_{k-1}}(\zb^k) \leq \frac{\rho_0\norms{\yb^0 - \yb^{\star}}^2}{2k}~~~\text{and}~~~\rho_{k-1} = \rho_0k.
\end{equation*}
Using these expressions into \eqref{eq:co1_est2} we get 
\begin{equation*}
\norms{\xb^k - \yb^k} \leq \frac{1}{\rho_0k}\left(2L_f + \sqrt{\rho_0^2\norms{\yb^0 - \yb^{\star}}^2}\right) = \frac{2L_f + \rho_0\norms{\yb^0-\yb^{\star}}}{\rho_0k}.
\end{equation*} 
Substituting this into \eqref{eq:co1_est1} and using the bound of $S_{\rho_k}$, we obtain \eqref{eq:co1_conv_rate1}.

Now, if we use \eqref{eq:DR_method3}, then it is equivalent to applying Algorithm~\ref{alg:A2} with \textbf{Option 1} to solve its constrained reformulation.
In this case, from the proof of Theorem~\ref{th:convergence2}, we can derive
\begin{equation*}
S_{\rho_{k-1}}(\zb^k) \leq \frac{2\rho_0\norms{\yb^0 - \yb^{\star}}^2}{(k+1)^2}~~~~\text{and}~~~\frac{\rho_0(k+1)^2}{4} \leq \rho_{k-1} \leq k^2\rho_0.
\end{equation*}
Combining these estimates and \eqref{eq:co1_est2}, we have $\norms{\xb^k - \yb^k} \leq \tfrac{8(L_f +  \rho_0\norms{\yb^0 - \yb^{\star}})}{\rho_0(k+1)^2}$.
Substituting this into \eqref{eq:co1_est1} and using the bound of $S_{\rho_{k-1}}$ we obtain \eqref{eq:co1_conv_rate2}.
\Eproof

\beforesubsec
\subsection{The proof of Theorem~\ref{co:three_objs}: Extension to the sum of three objective functions}\label{apdx:co:three_objs}
\aftersubsec
Using the Lipschitz gradient continuity of $h$ and \cite[Theorem 2.1.5]{Nesterov2004}, we have
\begin{equation*} 
{\!\!\!\!}\begin{array}{ll}
h(\yb^{k+1}) &\leq h(\hat{\yb}^k) + \iprods{\nabla{h}(\hat{\yb}^k), \yb^{k+1} - \hat{\yb}^k} + \tfrac{L_h}{2}\norms{\yb^{k+1} - \hat{\yb}^k}^2 \vspace{1ex}\\
&\leq h(\hat{\yb}^k) + \iprods{\nabla{h}(\hat{\yb}^k), \yb - \hat{\yb}^k} + \iprods{\nabla{h}(\hat{\yb}^k), \yb^{k+1} - \yb} + \tfrac{L_h}{2}\norms{\yb^{k+1} - \hat{\yb}^k}^2.
\end{array}{\!\!\!\!}
\end{equation*}
In addition, the optimality condition of \eqref{eq:y_prob2} is
\begin{equation*} 
0 = \nabla{g}(\yb^{k+1}) + \nabla{h}(\hat{\yb}^k) + \rho_k\nabla_y{\psi}(\xb^{k+1},\hat{\yb}^k) + \hat{\beta}_k(\yb^{k+1} - \hat{\yb}^k), ~\nabla{g}(\yb^{k+1})\in \partial{g}(\yb^{k+1}).
\end{equation*}
Using these expressions and the same argument as the proof of Lemma~\ref{le:key_descent_lemma}, we derive
\begin{equation}\label{eq:lm31_est2_b}
\begin{array}{ll}
\Phi_{\rho_k}(\zb^{k+1}) &{\!\!}\leq f(\xb) + g(\yb)  + h(\hat{\yb}^k) + \iprods{\nabla{h}(\hat{\yb}^k, \yb - \hat{\yb}^k} + \rho_k\ell_k(\zb)  \vspace{1ex}\\
& + \gamma_k\iprods{\hat{\xb}^k - \xb^{k+1}, \xb^{k+1} - \xb} + \hat{\beta}_k\iprods{\hat{\yb}^k - \yb^{k+1}, \yb^{k+1} - \yb} \vspace{1ex}\\
& + \tfrac{\rho_k\norms{\Bb}^2 + L_h}{2}\norms{\yb^{k+1} -\hat{\yb}^k}^2 . 
\end{array}
\end{equation}
Finally, with the same proof as in \eqref{eq:key_est1}, and $\hat{\beta}_k = \norms{\Bb}^2\rho_k + L_h$, we can show that
\begin{align}\label{eq:key_est1_b}
{\!\!\!\!\!}\Phi_{\rho_k}(\zb^{k+1}) & \leq  (1-\tau_k)\Phi_{\rho_{k-1}}(\zb^k)  + \tau_kF(\zb^{\star}) + \tfrac{\gamma_k\tau_k^2}{2}\norms{\tilde{\xb}^k - \xb^{\star}}^2 \nonumber\\
& -  \tfrac{\gamma_k\tau_k^2}{2}\norms{\tilde{\xb}^{k+1} - \xb^{\star}}^2 + \tfrac{\hat{\beta}_k\tau_k^2}{2}\norms{\tilde{\yb}^k - \yb^{\star}}^2 - \tfrac{\hat{\beta}_k\tau_k^2}{2}\norms{\tilde{\yb}^{k+1} - \yb^{\star}}^2 \nonumber\\
& - \tfrac{(1-\tau_k)}{2}\left[\rho_{k-1} - \rho_k(1-\tau_k)\right]\norms{\sb^k}^2,
\end{align}
where $\sb^k := \Ab\xb^k + \Bb\yb^k - \cb - \proj_{\Kc}\big(\Ab\xb^k + \Bb\yb^k - \cb\big)$.
In order to telescope, we impose conditions on $\rho_k$ and $\tau_k$ as
\begin{equation*}
\frac{(\rho_k\norms{\Bb}^2 + L_h)\tau_k^2}{1 - \tau_k} \leq (\rho_{k-1}\norms{\Bb}^2 + L_h)\tau_{k-1}^2 ~~\text{and}~~\rho_k = \frac{\rho_{k-1}}{1-\tau_k}.
\end{equation*}
If we choose $\tau_{k} = \frac{1}{k+1}$, then $\rho_k = \rho_0(k+1)$.
The first condition above becomes
\begin{equation*}
\begin{array}{ll}
&\frac{\rho_0\norms{\Bb}^2(k+1) + L_h}{k(k+1)} \leq \frac{\rho_0\norms{\Bb}^2k + L_h}{k^2}\vspace{1ex}\\
\Leftrightarrow &\rho_0\norms{\Bb}^2 k(k+1) + L_hk \leq \rho_0\norms{\Bb}^2k(k+1) + L_h(k+1).
\end{array}
\end{equation*}
which certainly holds.

The remaining proof of the first part in Corollary~\ref{co:three_objs} is similar to the proof of Theorem~\ref{th:convergence1}, but with $R_p^2 :=  \gamma_0 \norms{\xb^0 - \xb^{\star}}^2 + (L_h + \rho_0\norms{\Bb}^2)\norms{\yb^0 - \yb^{\star}}^2$ due to \eqref{eq:key_est1_b}.

We now prove the second part of  Corollary~\ref{co:three_objs}.
For the case (i) with $\mu_g > 0$, the proof is very similar to the proof of Theorem~\ref{th:convergence2}, but $\rho_k\norms{\Bb}^2$ is changed to $\hat{\beta}_k$ and is updated as $\hat{\beta}_k = \rho_k\norm{\Bb}^2 + L_h$.
We omit the detail of this analysis here.
We only prove the second case (ii) when $L_h < 2\mu_h$.

Using the convexity and the Lipschitz gradient continuity of $h$, we can derive
\begin{equation*}
\begin{array}{ll}
h(\breve{\yb}^{k+1}) &\leq (1-\tau_k)h(\yb^k) + \tau_kh(\tilde{\yb}^{k+1}) - \frac{\mu_h\tau_k(1-\tau_k)}{2}\norms{\tilde{\yb}^{k+1} - \yb^k}^2 \vspace{1ex}\\
&\leq (1-\tau_k)h(\yb^k) + \tau_kh(\tilde{\yb}^{k}) + \tau_k\iprods{\nabla{h}(\tilde{\yb}^k), \tilde{\yb}^{k+1} - \tilde{\yb}^k} + \frac{\tau_kL_h}{2}\norms{\tilde{\yb}^{k+1} - \tilde{\yb}^k}^2 \vspace{1ex}\\
&\leq (1-\tau_k)h(\yb^k) + \tau_kh(\yb^{\star}) + \tau_k\iprods{\nabla{h}(\tilde{\yb}^k), \tilde{\yb}^{k+1} -\yb^{\star}} \vspace{1ex}\\
& + \frac{\tau_kL_h}{2}\norms{\tilde{\yb}^{k+1} - \tilde{\yb}^k}^2 - \frac{\tau_k\mu_h}{2}\norms{\tilde{\yb}^k - \yb^{\star}}^2.
\end{array}
\end{equation*}
Using this estimate, with a similar proof as of \eqref{eq:proof2_est11}, we can derive
\begin{equation*}
\begin{array}{ll}
\breve{\Phi}_{k+1} & := f(\xb^{k+1}) + g(\breve{\yb}^{k+1}) + h(\breve{\yb}^{k+1}) + \rho_k\Qc_k(\breve{\yb}^{k+1}) \vspace{1ex}\\
&\overset{\tiny\eqref{eq:proof2_est2},\eqref{eq:proof2_13a},\eqref{eq:proof2_13b}}{\leq} (1-\tau_k)\big[ F(\zb^{k}) + \rho_k\ell_k(\zb^k) \big] + \tau_k\big[ F(\zb^{\star}) + \rho_k\ell_k(\tilde{z}^{k+1})\big] \vspace{1ex}\\
& +~ \tau_k\iprods{\nabla{f}(\xb^{k+1}), \tilde{\xb}^{k+1} - \xb^{\star}}  + \tau_k\iprods{\nabla{g}(\tilde{\yb}^{k+1}) + \nabla{h}(\tilde{\yb}^k), \tilde{\yb}^{k+1} {\!\!}- \yb^{\star}}   \vspace{1ex}\\
& +~\frac{\left(\rho_k\tau_k^2\norms{\Bb}^2 + \tau_kL_h \right)}{2}\norms{\tilde{\yb}^{k+1} - \tilde{\yb}^k}^2   -  \frac{\tau_k\mu_g}{2}\norms{\tilde{\yb}^{k+1} -  \yb^{\star}}^2 -   \frac{\tau_k\mu_h}{2}\norms{\tilde{\yb}^k - \yb^{\star}}^2\vspace{1ex}\\
& \leq (1-\tau_k)\big[ F(\zb^{k}) + \rho_k\ell_k(\zb^k) \big] + \tau_k F(\zb^{\star}) -\tfrac{\rho_k\tau_k}{2}\norms{\hat{\sb}^{k+1}}^2 \vspace{1ex}\\
& +~ \gamma_0\tau_k\iprods{\xb^{k+1} - \hat{\xb}^k,\xb^{\star} -  \tilde{\xb}^{k+1}}  + \tau_k^2\hat{\beta}_k\iprods{\tilde{\yb}^{k+1} - \tilde{\yb}^k, \yb^{\star} - \tilde{\yb}^{k+1}}  \vspace{1ex}\\
& +~ \frac{\left(\rho_k\tau_k^2\norms{\Bb}^2 +  \tau_kL_h\right)}{2}\norms{\tilde{\yb}^{k+1} - \tilde{\yb}^k}^2  - \frac{\tau_k\mu_g}{2}\norms{\tilde{\yb}^{k+ 1} -  \yb^{\star}}^2 - \frac{\tau_k\mu_h}{2}\norms{\tilde{\yb}^k - \yb^{\star}}^2.
\end{array}
\end{equation*}
Here, we use the optimality condition of \eqref{eq:x_cvx_subprob} and \eqref{eq:y_prob2b} into the last inequality, and $\nabla{f}$, $\nabla{g}$, and $\nabla{h}$ are subgradients of $f$, $g$, and $h$, respectively.

Using the same argument as the proof of \eqref{eq:key_est3}, if we denote $S_k  := \Phi_{\rho_{k-1}}(\zb^k) - F^{\star}$, then the last inequality above together with \eqref{eq:choice2} leads to
\begin{align}\label{eq:co2_key_est3} 
S_{k+1} & +  \tfrac{\gamma_0\tau_k^2}{2}\norms{\tilde{\xb}^{k+1} - \xb^{\star}}^2  + \tfrac{\hat{\beta}_k\tau_k^2 + \mu_g\tau_k}{2}\norms{\tilde{\yb}^{k+1} - \yb^{\star}}^2  \leq (1 - \tau_k)S_k   + \tfrac{\gamma_0\tau_k^2}{2}\norms{\tilde{\xb}^k {\!} - \xb^{\star}}^2 \nonumber\\
& + \tfrac{\hat{\beta}_k\tau_k^2 \!-\! \tau_k\mu_h}{2}\norms{\tilde{\yb}^k \!-\! \yb^{\star}}^2  -  \tfrac{(\hat{\beta}_k\tau_k^2 - \rho_k\tau_k^2\norms{\Bb}^2 - \tau_kL_h)}{2}\norms{\tilde{\yb}^{k+1}  - \tilde{\yb}^{k}}^2 \nonumber\\
&   - \frac{(1-\tau_k)}{2}\left[\rho_{k-1} - \rho_k(1-\tau_k)\right]\norms{\sb^k}^2,
\end{align}
where $\sb^k := \Ab\xb^k + \Bb\yb^k - \cb - \proj_{\Kc}\big(\Ab\xb^k + \Bb\yb^k - \cb\big)$. 
We still choose the update rule for $\tau_k$, $\rho_k$ and $\gamma_k$ as in Algorithm~\ref{alg:A2}.
Then, in order to telescope this inequality, we impose the following conditions:
\begin{equation*}
\hat{\beta}_k = \rho_k\norms{\Bb}^2 + \tfrac{L_h}{\tau_k}, ~~~\text{and}~~~\hat{\beta}_k\tau_k^2 - \mu_h\tau_k \leq (1-\tau_k)(\hat{\beta}_{k-1}\tau_{k-1}^2 + \mu_g\tau_{k-1}).
\end{equation*}
Using the first condition into the second one and noting that $1 - \tau_k = \frac{\tau_k^2}{\tau_{k-1}^2}$ and $\rho_k = \frac{\rho_0}{\tau_k^2}$, we obtain
$ \rho_0\norm{\Bb}^2 + L_h - \mu_h \leq \frac{\tau_k}{\tau_{k-1}}(L_h + \mu_g)$.
This condition holds if $\rho_0 \leq \frac{\mu_g + 2\mu_h - L_h}{2\norms{\Bb}^2} > 0$.
Using \eqref{eq:co2_key_est3}  we have the same conclusion as in Theorem~\ref{th:convergence2}.
\Eproof

\beforesec
\bibliographystyle{plain}

\end{document}